\documentclass[namelimits,sumlimits,a4paper]{amsart}
\usepackage{amsfonts}
\usepackage[hidelinks]{hyperref}
\usepackage{appendix}
\usepackage{amsmath}
\usepackage{amsthm}
\usepackage{comment}
\usepackage{amssymb}
\usepackage{amsmath,amscd,mathtools}
\usepackage{graphicx} 
\usepackage{xcolor}
\usepackage{enumitem}
\usepackage{nicematrix}
\usepackage{pgffor}
\usepackage{marginnote} 
\usepackage{xparse}
\usepackage{tikz}
\usepackage{slashed}

\newcommand{\R}{\mathbb{R}}
\newcommand{\Z}{\mathbb{Z}}
\newcommand{\C}{\mathbb{C}}

\newcommand{\Spin}{\operatorname{Spin}}
\newcommand{\U}{\operatorname{U}}
\newcommand{\SU}{\operatorname{SU}}
\newcommand{\SO}{\operatorname{SO}}
\newcommand{\so}{\mathfrak{so}}
\newcommand{\A}{\mathcal{A}}

\newcommand{\Hom}{\operatorname{Hom}}
\newcommand{\Map}{\operatorname{Hom}}
\newcommand{\diag}{\operatorname{diag}}

\newcommand{\vol}{\operatorname{vol}}
\newcommand{\arctanh}{\operatorname{arctanh}}
\newcommand{\csch}{\operatorname{csch}}

\newcommand{\st}{\big|\;}
\newcommand{\abs}[1]{|#1|}

\theoremstyle{definition}
\newtheorem{definition}{Definition}[section]
\theoremstyle{plain}
\newtheorem{lemma}[definition]{Lemma}
\newtheorem{theorem}[definition]{Theorem}
\newtheorem{introtheorem}{Theorem}
\newtheorem{introconj}[introtheorem]{Conjecture}

\newtheorem{conjecture}[definition]{Conjecture}
\newtheorem{proposition}[definition]{Proposition}
\newtheorem{corollary}[definition]{Corollary}

\newtheorem{question}[definition]{Question}

\theoremstyle{remark}
\newtheorem{remark}[definition]{Remark}

\newif\ifprintlabels

\printlabelsfalse

\ifprintlabels
    \usepackage[margin=4cm, includefoot, footskip=30pt,]{geometry} 
\else 
\fi

\AtBeginDocument{%
  \let\originallabel\label
  \RenewDocumentCommand{\label}{m}{
    \originallabel{#1}
    \ifprintlabels
      \marginnote{\footnotesize \texttt{#1}}
    \fi
  }
}

\title{Gauge theory on Cohomogeneity One Spin(7)-Manifolds from Aloff--Wallach Spaces}
\author{Hanci Chi and Enric Solé-Farré}
\thanks{ }
\begin{document}
\begin{abstract}
We begin a systematic study of $\mathrm{Spin}(7)$-instantons on complete non-compact cohomogeneity one $\mathrm{Spin}(7)$-manifolds with Aloff--Wallach spaces $N_{k,l}=\mathrm{SU}(3)/\mathrm{U}(1)_{k,l}$ as principal orbits. We provide a complete classification of admissible principal $\mathrm{U}(1)$- and $\mathrm{SO}(3)$-bundles and smooth bounded solutions to the $\mathrm{Spin}(7)$-instanton equation on them. We study bubbling phenomena for the constructed $\mathrm{Spin}(7)$-instantons and discuss the compactification of the moduli space of $\mathrm{Spin}(7)$-instantons in both the ALC and AC settings.
\end{abstract}
\maketitle
\vspace{-1.3cm}
\section{Introduction}
The study of Riemannian manifolds with exceptional holonomy groups $G_2$ and $\Spin(7)$ has been an active research direction since Berger's classification of holonomy groups, particularly fruitful over the last three decades due to the constructions of Bryant and Salamon \cite{BS89} in 1989 in the non-compact case, and Joyce \cite{Joyce96} in 1996 in the compact case. 

In 1998, Donaldson and Thomas~\cite{DT98}, and subsequently Donaldson and Segal~\cite{DS11}, proposed using higher-dimensional instanton equations to develop gauge-theoretic enumerative invariants on manifolds with special holonomy, generalising the Donaldson--Thomas invariants for  Calabi--Yau threefolds and fourfolds. A central
challenge in this programme is to understand the compactness
properties of instanton moduli spaces and the degenerations that
occur at their boundaries. Although the proposed invariants concern
primarily compact manifolds, instantons on complete non-compact
spaces provide models for bubbling and singularity formation, as well as building blocks for gluing constructions that resolve singularities of the underlying metric or connection.

On a $\Spin(7)$-manifold $(M, g, \Phi)$, the parallel defining $4$-form determines a natural instanton equation whose solutions are automatically Yang--Mills connections:
\begin{equation}
    *F_A=-\Phi\wedge F_A.
    \label{eq:intro-spin7-instanton}
\end{equation}

In this paper, we study $\Spin(7)$-instantons on $\Spin(7)$-manifolds
where both the connections and the underlying metric are invariant under the action of a group $G$ whose generic orbits are of codimension one. In this case, the equations for an invariant $\Spin(7)$-metric reduce to ordinary differential equations along a radial geodesic, with boundary conditions determined by smooth extension across the singular orbit and by the geometry at infinity. This approach was originally used by Bryant and Salamon in \cite{BS89} to produce the first examples of complete $\Spin(7)$-metrics. More recently, the works of Reidegeld \cite{Reidegeld_11}, Lehmann \cite{Lehmann22} and Chi \cite{Chi22,Chi25} constructed families of complete $\Spin(7)$-metrics with Aloff--Wallach principal orbits. These include asymptotically locally conical (ALC) metrics and asymptotically conical (AC) metrics. Most of these metrics are known through qualitative analysis of their defining equations rather than explicit formulae.

For connections invariant under the same group action, the instanton equation also reduces to ordinary differential equations. However, local solvability does not settle the global gauge-theoretic problem. One must determine which equivariant bundles and connections extend across the singular orbit, which initial conditions yield solutions with controlled behaviour at infinity, and how these solutions vary with the background metric. Even when the metric is explicit, the resulting instanton equations need not be explicitly integrable. For non-explicit metric families, these questions require simultaneous control of the metric coefficients and the connection.

The study of invariant instantons on complete non-compact manifolds with special holonomy was initiated by Clarke \cite{Clarke14}, who constructed explicit examples of $G_2$- and $\Spin(7)$-instantons on the Bryant--Salamon spaces. Subsequent work has moved beyond constructing individual examples to describing moduli spaces, determining asymptotic behaviour, and analysing the bubbling phenomena that occur at the boundaries of families. These developments provide concrete models for the compactness questions discussed above.

In six dimensions, Stein~\cite{Stein23} studied and classified $\SU(2)^2$-invariant Calabi--Yau instantons and monopoles on the smoothing and small resolution of the conifold, and on the canonical bundle of $\mathbb{CP}^1\times\mathbb{CP}^1$, equipped with their AC Calabi--Yau metrics. He constructed one-parameter families of instantons and classified the corresponding invariant moduli spaces under a natural curvature decay condition.  

In seven dimensions, Lotay and Oliveira~\cite{LO18} developed the existence and classification theory of $\SU(2)^2$-invariant $G_2$-instantons on $\mathbb{R}^4\times S^3$, considering both the
AC Bryant--Salamon metric and the explicit ALC metric of Brandhuber--Gomis--Gubser--Gukov (BGGG). Their work exhibited bubbling and removable-singularity phenomena in families of solutions, consistent with the compactness and removable-singularity results of Tian and Tao–Tian \cite{Tian00} \cite{TT04}. Stein and Turner~\cite{ST24} subsequently completed the classification of $\SU(2)^3$-invariant $G_2$-instantons on the AC Bryant--Salamon spinor bundle of the three-sphere by constructing an additional one-parameter family. Alonso~\cite{Alonso24} extended this direction to coclosed, not necessarily torsion-free, $G_2$-structures on $\mathbb{R}^4\times S^3$. 

The study of invariant instantons also extends to non-explicit background metrics. Matthies, Nordstr\"om and Turner~\cite{MNT24} constructed a one-parameter family of
$\SU(2)^2\times\U(1)$-invariant $G_2$-instantons with gauge group $\SU(2)$ and bounded curvature on the AC limit of the $\mathbb{C}_7$ family. Their dynamical systems argument combines
perturbations of an abelian solution with the analysis of a solution on the asymptotic cone. For the ALC members of the $\mathbb{B}_7$ family, Stein and Turner~\cite{ST25} constructed and classified a two-parameter family of non-abelian invariant instantons, extending the results of Lotay and Oliveira on the BGGG metric.

In eight dimensions, Papoulias~\cite{Papoulias_22} studied invariant $\Spin(7)$-instantons and Hermitian Yang--Mills connections on $T^*S^4$ with its AC Stenzel Calabi--Yau metric. His analysis describes the invariant non-abelian moduli space and exhibits bubbling along the Cayley zero section, with a Hermitian Yang--Mills connection appearing as a limiting configuration. More recently, Alonso, Madnick and Windes~\cite{AMW25} studied invariant gauge theory on the hyperkähler Calabi metric on $T^*\mathbb{CP}^2$. Our results overlap with part of their analysis
on the exceptional hyperkähler metrics (cf. Appendix~\ref{sec:explicit-examples}). A complementary construction is due to Clarke and Oliveira~\cite{CO21}, who obtained explicit one-parameter families of $\Spin(7)$-instantons on AC manifolds and orbifolds obtained by filling squashed $3$-Sasakian manifolds. Their evolution ansatz also reduces the instanton equation to an ODE system but it does not arise from a cohomogeneity one setting.

The present paper develops these questions for the families of complete $\Spin(7)$-metrics with Aloff--Wallach principal orbits constructed in \cite{Chi25}. We classify admissible equivariant $\U(1)$- and $\SO(3)$-bundles, construct instantons on the resulting extensions and determine their behaviour at infinity. When a principal-orbit bundle admits both reducible and irreducible extensions over the singular orbit, we relate the corresponding instanton families by bubbling along that orbit. In the ALC case, we compactify these families by adjoining bubbling limits and additional ideal boundary points parametrising $\U(2)$-invariant anti-self-dual connections on flat $\R^4$, obtained by rescaling in the normal directions to the singular orbit.

The main feature of our approach is its applicability to background metrics that are generally non-explicit. We study the metric and instanton equations as a common dynamical system, in which smooth extension and asymptotic conditions correspond to trajectories associated with distinguished equilibria. Invariant regions and comparison arguments control global solutions, while trajectory breaking and rescaling capture their degenerations. This allows us to compare the different moduli spaces of $\Spin(7)$-instantons and track their behaviour in the limiting case as the ALC metric approaches its AC limit, answering a question of Lotay--Oliveira~\cite{LO18}.
\subsection{Main results}
We systematically study $\Spin(7)$-instantons on complete non-compact cohomogeneity one $\Spin(7)$-manifolds with Aloff–Wallach spaces as principal orbits. Aloff--Wallach spaces are the homogeneous spaces given by a $\U(1)$-quotient of $\SU(3)$, where the circle is embedded in $\mathrm{SU}(3)$ as $ z  \mapsto \operatorname{diag}\bigl( z^k ,\; z^l,\; z^{-l-k} \bigr)$. We denote these spaces bys $N_{k,l} \coloneqq  \mathrm{SU}(3)/\mathrm{U}(1)_{k,l}$.

Up to reparametrisation of $\U(1)$ and the action of the Weyl group of $\SU(3)$, we may assume that $k$ and $l$ are coprime, with $k\geq l\geq 0$. For notational convenience, we introduce the quantities
\begin{align*}
     m \coloneqq -l-k \,,\qquad &\widehat{k}\coloneqq k+2l\,, \qquad \widehat{l} \coloneqq -2k-l\,, \qquad \widehat{m}\coloneqq k-l\; \qquad \text{and} \qquad \Delta \coloneqq k^2+kl+l^2.
\end{align*}
Note that by our choice of normalisation of the Aloff--Wallach spaces, we have $ \widehat{k}, \widehat{m}\geq 0$ and $\widehat{l}<0$. We say an Aloff--Wallach space is generic whenever $kl\widehat{m} \neq 0$ and exceptional otherwise. We have the following straightforward but extremely useful identities:
\begin{equation}
     (\widehat{i}j-i \widehat{j})(1 - \delta_{ij})  =  2 \varepsilon_ {ijh} \Delta   \qquad \qquad   \widehat{i}\,\widehat{j}+3ij  =  2 (3\delta_{ij} -1 ) \Delta\;,
\end{equation}
where $i,j,h \in \{k,l,m\}$ and $\delta_{ij}$ and $\varepsilon_{ijh}$ are the usual Kronecker and Levi-Civita symbols, with $\varepsilon_{klm}=1$.

The collection of Aloff--Wallach spaces can also be parametrised using $r\coloneqq \frac{l}{k} \in [0,1] \cap \mathbb{Q}$. Under this representation, the two exceptional Aloff--Wallach spaces, $N_{1,0}$ and $N_{1,1}$ correspond to the boundary points $r=0$ and $r=1$, respectively. Note that this rational parametrisation of Aloff--Wallach spaces carries the subspace topology inherited from $\R$.

Each generic Aloff--Wallach space $N_{k,l}$ admits three distinct cohomogeneity one orbifold resolutions, denoted by $M^{\abs{i}}_{k,l}$ for $i\in\{k,l,m\}$, each with singular orbit diffeomorphic to $\C P^2$.

The exceptional space $N_{1,1}$ admits two distinct resolutions. The smooth resolution $M^1_{1,1}$ is diffeomorphic to $T^*\C P^2$ and carries the hyperkähler metric constructed by
Calabi~\cite{Calabi79}. The orbifold resolution $M^2_{1,1}$, whose normal fibres are modelled on $\R^4/(\Z/2\Z)$, was considered by Bryant and Salamon~\cite{BS89}.

The other exceptional space, $N_{1,0}$, admits two smooth resolutions, $M^1_{1,0}$ and $M^0_{1,0}$, with singular orbits diffeomorphic to $\C P^2$ and $S^5$, respectively.

The works of Reidegeld \cite{Reidegeld_11}, Lehmann \cite{Lehmann22} and Chi \cite{Chi22,Chi25} characterise the existence of complete $\Spin(7)$-metrics $\gamma^i_\theta$ on $M^i_{k,l}$ for $0< \theta\leq \theta^*$.  For $0<\theta<\theta_*$, the metric is asymptotically locally conical (ALC): its end is asymptotic to a circle bundle of asymptotically constant fibre length over the $G_2$-cone on the homogeneous nearly Kähler flag manifold $\operatorname{SU}(3)/T^2$. At the limit parameter $\theta=\theta_*$, the metric is asymptotically conical (AC), with asymptotic cone determined by a homogeneous nearly parallel $G_2$-structure on $N_{k,l}$.

We classify the admissible equivariant principal  $\operatorname{SO}(3)$-bundles and their invariant connections. In the generic case, the filling $M^i_{k,l}$ has a distinguished irreducible extension, which we denote by $Q_{\widehat{i}}$. The same principal-orbit bundle admits a reducible extension $P_{\widehat{i}}$ exactly when $\widehat{i} = 0\mod{i}$. We have
\begin{introtheorem}
\label{thm:intro-existence}
Let $k>l>0$ be coprime, so $N_{k,l}$ is a generic Aloff--Wallach space. Take $i\in\{k,l,\abs{m}\}$, and equip $M^i_{k,l}$ with one of the complete metrics $\gamma^i_\theta$. Consider invariant instantons that extend smoothly across the singular orbit. 

Fix $\theta \in(0, \theta^*)$, so we are in the ALC regime. We have that
\begin{enumerate}[label=(\roman*)]
  \item the bundle $Q_{\widehat{i}}$ carries a one-parameter family of irreducible $\operatorname{Spin}(7)$-instantons, parametrised by a closed half-line, and 
  \item if $P_{\widehat{i}}$ exists, it carries a closed two-parameter family with boundary of irreducible instantons.
\end{enumerate}
 In both cases, interior points of the family converge exponentially to the abelian solution on the same metric with the same asymptotic limit, while boundary points converge polynomially.
 
 For the AC metric $\gamma^i_{\theta^*}$, we have that 
 \begin{enumerate}[label=(\roman*)] \setcounter{enumi}{2}
  \item there is a unique bounded invariant $\Spin(7)$-instanton on $Q_{\widehat{i}}$, and 
  \item if $P_{\widehat{i}}$ exists, it carries a one-parameter family of $\Spin(7)$-instantons.
\end{enumerate}
\end{introtheorem}
Theorem \ref{thm:intro-existence} summarises the existence and classification results of Section \ref{sec:long-time-SO3}. The exceptional Aloff--Wallach spaces admit additional components
in the invariant connection ansatz and are analysed separately. For bounded solutions, the resulting families have the same qualitative behaviour as in the generic case. On
$P_0 \to M^2_{1,1}$ and $Q_0 \to M^2_{1,1}$, however, the reduced equations admit an additional $S^1$-parameter.

The asymptotic behaviour in the ALC case has a natural interpretation in terms of monopoles.The $\Spin(7)$-instantons converge to pullbacks of abelian $G_2$-monopole on the metric cone over the flag $\SU(3)/T^2$, with a parallel Higgs field. The limiting Higgs field governs the asymptotic decay of the non-abelian component: the decay is exponential in the interior of the families and polynomial at the critical value, where the mass term vanishes. This critical value also corresponds to trivial holonomy around the asymptotic circle. Here, exponential convergence means convergence to the abelian solution on the same metric with the same limit; the abelian background itself generally retains a polynomial tail. This distinction is consistent with the analysis of positive-mass monopoles on AC $G_2$-manifolds by Fadel--Nagy--Oliveira \cite[Thm.~6.1, Prop.~6.3]{FNO}, where the components transverse to the Higgs field decay exponentially, while the longitudinal component admits polynomial decay.

Let us now suppose that both $P_{\widehat{i}}$ and $Q_{\widehat{i}}$ exist. They restrict to the same equivariant bundle away from the singular orbit, so connections on these different extensions can be compared on the principal part of $M^i_{k,l}$. We show that this identification has a direct analytic significance: instantons on the reducible extension can converge away from $\mathbb{CP}^{2}$ to instantons on the irreducible extension.

The convergence is accompanied by curvature concentration along the singular orbit. Rescaling the normal directions reveals a four-dimensional anti-self-dual instanton on $\mathbb{R}^{4}/\mathbb{Z}_{|i|}$. Thus the change of bundle extension and the appearance of a normal instanton bubble are two descriptions of the same degeneration.
\begin{introtheorem}
\label{thm:intro-bubbling}
Assume the hypotheses of Theorem~\ref{thm:intro-existence} and that $P_{\widehat{i}}$ exists.
\begin{enumerate}[label=(\roman*)]
  \item In the case of an AC metric, the non-compact end of the instanton family on $P_{\widehat{i}}$ converges smoothly on compact subsets away from the singular orbit to the unique instanton on $Q_{\widehat{i}}$. Suitable normal rescalings converge to the standard radial anti-self-dual instanton on  $\mathbb{R}^{4}/\mathbb{Z}_{|i|}$.
  \item For an ALC $\Spin(7)$-metric $\gamma^i_\theta$, every instanton in the family on $Q_{\widehat{i}}$ arises as a limit of instantons on $P_{\widehat{i}}$ in this way, with the same normal anti-self-dual bubbling model.
\end{enumerate}
\end{introtheorem}
Combining Theorems \ref{thm:intro-existence} and \ref{thm:intro-bubbling}, we obtain a natural one-point compactification of the moduli space of bounded invariant instantons modulo gauge $\mathcal M(P_{\widehat{i}})$ in the AC case by adjoining the unique instanton in $\mathcal{M}(Q_{\widehat{i}})$, with $C^\infty_{\mathrm{loc}}$ convergence modulo gauge away from the singular orbit. The ALC case has an additional source of non-compactness. Some sequences have no bounded limiting connection away from the singular orbit, but still possess well-defined normal rescaling limits. These limits correspond to ideal $\U(2)$-invariant ASD connections with gauge group $\SU(2)$ on $\R^4$ (cf. Section \ref{sec:invariant-ASD}), governed by the ODE:
\begin{equation} \label{eq:U2-ASD}
    \tau \frac{d\Xi}{d\tau}=2\Xi-\Psi^2 \qquad \qquad \qquad \tau\frac{d\Psi}{d\tau}=(2-\Xi)\Psi\;,
\end{equation} 
with prescribed asymptotics at $\tau =0$. These rescaled profiles record the local degeneration near the singular orbit. We regard them as ideal boundary data, with convergence understood through the normal rescaling maps on compact subsets where the limiting profile is defined. The standard charge one $\SU(2)$-ASD connection is a distinguished solution of this system. Up to gauge and dilation, it is the unique non-flat solution in this invariant class that has finite Yang–Mills. 
\begin{introtheorem}
\label{thm:intro-compactification-ALC}
Let $\gamma^i_{\theta}$ be an ALC $\Spin(7)$-metric on $M^i_{k,l}$, and assume $\widehat{i}= 0 \mod i$, so $P_{\widehat{i}}$ exists. We have the compactifications
\begin{align*}
    \overline{\mathcal M(Q_{\widehat{i}})} &=\mathcal M(Q_{\widehat{i}}) \sqcup \{*\} \\
    \overline{\mathcal M(P_{\widehat{i}})}& =\mathcal M(P_{\widehat{i}})\sqcup\mathcal M(Q_{\widehat{i}})\sqcup  \mathbb{I}_{\,\widehat{i}}\;,
\end{align*}
where $\mathbb{I}_{\,\widehat{i}}\cong \left[0,\frac{\pi}{4}\right] \sqcup \{\pi\}$ parametrises ideal boundary points represented by $\U(2)$-invariant ASD connections, solving Equation \eqref{eq:U2-ASD} with asymptotics
$$ \Xi(\tau) =  2 \cos (\beta )\tau^2 + o(\tau^2) \qquad \qquad \qquad \Psi(\tau) =  2 \sin(\beta) \tau ^2  + o(\tau^2) \;,$$
with $\beta \in \left[0,\frac{\pi}{4}\right] \sqcup \{\pi\}$. The ideal boundary point for $\overline{\mathcal M(Q_{\widehat{i}})}$ corresponds to an $\U(2)$-invariant ASD connection, solving Equation \eqref{eq:U2-ASD} with asymptotics
$$ \Xi =  2 + \frac{\tau ^4}{6} + o(\tau^4) \qquad \qquad \qquad \Psi(\tau) =  2 -  \frac{\tau ^4}{12}  + o(\tau^4) \;.$$
\end{introtheorem}

The AC $\Spin(7)$-instantons in Theorem~\ref{thm:intro-existence} converge to homogeneous $G_2$-instantons on the nearly parallel Aloff--Wallach link. These homogeneous connections were first studied by Ball--Oliveira~\cite{BO19}. Our global existence results identify which of them occur as limits of instantons extending across the singular orbit.

The answer depends on the chosen filling. For $i=m$, the instanton on $Q_{\widehat m}$ has an irreducible limit, while for $i=l$ its limit is abelian. For $i=k$, the answer changes at an explicit value of $r=l/k$:
$$ r_*=-\frac{1}{2}+ \frac{\sqrt{861+222\sqrt{19}}}{34} \approx 0.75773577\;.$$
\begin{introtheorem}
\label{thm:intro-threshold}
On the AC metric on $M^k_{k,l}$, the unique bounded invariant instanton on $Q_{\widehat k}$ converges to an irreducible homogeneous $G_2$-instanton if $l/k<r_*$, and to the abelian homogeneous $G_2$-instanton if $l/k>r_*$. The irreducible members of the family on $P_{\widehat k}$, when this extension exists, have the same respective asymptotic limits.
\end{introtheorem}
The threshold value $r_*$ has a natural interpretation from a dynamical system perspective. Allowing $r$ to vary continuously in the reduced equations, $r^*$ is the bifurcation parameter at which the irreducible equilibrium meets the abelian equilibrium. The moduli spaces on actual Aloff--Wallach spaces sample this continuous family at rational values of $r$. We have no geometric explanation for the existence or value of $r_*$.

The transition from ALC to AC geometry also relates the families themselves.
\begin{introtheorem}
In the limit $\theta \rightarrow \theta_*$, the boundary point of $\mathcal{M}(Q_{\widehat{i}})$ converges, in the initial-data parametrisation, to the unique AC $\Spin(7)$-instanton.

If $\widehat{i}= 0 \mod i$, the boundary curves of the $\Spin(7)$-instanton moduli space $\mathcal{M}(P_{\widehat{i}})$ converge, in the initial-data parametrisation, to the curve describing the $\Spin(7)$-instantons in the AC case.
\end{introtheorem}

The exceptional space $M^0_{1,0}$ exhibits a different degeneration mechanism. Its singular orbit is $S^5$, with three-dimensional normal fibres, so the natural normal model is a monopole on $\mathbb R^3$. The components of the instanton connection tangent to the normal fibres determine a three-dimensional connection, while a tangential component along the singular orbit supplies the Higgs field. We construct local families for which these fields converge, after rescaling, to the BPS monopole.
\begin{introtheorem}
Let $\gamma^0_\mu$ be an ALC $\Spin(7)$-metric on $M^0_{1,0}$. There exists a one-parameter family $A_\lambda$ of local invariant $\Spin(7)$-instantons on $Q_2$ extending smoothly across the singular orbit $S^5$. As $\lambda\to-\infty$, suitable normal rescalings of the normal component of $A_\lambda$, together with the appropriately rescaled Higgs field determined by its tangential component, converge to the BPS monopole on $\mathbb R^3$.
\end{introtheorem}
This gives a codimension-three counterpart to the normal ASD bubbling described in Theorem \ref{thm:intro-compactification-ALC}. The argument is local and does not require the instantons to extend over the complete non-compact manifold, whose global existence remains open. We conjecture 
\begin{introconj}
    Let $\gamma^0_\mu(t)$ be a complete $\Spin(7)$-metric on $M^0_{1,0}$, with $\mu \in (0, \mu^*]$, and consider the principal $\SO(3)$-bundle $Q_2 \to M^0_{1,0}$. For the ALC metrics, $\mu <\mu^*$, there exists a one-parameter family of $\Spin(7)$-instantons, indexed by $\lambda \in(-\infty,  \lambda^*_\mu]$. For the AC metric $\gamma^0_{\mu^*}$, there exists a unique $\Spin(7)$-instanton, corresponding to $\lambda^*$ and satisfying $$\lim_{\mu \to \mu^*} \lambda^*_\mu = \lambda^* \;.$$
\end{introconj}
Finally, we study $\Spin(7)$-instantons on the metric cones where the link is a homogeneous nearly parallel $G_2$-manifold. The $\Spin(7)$-instanton equation on the cone is an autonomous evolution equation whose stationary points are $G_2$-instantons on the link. In the Chern--Simons formulation discussed in~\cite{DS11}, its integral curves are gradient flow lines between critical points. We prove
\begin{introtheorem}
For the positive nearly parallel $G_2$-structure on $N_{k,l}$, there exists a Chern--Simons gradient flow line joining the abelian and irreducible homogeneous instantons on $L_{\widehat m}$. For the negative nearly parallel structure, an analogous flow line exists on $L_{\widehat k}$ whenever $l/k<r_*$. 
\end{introtheorem}
Here $L_{\widehat i}\to N_{k,l}$ is the homogeneous principal $\SO(3)$-bundle determined by the homomorphism
\begin{align*}
    \theta_{\widehat i}:\U(1)&\longrightarrow\SO(3)\\
z& \mapsto\iota(z^{\widehat i}),
\end{align*}
where $\iota:\U(1)\hookrightarrow\SO(3)$ is the standard embedding as rotations around a fixed axis.
\subsection{Acknowledgements}
The authors thank the organisers of the International Congress of Basic Science, during which a significant part of this work was carried out, and Lorenzo Foscolo for his helpful comments and for bringing the work of FNO to the authors' attention. HC also thanks Teng Huang for suggesting this direction of research at an early stage. ESF thanks Izar Alonso for feedback on an earlier draft.

HC acknowledges support from the NSFC under grant No. 12301078. ESF was supported by Beijing Natural Science Foundation grant No. IS26009.
\subsection{Use of Artificial Intelligence}
The authors used large language models (LLMs) as an aid in literature review, drafting, and improving the exposition. All mathematical statements, arguments, computations, and bibliographic references contained in the manuscript were independently checked and verified by the authors, who take full responsibility for the contents of the paper.
\section{Basic Setup}
For a given Aloff--Wallach space $N_{k,l}$, we choose the following basis for $\mathfrak{su}(3)$.
\begin{equation}\label{eq:su3-basis}
\begin{alignedat}{4}
E_1 &= \begin{pmatrix}
    0&1&0\\
    -1&0&0\\
    0&0&0
\end{pmatrix}, &\quad
E_2 &= \sqrt{-1}\begin{pmatrix}
    0&1&0\\
    1&0&0\\
    0&0&0
\end{pmatrix}, &\quad
E_3 &= \begin{pmatrix}
    0&0&-1\\
    0&0&0\\
    1&0&0
\end{pmatrix}, &\quad
E_4 &= \sqrt{-1}\begin{pmatrix}
    0&0&1\\
    0&0&0\\
    1&0&0
\end{pmatrix},\\[1ex]
E_5 &= \begin{pmatrix}
    0&0&0\\
    0&0&1\\
    0&-1&0
\end{pmatrix}, &\quad
E_6 &= \sqrt{-1}\begin{pmatrix}
    0&0&0\\
    0&0&1\\
    0&1&0
\end{pmatrix}, &\quad
E_7 &= \sqrt{-1}\begin{pmatrix}
    \widehat{k}&0&0\\
    0&\widehat{l}&0\\
    0&0&\widehat{m}
\end{pmatrix}, &\quad
E_8 &= \sqrt{-1}\begin{pmatrix}
    k&0&0\\
    0&l&0\\
    0&0&m
\end{pmatrix}.
\end{alignedat}
\end{equation}
The vector $E_8$ is the infinitesimal generator of the subgroup $U(1)_{k,l}$ inside $SU(3)$, and the vectors $E_7$ and $E_8$ span the Cartan subalgebra of $\mathfrak{su}(3)$. For a generic Aloff--Wallach space, the isotropy representation splits into 4 inequivalent summands.
$$\mathfrak{su}(3)/\mathfrak{u}_{k,l}(1)=\langle E_1 ,E_2 \rangle_{\widehat{m}}+\langle E_3,E_4\rangle_{\widehat{l}}+\langle E_5,E_6\rangle_{\widehat{k}}+\langle E_7\rangle. $$
Let $\{e_i\}$ be the co-frame dual to $\{E_i\}$. The soldering $SU(3)$ 1-form $\Sigma$ takes the form
$$
\Sigma=\begin{pmatrix}
\widehat{k}\sqrt{-1}e_7+k\sqrt{-1}e_8&e_1+e_2\sqrt{-1}&-e_3+e_4\sqrt{-1}\\
-e_1+e_2\sqrt{-1}&\widehat{l}\sqrt{-1}e_7+l\sqrt{-1}e_8&e_5+e_6\sqrt{-1}\\
e_3+e_4\sqrt{-1}&-e_5+e_6\sqrt{-1}&\widehat{m}\sqrt{-1}e_7+m\sqrt{-1}e_8\\
\end{pmatrix}
$$
The Maurer--Cartan equation for $\Sigma$ translates to
\begin{equation}
\label{eq:MC-SU(3)}
\begin{alignedat}{4}
de_1&=3m e_{27}-\widehat{m}e_{28}-e_{35}+e_{46}, & \qquad \qquad  \qquad \qquad de_2&=-3m e_{17}+\widehat{m}e_{18}+e_{36}+e_{45}\\
de_3&=e_{15}-e_{26}+3le_{47}-\widehat{l}e_{48}, &\qquad de_4&=-e_{16}-e_{25}-3le_{37}+\widehat{l}e_{38}\\
de_5&=-e_{13}+e_{24}+3ke_{67}-\widehat{k}e_{68}, &\qquad de_6&=e_{14}+e_{23}-3ke_{57}+ \widehat{k}e_{58}\\
de_7&=\frac{m}{\Delta}e_{12}+\frac{l}{\Delta}e_{34}+\frac{k}{\Delta}e_{56}, &\qquad de_8&=-\left(\frac{\widehat{m}}{\Delta}e_{12}+\frac{\widehat{l}}{\Delta}e_{34}+\frac{\widehat{k}}{\Delta}e_{56}\right).
\end{alignedat}
\end{equation}
We can identify $T_pN_{k,l} \cong \C^3 \oplus \R \cong \R^7$ via the map $dz_1= e_1 +i e_2$, $dz_2 = e_3+ie_4$, $dz_3= e_5+ ie_6$ and $\eta =e_7$ for the real factor. Under this identification, we have the usual $G_2$ 3-form, given by 
\begin{align*}
    \varphi_0 &= - \operatorname{Re}(dz_1\wedge dz_2 \wedge dz_3) + \eta \wedge \frac{i}{2}( dz_1 \wedge d \overline{z_1} + dz_2 \wedge d \overline{z_2} +dz_3 \wedge d \overline{z_3})\\
    &= e_{146}+e_{236}+e_{245} -e_{135} + (e_{12} + e_{34}+e_{56})\wedge e_7\;,
\end{align*} 
compatible with the Euclidean inner product on $\R^7$.  Since the construction is $\operatorname{SU}(3)$-equivariant, we have a four-dimensional family of $G_2$-structures on $N_{k,l}$, given by a rescaling of each coordinate, $$(dz_1, dz_2, dz_3, \eta) \mapsto (a\;dz_1, b\;dz_2, c\;dz_3, f\;\eta)\;.$$
Explicitly, we have the four-parameter family of $G_2$-structures
$$ \varphi = abc(e_{146}+e_{236}+e_{245}-e_{135}) + (a^2\;e_{12} + b^2 \;e_{34}+ c^2\;e_{56})\wedge f e_7\;. $$
The associated metric $g_\varphi$ and Hodge dual $\psi = *\varphi$ on $N_{k,l}$ are given by
\begin{align*}
    g_\varphi&= a^2( e_1^2 + e_2^2) + b^2 ( e_3^2 + e_4^2) + c^2 ( e_5^2 + e_6^2)  + f^2e_7^2 \\ \psi &= *\varphi = -abcf(e_{136}+e_{145}+e_{235}-e_{246} )\wedge e_7 +(bc)^2 e_{3456} + (ac)^2 e_{1256} +(ab)^2 e_{1234}
\end{align*} 
Using the Maurer-Cartan formulae \eqref{eq:MC-SU(3)}, we can compute the differentials of the characteristic forms $\varphi$ and $\psi$,
\begin{align*}
    d \varphi=& - f(a^2+b^2+c^2)\left( e_{136}+e_{145}+e_{235}-e_{246} \right)\wedge e_7 \\&+  \left(\frac{b^2m+a^2l}{\Delta}f+4abc\right)e_{1234} +\left(\frac{a^2k+c^2m}{\Delta}f+4abc\right)e_{1256} +\left(\frac{b^2k+c^2l}{\Delta}f+4abc\right)e_{3456} \\
    d \psi =&0\;.
\end{align*}
\begin{remark}
    There is an additional parameter one could consider to construct the $G_2$-structure, namely $$\varphi_t=-\operatorname{Re}(e^{it} dz_1 \wedge dz_2 \wedge dz_3) + \eta \wedge \frac{i}{2}( dz_1 \wedge d \overline{z_1} + dz_2 \wedge d \overline{z_2} +dz_3 \wedge d \overline{z_3})\;.$$ 
    The choice $t=0 \mod 2\pi $ is equivalent to $d\psi=0$, as discussed in \cite{CMS96}.
\end{remark}
Let us now consider a product $(-\varepsilon, \varepsilon) \times N_{k,l}$, with $\operatorname{Spin(7)}$-structure $\Phi= dt \wedge \varphi +\psi$. The $\operatorname{Spin(7)}$ holonomy condition reduces to $ d\varphi = \partial _t\psi$ and $d\psi=0$.
For a family of homogeneous metrics, this translates into 
\begin{subequations}
\begin{align}
    \partial_t (abcf) & =  f(a^2+b^2+c^2)\;, \label{eq:evo-abcf}\\  \partial_t (bc)^2 & = \frac{b^2k+c^2l}{\Delta}f+4abc\;, \label{eq:evo-bc2}\\
    \partial_t (ac)^2 & = \frac{a^2k+c^2m}{\Delta}f+4abc\;,\label{eq:evo-ac2}\\
    \partial_t (ab)^2 & = \frac{b^2m+a^2l}{\Delta}f+4abc \label{eq:evo-ab2}\;.
\end{align}
\end{subequations}
These are equivalent to the system 
\begin{subequations}\label{eq:evo-spin7}
\begin{align}
    \frac{\dot a}{a}&=  \frac{c}{ab}+ \frac{b}{ac} -\frac{a}{bc}+\frac{m}{2\Delta}\frac{f}{a^2}\;, \label{eq:evo-a}\\
    \frac{\dot b}{b}&= \frac{c}{ab}+\frac{a}{bc}-\frac{b}{ac}+\frac{l}{2\Delta}\frac{f}{b^2} \;, \label{eq:evo-b}\\
    \frac{\dot c}{c}&=  \frac{a}{bc}+\frac{b}{ac}-\frac{c}{ab}+\frac{k}{2\Delta}\frac{f}{c^2}\;, \label{eq:evo-c} \\
    \frac{\dot f}{f}&= -\frac{m}{2\Delta}\frac{f}{a^2}-\frac{l}{2\Delta}\frac{f}{b^2}-\frac{k}{2\Delta}\frac{f}{c^2}  \label{eq:evo-f} \;.
\end{align}
\end{subequations}
\begin{remark}
In previous work (cf. \cite{Chi22} \cite{Chi25}), the convention $f>0$ was chosen, forcing the $\mathrm{Spin}(7)$ equations to split into two sign choices. In the present paper, we allow $f$ to take both positive and negative values, thereby absorbing the sign ambiguity into the metric coefficient $f$ itself. This enables us to treat the two cases uniformly within a single system of equations.
\end{remark}
With respect to the metric $g=  dt^2  + g_\varphi$, we have the orthonormal co-frame:
\begin{subequations} \label{eq:orthonormal-basis}
\begin{align}
\omega_0=dt,\qquad \qquad\omega_1&=a\;e_1,\qquad \qquad\omega_2=a\;e_2,\qquad \qquad\omega_3=b\;e_3,\\
\omega_4=b\;e_4, \qquad \qquad\omega_5&=c\;e_5,\qquad \qquad\omega_6=c\;e_6,\qquad \qquad \omega_7=f\;e_7.
\end{align}
\end{subequations}
With respect to this basis, the $\Spin(7)$ $4$-form translates to 
\begin{align*}
\Phi=&\omega_{0146}+\omega_{0236}+\omega_{0245} -\omega_{0135} + \omega_{0127} + \omega_{0347}+\omega_{0567}\\& -\left(\omega_{1367}+\omega_{1457}+\omega_{2357}-\omega_{2467}\right)+\omega_{1234}+\omega_{1256}+\omega_{3456}\;.
\end{align*}
The ODE  system \eqref{eq:evo-spin7} describes the geometry on the cylinder $(-\varepsilon, \varepsilon) \times N_{k,l}$, which represents a neighbourhood of the principal orbits in a cohomogeneity one manifold. To obtain a complete non-compact space, this cylinder must be capped off at a singular orbit by letting the scale factors $a,b,c$ vanish at an appropriate rate, so that the metric extends smoothly or as an orbifold across the boundary. 

Cohomogeneity one spaces are characterised by their group diagram $K \subset H \subset G$, so $$M \cong (G/H) \cup (0,\infty) \times (G/K) \cong \Sigma_S \cup (0,1) \times \Sigma\;.$$
In our case, $G\cong \SU(3)$ and $K\cong \U(1)_{k,l}$, so the principal orbit is $\Sigma \cong N_{k,l}$, the Aloff--Wallach space. The possible singular orbits are governed by an intermediate subgroup $H$ satisfying $K=U(1)_{k,l} \subset H \subset SU(3)$. The asymptotic behaviour of solutions to \eqref{eq:evo-spin7} near the singular end determines precisely which embeddings of $K$ into $H$ are admissible. In \cite{Reidegeld_11}, Reidegeld classified all the possible intermediate groups $H$ admitting a local holonomy $\Spin(7)$-metric, and showed that $H \cong U(2)$, except when the principal orbit is $N_{1,0}$, where $H$ is either $\U(2)$ or $\SU(2)$. 

There are three distinct embeddings of $U(2)$ in $SU(3)$ containing the subgroup $U(1)_{k,l}$, which are determined by the following embeddings of $U(1)_{k,l}$ in $U(2)$. 
$$ \rho_{\abs{m}}: z\mapsto \diag (z^k, z^l) 
\qquad \qquad 
 \rho_k: z\mapsto \diag (z^l, z^m) 
\qquad \qquad 
 \rho_l: z\mapsto \diag (z^m, z^k) \;. 
$$
For each $\rho_i$ with $i\in \{k, l, \abs{m}\}$, we have an orbifold desingularisation of the metric cone over $N_{k,l}$, labelled $M^i_{k,l}$, that can be visualised as follows. Each generic Aloff--Wallach space $N_{k,l}$ can be realised as three different lens space fibrations over $\C P^2$:
$$ S^3/ (\Z/i\Z)  \rightarrow N_{k,l} \rightarrow \C P^2 \qquad i \in \{ k, l, \abs{m}\},$$
giving rise to the three topologically distinct cohomogeneity one orbifold resolutions, $M^i_{k,l}$, where each is an $\R^4/ (\Z/i\Z)$-bundle over $\C P^2$. In particular, we have a smooth desingularisation whenever $l=1$, $M^{1}_{k,1}$. In the exceptional case $N_{1,1}$, there are only two distinct resolutions. The smooth one $M^{1}_{1,1}$ is diffeomorphic to $T^*\C P^2$ and carries a hyperkähler metric, originally constructed by Calabi \cite{Calabi79}. The $\Z/2\Z$ orbifold resolution $M^2_{1,1}$ was originally considered by Bryant and Salamon in \cite{BS89}.

In the exceptional case $N_{1,0}$, the Aloff--Wallach space has two sphere fibrations 
$$ S^3 \rightarrow N_{1,0} \rightarrow \C P^2  \qquad \qquad   S^2 \rightarrow N_{1,0} \rightarrow S^5\;,$$
which yield two smooth resolutions, $M^1_{1,0}$ and $M^0_{1,0}$ respectively. 

The following lemma about the relative size of the exceptional orbit is straightforward, but will be convenient for normalisation purposes later on.
\begin{lemma}\label{lemma:volume_normalisation}
Equip $N_{k,l}$ with the metric $g_1=\sum_{j=1}^7 e_j^2$ and
volume form $\vol_1=e_{1234567}$. Consider the homogeneous fibration
\[
\U(2)/\U(1)_{k,l}\longrightarrow N_{k,l}\longrightarrow\C P^2
\]
determined by the inclusions
\[
\U(1)_{k,l}\hookrightarrow\U(2)
\xhookrightarrow{\rho_i}\SU(3),
\qquad i\in\{k,l,|m|\}\setminus\{0\}.
\]
The quotient metric on $\C P^2$ is the Fubini--Study metric, and
the volumes of the total space, base, and fibre, with their
respective induced metrics, are
\[
\vol_1(N_{k,l})=\frac{\pi^4}{2\Delta},
\qquad
\vol(\C P^2)=\frac{\pi^2}{2},
\qquad
\vol(S^3/\Z_i)=\frac{\pi^2}{\Delta}.
\]
\end{lemma}
\begin{proof}
    Take the inner product $B(X,Y) = -\frac{1}{2} \operatorname{Tr}(XY)$, so $\{E_i\}$ is an orthogonal basis, and 
    $$B(E_i,E_j) =\delta_{ij}\; \; \text{ for }\;\; i,j\in\{ 1, \dots, 6\}\;,  \qquad B(E_7,E_7) = 3\Delta \qquad \text{and} \qquad B(E_8,E_8) =  \Delta\;.$$
    It is a classical result of Macdonald \cite{MacD80} that the volume for $\SU(3)$ with respect to $B$ is $ \sqrt{3} \pi^5$, and since $k$ and $l$ are coprime, the $\U(1)_{k,l}$-subgroup generated by $E_8$ has period $2 \pi$, so 
     $$ \vol_1({N_{k,l}} ) = \frac{\vol(\SU(3))}{2\pi \abs{E_7} \abs{E_8}} =  \frac{\pi^4}{2 \Delta}\;.$$
    Now, note that the induced metric on $\C P^2$ is the Fubini--Study metric, so the volume computation is straightforward. The volume of the fibres follows from the relation $\vol(N_{k,l}) = \vol(\C P^2) \vol( S^3/\Z_i)$.
\end{proof}
\subsection{Complete \texorpdfstring{$\Spin(7)$}{Spin(7)}-metrics}
In many constructions of non-compact cohomogeneity one structures, normalising the orbit space using the mean curvature of the principal orbit leads to a compactified phase space. The existence problem then reduces to finding integral curves of the normalised system connecting critical points corresponding to the initial conditions and asymptotic limits. 

We introduce a different change of coordinates that also leads to a compactified system. Take the time reparametrisation $dt=fds$ and consider functions
\begin{equation}\label{eq:change-ABC} 
    A=\frac{af}{bc},\quad B=\frac{bf}{ca},\quad C=\frac{cf}{ab}.
\end{equation}
Denoting the derivative with respect to $s$ as $(\,\cdot\,)'$, the $\Spin(7)$ condition \eqref{eq:evo-spin7} transforms to
\begin{subequations}\label{eq:new-evo-spin7}
\begin{align}
    A'&=  A\left(B+C-3A-\frac{k}{\Delta}AB-\frac{l}{\Delta}AC\right)\;, \label{eq:new-evo-a}\\
    B'&=  B\left(C+A-3B-\frac{m}{\Delta}BC-\frac{k}{\Delta}AB\right) \;, \label{eq:new-evo-b}\\
    C'&=  C\left(A+B-3C-\frac{l}{\Delta}AC-\frac{m}{\Delta}BC\right)\;, \label{eq:new-evo-c} \\
    f'&= -\frac{f}{2\Delta}\big(k\;AB+l\;AC+m\;BC\big) \label{eq:new-evo-f} \;.
\end{align}
\end{subequations}
In these coordinates, the $(A,B,C)$-subsystem is independent of $f$, so we restrict our attention to it. In these coordinates, the singular orbits diffeomorphic to $\mathbb{C}P^2$ correspond to the fixed points
\begin{equation}
    p_0^m=\left(0,-\frac{2\Delta}{m},-\frac{2\Delta}{m}\right),\qquad  p_0^l=\left(-\frac{2\Delta}{l},0,-\frac{2\Delta}{l}\right)\qquad \text{and} \qquad p_0^k=\left(-\frac{2\Delta}{k},-\frac{2\Delta}{k},0\right).
\end{equation}
The two nearly parallel $G_2$-structures on the Aloff--Wallach space $N_{k,l}$ (cf. \cite{CMS96}) are each represented by critical points with $ABC\neq 0$. Integral curves converging to either $p_{k,l}^\pm$ correspond to asymptotically conical (AC) metrics. 

Finally, the origin is a degenerate fixed point, corresponding to metrics with an asymptotically locally conical (ALC) end. If we look at the full configuration space $(A,B,C,f)$, the fixed point is transformed into the critical manifold $\{(0,0,0,\ell)\mid \ell\neq 0\}$, where $\abs{\ell}$ correspond to the asymptotic length of the $S^1$-fibre, up to fixed normalisation.

Note that the change of variables, Equation \eqref{eq:change-ABC}, forces the functions $A, B, C$ and $f$ to have the same sign. Therefore, integral curves emanating from $p_0^m$ are contained in the subspace $\mathcal{C}_+:=\{A,B,C,f\geq 0\}$, while integral curves emanating from $p_0^k$ and $p_0^l$ are contained in $\mathcal{C}^-\coloneqq \{A,B,C,f\leq 0\}$.

The linearisation of the dynamical system \eqref{eq:new-evo-spin7} at $p_0^i$, denoted $L^i_0$, is  
\begin{equation}\label{eq:linearisation-ABC}
L_0^m = \frac{2\Delta}{m} \begin{pmatrix}
    -2&0&0\\
    -\frac{\widehat{m}}{m}&1&-3\\
    \frac{\widehat{m}}{m}&-3&1\\
\end{pmatrix} \qquad \qquad 
L_0^l =  \frac{2\Delta}{l} \begin{pmatrix}
    1 & \frac{\widehat{l}}{l} & - 3\\
    0 & -2 & 0 \\
     -3  & -\frac{ \widehat{l}}{l}& 1  
\end{pmatrix}
\qquad \qquad 
L_0^k = \frac{2\Delta}{k} \begin{pmatrix}
    1 &  -3 & - \frac{\widehat{k}}{k}\\
    -3 & 1 &  \frac{\widehat{k}}{k}\\
    0 & 0 & -2
\end{pmatrix}
\end{equation}
By the Hartman--Grobman theorem, we are interested in considering integral curves emanating from the fixed points. Since $p_0^m \in \mathcal{C}_+$ and $p_0^l , p_0^k \in \mathcal{C}_-$, we are interested in the positive eigenspace of $L_0^m$ and the negative eigenspaces of $L_0^l$ and $L_0^k$, as $f\leq 0$ in the latter case. In all  cases, the space is 2-dimensional, with eigenvalues $\frac{-4\Delta}{i}$ for the corresponding $i \in \{k,l,m\}$. The corresponding fixed point for the fibre function $f$ is $f=0$, with linearisation $L_0(f)=-\frac{2\Delta}{i}$.

The admissible eigenvectors in each case are 
\begin{equation} \label{eq:eigenvectors-spin7}
\begin{array}{c|c|c|c}
&p_0^{k}&p_0^{l}&p_0^{m}\\
\hline
&&&\\
u_1& \begin{pmatrix}
1\\
1\\
0
\end{pmatrix}&\begin{pmatrix}
1\\
0\\
1
\end{pmatrix}&\begin{pmatrix}
0\\
-1\\
-1
\end{pmatrix}\\
&&&\\
u_2& \begin{pmatrix}
-2\widehat{k}\Delta\\
0\\
-6k\Delta\\
\end{pmatrix}&\begin{pmatrix}
0\\
-6l\Delta\\
-2\widehat{l}\Delta
\end{pmatrix}&\begin{pmatrix}
-6m\Delta\\
-2\widehat{m}\Delta\\
0
\end{pmatrix}\\[1ex]
\end{array}.
\end{equation}
Hence, we consider the one-parameter family of integral curves,
\begin{equation}\label{eq:local_gammai}
    \gamma^i_\theta  = p_0^i+\left(  \cos(\theta) u_1 + \sin (\theta) u_2\right) e^{-\frac{4\Delta}{i}s} +  o\left(  e^{-\frac{4\Delta}{i}s} \right)
\end{equation}
for $\theta \in [0,\pi]$. The radial parameter is fixed, since it corresponds to a time reparametrisation and does not give rise to distinct metrics.

These local integral curves correspond to smooth $\Spin(7)$-metrics near the singular orbit $\C P^2$ of $M^i_{k,l}$ and match the ones considered by Reidegeld \cite{Reidegeld_11}, up to normalisation. 

Recently, it was proved that these local solutions are forward complete for suitable choices of $\theta$:
\begin{theorem}[{\cite[Thm. 1.1]{Chi25}}]
For any coprime pair $(k, l)$, on each $M^i_{k,l}$ with $i \in \{k, l, \abs{m}\}$, there exists a continuous one-parameter family of forward complete $\operatorname{Spin}(7)$-metrics $\gamma^i_\theta$ for $\theta \in (0, \theta^*]$. For $\theta \in (0, \theta^*)$, the metric $\gamma^i_\theta$ is asymptotically locally conical (ALC), with its ALC limit modelled on an $S^1$-bundle over the $G_2$ cone on the nearly Kähler $\operatorname{SU}(3)/T^2$. For the endpoint $\theta^*$, the limiting metric $\gamma^i_{\theta^*}$ is asymptotically conical (AC).
\end{theorem}
The $\Spin(7)$ metrics on $M^m_{k,l}$ are represented by integral curves in the forward invariant set $\mathcal{C}_+:=\{A,B,C,f\geq 0\}\cap\{B,C\geq A\}$, while the $\Spin(7)$ metrics on $M^k_{k,l}$ and $M_{k,l}^l$ are represented by integral curves in the backward invariant set $\mathcal{C}_-:=\{A,B,C,f\leq 0\}\cap\{B,C\geq A\}$. For the invariance of $\mathcal{C}^\pm$, see Lemma \ref{lemma:comparing-ABC} below.

The existence proof relies on the following identities and inequalities, which will also be useful in our work.
\begin{lemma}
\label{lemma:comparing-ABC}
Let $(A,B,C)$ be a solution to the system. We have the barrier-like identities
\begin{subequations}
\begin{align}
\left.(B-A)'\right|_{A=B}=\frac{\widehat{k}}{\Delta}ABC \,, \qquad \left.(A-C)'\right|_{A=C}=\frac{\widehat{l}}{\Delta}ABC \; \quad \text{and} \quad \left.(C-B)'\right|_{B=C}=\frac{\widehat{m}}{\Delta}ABC \;.
\end{align}
\end{subequations}
\end{lemma}
\begin{proposition}\label{prop:ABC-inequalities}
    Along $\gamma_{\theta}^{k+l}$ and $\gamma_\theta^k$, we have the inequalities $A \leq B \leq C$. Along $\gamma_\theta^l$, we have $B\geq A$ and $C \geq A$. 
\end{proposition}
\begin{proof}
Using the local expression for $\gamma_\theta^i$, \eqref{eq:local_gammai}, all the desired inequalities are satisfied along the integral curve initially. Then the claim follows from Lemma \ref{lemma:comparing-ABC}, since the integral curves $\gamma_\theta^k$ and $\gamma_\theta^l$ flow in reversed time.  
\end{proof}
Finally, we collect some miscellaneous results about these $\Spin(7)$-metrics. 
\begin{lemma}\label{lemma:identities-klm-khatlhatmhat}
Let $F= \frac{1}{\Delta}\left( kAB + lAC +mBC\right) $ and $\widehat{F}= \frac{1}{\Delta}\left(\widehat{k}AB  + \widehat{l}AC + \widehat{m}BC\right)$. The following identities hold:
\begin{subequations}
\begin{align}
    k \widehat{F }-\widehat{k}F = 2C(B-A)\;,& \qquad&\widehat{k}\widehat{F} +3k F= 4AB - 2AC-2BC\;,\\
    l \widehat{F }-\widehat{l}F = 2B(A-C)\;,& \qquad&\widehat{l}\widehat{F} +3l F= 4AC - 2AB-2BC\;,\\
    m \widehat{F }-\widehat{m}F = 2A(C-B)\;,& \qquad&\widehat{m}\widehat{F} +3m F= 4BC - 2AB-2AC\;.
\end{align}
\end{subequations}
\end{lemma}
\begin{proposition}\label{prop:AC-inequalitites}
Let $p^\pm_{k,l}=(A_\pm , B_\pm , C_\pm)$, and let $\mathsf{b}_\pm=\frac{B_\pm}{A_\pm}$ and $\mathsf{c}_\pm=\frac{C_\pm}{A_\pm}$. 
\begin{enumerate}[label=(\roman*)]
    \item The pair $(\mathsf{b},\mathsf{c})$ satisfies the algebraic equations
    \begin{equation}
    \label{eq:homogenised_AC_equation}
    \begin{split}
        \mathsf{L}_1&:=(\mathsf{b}+\mathsf{c}-3)^2+4(\mathsf{b}-\mathsf{c})^2-4=0\\
        \mathsf{L}_2&:=k\mathsf{b}(3+3\mathsf{c}-5\mathsf{b})-l\mathsf{c}(3+3\mathsf{b}-5\mathsf{c})=0.
    \end{split}
    \end{equation}
    \item The difference is bounded by $ \abs{\mathsf{b}-\mathsf{c}}<\frac{3}{5}$.
    \item The sum satisfies
    $$
    \frac{23}{5}<\mathsf{b}+\mathsf{c}<5 \quad \text{for }(\mathsf{b}_+,\mathsf{c}_+),
    \qquad \text{and} \qquad
    1<\mathsf{b}+\mathsf{c}<\frac{7}{5} \quad \text{for }(\mathsf{b}_-,\mathsf{c}_-).
    $$
    \item The inequalities $\mathsf{b}_- \geq \mathsf{c}_-$ and $ \mathsf{b}_+ \leq \mathsf{c}_+$ hold. 
\end{enumerate}
\end{proposition}
\begin{proof}
The first three properties are proved in \cite{Chi25}. The last statement is a corollary of Proposition \ref{prop:ABC-inequalities}.
\end{proof}
\begin{proposition}\label{prop:ALC-asymptotic-behaviour}
    Let $ \gamma_\theta^i$ be an ALC metric. Then, near the non-compact end, we have the asymptotic behaviour 
    $$ A(s) =s^{-1} +o(s^{-1}) \qquad  B(s) = s^{-1} +o(s^{-1})   \qquad  C(s) = s^{-1} +o(s^{-1})\;.$$
\end{proposition}
\begin{proof}
Since the metric is asymptotically locally conical, we have $a(t), b(t),c(t) \sim t$ and $f\to  l_\theta$, so $s \sim t$ near the non-compact end. Hence, we have $A,B,C \sim \frac{1}{s}$. It remains to show that the leading coefficients are exactly one. In order to do so, it is convenient to consider the blown-up system.

Let $\tau = \log \abs{s}$ and introduce the additional variable $\widehat{s}=  \frac{1}{s}$, and consider the rescaled functions $A =\widehat{s} \widehat{A}$, $B =\widehat{s} \widehat{B}$ and $C =\widehat{s} \widehat{C}$. The system \eqref{eq:new-evo-spin7} becomes 
\begin{subequations}
    \begin{align}
        \partial_\tau \widehat{A} & = \widehat{A}\left[ \widehat{B} + \widehat{C}- 3\widehat{A} +1 - \frac{\widehat{s}}{\Delta}\left(k\widehat{A}\widehat{B}  +l \widehat{A}\widehat{C} \right)\right]\,, \\
        \partial_\tau \widehat{B} & = \widehat{B}\left[ \widehat{A} + \widehat{C}- 3\widehat{B} +1 - \frac{\widehat{s}}{\Delta}\left(k\widehat{A}\widehat{B} +m \widehat{B}\widehat{C} \right)\right]\,, \\
        \partial_\tau \widehat{C} & = \widehat{C}\left[ \widehat{A} + \widehat{B}- 3\widehat{C} +1 - \frac{\widehat{s}}{\Delta}\left(l\widehat{A}\widehat{C}  +m \widehat{B}\widehat{C} \right) \right]\,, \\
        \partial_\tau \widehat{s}  & =  -\widehat{s}\,.
    \end{align}
\end{subequations}
It is straightforward to compute the fixed points of the system. The only one satisfying $\widehat{A} \widehat{B}\widehat{C} \neq 0$ is $(1,1,1,0)$, as needed. 
\end{proof}
Finally, let us establish some asymptotic results in the neighbourhood of the singular orbit $Z_0$ when it is diffeomorphic to $\C P^2$. We carry out the computations for $\gamma_{k,l}^m$, but the results carry over verbatim.
\begin{lemma}[{\cite[Thm. 7.1]{Reidegeld_11}}]
    The $\Spin(7)$-metric $\gamma^m_{k,l}$ has asymptotic behaviour 
    $$a(t) =  t +  o(t) \qquad b(t) = b_0 +o(1) \qquad c(t) = b_0 +o(1)\qquad f(t) = -\frac{2\Delta}{m}t +o(t)\;.$$ 
\end{lemma}
\begin{proof}
    The equality $b(0)=c(0)$ follows from the isotropy representation of the singular orbit. Since the fixed point corresponding to the singular orbit is $B=-\frac{2\Delta}{m}$, we may take the leading order expansions $a(t) = \alpha t + o(t)$, $f= -\frac{2\Delta}{m} \alpha t + o(t)$. Substituting in Equation \eqref{eq:evo-a} we have
    $$\dot{a} =  \frac{c}{b}+\frac{b}{c}- \frac{a^2}{bc}+ \frac{m}{2\Delta }\frac{f}{a} =  1 + o(1)\;,$$
    so $\alpha=1$ and the claim follows.
\end{proof}
\begin{lemma}\label{lemma:normal_metric}
   In a neighbourhood of the singular orbit, the metric in the normal direction converges to $g_{eucl}$, the Euclidean metric on the quotient $\R^4/\Z_{\abs{m}}$. 
\end{lemma}
\begin{proof}
    By equivariance, it suffices to prove the claim at an arbitrary point $p$ in the singular orbit. The sphere bundle of the normal bundle at $p$ is $S^3/\Z_m\cong \SU(2)/\Z_m$. Its tangent space with respect to the $\SU(3)$-action is spanned by $E_1$, $E_2$ and 
    $$H= \sqrt{-1} \begin{pmatrix} 1 &0&0\\ 0 &-1& 0 \\ 0&0&0      
    \end{pmatrix} = \left( - \frac{m}{2\Delta} E_7  + \frac{\widehat{m}}{2\Delta}E_8\right) \Bigg|_{N_{k,l}} = - \frac{m}{2\Delta} E_7\;.$$
    Let \(\sigma_1,\sigma_2,\sigma_3\) be the dual basis to \(H , E_1,E_2\) of $\mathfrak{su}(2)$, so 
    $d \sigma_i =  -2\varepsilon_{ijk} \, \sigma_{jk}$.
    Restricting to the Aloff--Wallach space $N_{k,l}$, the metric in  the normal directions to the singular orbit in geodesic coordinates reads
    \[g_{\R^4/\Z_m} = dt^2 +a^2(t)(\sigma_2^2 + \sigma^2_3) + \left(\frac{m}{2\Delta}\right)^2f^2(t) \sigma_1^2 +o(t^2) =  dt^2 + t^2\sum_{i=1}^3\sigma_i^2 + o(t^2)\]
    which one immediately recognises as the flat Euclidean metric at leading order.
\end{proof}

\subsection{Principal bundles and connections over cohomogeneity one manifolds}
We need to understand which principal bundles over the principal orbit extend to the singular orbit. For this to work, one considers homogeneous principal bundles. It is a folklore result that homogeneous principal $L$-bundles over a homogeneous base $G/H$ are classified by conjugacy classes of morphisms $\psi: H \rightarrow L$, where the corresponding principal bundle is the associated bundle $G \times_\psi  L$. A detailed discussion can be found in \cite{AMW25}. We present a short overview for completeness. For a base with a cohomogeneity one action, we have
\begin{proposition}\label{prop:admissible cohomogeneity bundles}
Let $M$ be a cohomogeneity one manifold with principal orbit
$G/K$ and singular orbit $G/H$. A homogeneous principal
$L$-bundle over $G/K$ admits a $G$-equivariant extension over
a tubular neighbourhood of $G/H$ if and only if it is an
equivariant pullback under $\pi:G/K\to G/H$.
\end{proposition}
\begin{proof}
Write a tubular neighbourhood of the singular orbit as
$G\times_H D(V)$, where $V$ is the representation corresponding to the normal slice and $D(V)= \left\{ v\in V \st {||v||} < \varepsilon\right\}$ for some $\varepsilon>0$.

An equivariant principal bundle on this neighbourhood is
equivariantly isomorphic to the pullback of its restriction
to the zero section $G/H$. Its restriction to a principal
orbit is therefore a pullback under $\pi$.

Conversely, any homogeneous principal bundle over $G/H$
pulls back under the tubular projection
$G\times_H D(V)\to G/H$.
\end{proof}
Thus, given a chain of closed subgroups $K\subseteq H\subseteq G$,
we wish to characterise the homogeneous principal $L$-bundles
over $G/K$ that arise as equivariant pullbacks under the
projection $\pi:G/K\to G/H$.
\begin{lemma}\label{lemma:homogeneous pullback bundle}
Let $\iota:K\hookrightarrow H$ denote the inclusion. A
homogeneous principal $L$-bundle over $G/K$, classified by
$\varphi:K\to L$, is the equivariant pullback of a homogeneous
bundle over $G/H$ if and only if $\varphi$ is conjugate to
$\psi\circ\iota$ for some homomorphism $\psi:H\to L$.
\end{lemma}
\begin{proof}
Let $P_\psi=G\times_\psi L\to G/H$. There is a natural
$G$-equivariant isomorphism
\[
G\times_{\psi\circ\iota}L\longrightarrow\pi^*P_\psi,
\qquad
[g,\ell]_K\longmapsto(gK,[g,\ell]_H).
\]
The claim follows from the classification of homogeneous
principal bundles by conjugacy classes of homomorphisms.
\end{proof}
We need to understand the conjugacy classes of morphisms for the different Lie groups involved in our construction. The following is a routine computation:
\begin{lemma}\label{lemma:morphism classification}
The conjugacy classes of group homomorphisms between the corresponding groups are as follows:
\begin{enumerate}[label=(\roman*)]
    \item $\Map(\U(1),\U(1))\cong \Z$, of the form $\phi_n: z \mapsto z^n$ for $n \in \Z$;
        \item $\Map(\SU(2),\U(1))\cong \{e\}$;
    \item $\Map(\U(2),\U(1))\cong \Z$, of the form $\tau_n: A \mapsto \det(A^n)$  for $n \in \Z$;
    \item $\Map(\U(1), \SO(3))= \mathbb{N}$ given by $\theta_n:z \mapsto \begin{pmatrix}
        1 &\\ & z^n
    \end{pmatrix}$, with $z \in U(1) \subseteq  \operatorname{GL}(1,\C) \hookrightarrow \operatorname{GL}(2,\R)$;  
    \item $\Map(\SU(2),\SO(3))\cong \{e, \pi\}$, the trivial morphism and the covering map;
    \item $\Map(\U(2),\SO(3))\cong \mathbb{N} \cup \{\widehat{\pi}\}$, with 
    \begin{align*}
        \widehat{\tau}_n : \U(2) & \to \SO(3)\\ A & \mapsto \theta_n\circ \tau_1  (A) =\theta_1 \circ \tau_n (A) \;,        
    \end{align*}
    and $\widehat{\pi}$ is the projectivisation map $\widehat{\pi}: \U(2) \rightarrow \operatorname{PU}(2)\cong \SO(3)$.
\end{enumerate}
\end{lemma}
\begin{remark}
Note that in \cite{BO19}, Ball and Oliveira parametrise morphisms from $\U(1)$ to $\SO(3)$ by integers rather than natural numbers. This is an over-counting, since $\theta_n$ and $\theta_{-n}$ are conjugated by  
$
    C =\begin{pmatrix}
    -1 & 0 & 0 \\ 0 &0 & 1 \\ 0&1 &0
    \end{pmatrix}
$.
\end{remark}LOL
Next, we classify admissible cohomogeneity one principal $G$-bundles with $G=\U(1)$ and $\SO(3)$ over $M^i_{k,l}$. We distinguish two cases, depending on the topology of the singular orbit.
\begin{proposition}[Generic case - $\C P^2$ as a singular orbit]\label{prop:admissible-bundles-CP2}
    Let $N_{k,l}$ be an Aloff--Wallach space, and let $M_{k,l}^i$ be the associated orbifold, with $i\neq 0$. Then a principal $\U(1)$-bundle $\phi_n$ on $N_{k,l}$ is admissible for $M^i_{k,l}$ iff $n$ is a multiple of $i$, $n \equiv_i 0$. 
    
    A principal $\SO(3)$-bundle $\theta_n$ on $N_{k,l}$ is admissible for $M^i_{k,l}$ if and only if either
    \begin{itemize}
        \item $n \equiv_i 0$, giving rise to a principal bundle $P_n$, or 
        \item $n =\abs{\widehat{i}}$, giving rise to a principal bundle $Q_n$.
    \end{itemize}
    In particular, a bundle $\theta_n$ admits a unique equivariant extension over the singular orbit, except for two infinite families
    \begin{enumerate}[label=(\roman*)]
        \item $M^1_{k,1}$, with $n = 2k+1$ for all $k \geq1$, and
        \item $M^2_{k, 2}$, with $n= 2k+2$, with $k\geq 3$ odd;
    \end{enumerate}
   and three exceptional cases:
    \begin{enumerate}[label=(\roman*)]\setcounter{enumi}{2}
        \item over $M^1_{1,0}$ with $n=1$, 
        \item over $M^2_{1,1}$ with $n=0$, and 
        \item over $M^2_{2,1}$ with $n=4$.
    \end{enumerate}
\end{proposition}
Before we move into the proof, note that the condition $i \neq 0$ is precisely the requirement for the singular orbit to be diffeomorphic to $\C P^2$. Moreover, in the case $i=1$, we are abusing notation slightly and taking the convention that $n \equiv_1 0$ for all $n \in \Z$.
\begin{proof}
By Proposition \ref{prop:admissible cohomogeneity bundles}, a bundle extends if and only if the homogeneous bundle on the principal orbit is a pullback of a bundle on the singular orbit. The proposition follows from Lemmas \ref{lemma:homogeneous pullback bundle} and \ref{lemma:morphism classification}.

We work out the computations explicitly for the embedding $\rho_m$,  the other cases being similarly straightforward. For the gauge group $L=\U(1)$, we have
$$\tau_n \circ \rho_{\abs{m}} =  \det \big( \operatorname{diag}( z^k, z^l)\big)^n = z^{n(k+l)} = \phi_{n(k+l)}\;.$$
For $L=\SO(3)$, we have
$$\widehat{\tau}_n \circ \rho_{\abs{m}}= \begin{pmatrix}
        1 & 0 \\  0 &  \det \big( \operatorname{diag}( z^k, z^l)\big)^n
    \end{pmatrix} = \begin{pmatrix}
        1 &0 \\ 0 &  z^{n(k+l)}
    \end{pmatrix} = \theta_{n(k+l)}\;,$$
which give rise to the reducible principal $\SO(3)$-bundles $P_n$. 

For the lift of the projective map $\widehat{\pi}: \U(2) \rightarrow  \operatorname{PU}(2)\cong \SO(3)$, we have 
\[\widehat{\pi} \circ \rho _{\abs{m}}(z)= \widehat{\pi}(\diag(z^k, z^l)) = \begin{bmatrix}
    z^k &  0 \\ 0 & z^l 
\end{bmatrix} = \begin{bmatrix}
   1 &  0 \\ 0 & z^{l-k} 
\end{bmatrix} = \theta_{\widehat{m}}(z)\;,\]
giving rise to the principal $\SO(3)$-bundle $Q_n$.

Characterising uniqueness of the extension over the singular orbit comes down to showing that there are no solutions to the modular equations $\widehat{i} \equiv _i 0$ for $i \in \{k,l, \abs{m}\}$ other than the cases listed in the proposition, which is straightforward.
\end{proof}

In particular, we see that each generic Aloff--Wallach space admits three different irreducible $\SO(3)$-principal bundles $Q _{\widehat{i}} \to N_{k,l}$, with $i\in \{k,l,m\}$, but only one of the three admits a smooth extension over a given desingularisation $M^i_{k,l}$.

While it is clear, by construction, that the principal $\SO(3)$-bundles $P_n$ are reducible, let us verify that $Q_n$ is irreducible. Indeed, restricting the exceptional bundle to the singular orbit $\C P^2$, one can easily identify the associated bundle $Q_{\widehat{i}}$ with the adjoint bundle of $E$, $V=\mathfrak{su}(E)$, where $E$ is the quotient bundle $0 \to \mathcal{O}(-1) \to \C^3 \to E \to 0$. In particular, $p_1(V)=-3H^2$ for $H$ the Poincaré dual of the hyperplane class. Since $-3$ is not a square, it is irreducible. 

Similarly, one proves
\begin{proposition}[Exceptional case -  $S^5$ as a singular orbit]\label{prop:admissible-bundles-S5}
On $M^0_{1,0}$, the only admissible $\U(1)$-bundle is the trivial one. For $\SO(3)$, the only admissible bundles are $P_0$ and $Q_2$. 
\end{proposition}
\begin{proof}
    In this case the singular orbit is $S^5\cong \SU(3) / \SU(2)$. Since the only morphism $\SU(2)\to  \U(1)$ is the trivial one, the only bundle that extends is the trivial one. 

    For the covering map $\pi: \SU(2) \rightarrow \SO(3)$, one has 
    $$ \pi \circ \rho_l(z)=\pi \circ \rho_0(z) =\pi \left( \diag (z, z^{-1}) \right) = \theta_2 (z)$$
    since $\SU(2)$ is the double cover of $\SO(3)$.
\end{proof}
Let us now consider admissible cohomogeneity one connections that these bundles admit. We need to determine the invariant connections on each of these homogeneous bundles. These are characterised by the following classical result, due to Wang:
\begin{theorem} Consider $G/H$ a reductive homogeneous space and let $\phi \in \Map (H, L)$ be a morphism characterising a homogeneous principal $L$-bundle $G \times_\phi L \to G/H$.  Denote by $\mathfrak{m}$ a reductive complement of $\mathfrak{h}$ in $\mathfrak{g}$. There is a one-to-one correspondence between homogeneous connections on $G \times_\phi L$ and $\operatorname{Ad}_H$-equivariant morphisms  $F: \mathfrak{m} \rightarrow \mathfrak{l}$.
\end{theorem}
Since $\mathfrak{m}$ is reductive, $Ad_H(\mathfrak{m})\subset\mathfrak{m}$, the space $\Hom_H(\mathfrak{g},\mathfrak{l})$ admits the splitting $\Hom_H(\mathfrak{h},\mathfrak{l})\oplus \Hom_H(\mathfrak{m},\mathfrak{l})$. The map $d \phi$ lies  in $\Hom_H(\mathfrak{h},\mathfrak{l})$, and the space of homogeneous connections is identified with $d\phi + \Hom_H(\mathfrak{m}, \mathfrak{l}) \subseteq \Hom_H(\mathfrak{g},\mathfrak{l})$.

Let us apply Wang's result to the different homogeneous spaces in our construction. Over the principal orbit, this was carried out by Ball-Oliveira \cite{BO19}. We give a short overview of their work.  Recall that over the principal orbit $\SU(3)/ \U(1)_{k,l}$, we have the isotropy decomposition
$$\mathfrak{m}\cong  \langle E_1, E_2\rangle_{\widehat{m}} \oplus \langle E_3, E_4\rangle_{\widehat{l}} \oplus \langle E_5, E_6\rangle_{\widehat{k}} \oplus \langle E_7\rangle \;,$$
where the subscript denotes the induced $\U(1)$-action by the adjoint. Hence, for abelian connections, we have
\begin{proposition}
Let $n \in \Z$, and consider $P_n \rightarrow N_{k,l}$ the principal $U(1)$-bundle associated to $\phi_{n}$. For $(k,l) \neq (1,1)$  the space of homogeneous $\U(1)$-connections is a one-dimensional affine space, given by 
$$\A_n = \left\{ ( \xi e_7+ n e_8 ) \otimes T\st  \xi  \in \R\right\}, 
$$
where $T \in \mathfrak{u}(1)$ is a generator of the Lie algebra. 

For $(k,l)=(1,1)$ and $n\neq 0$, the space of homogeneous $\U(1)$-connections is a three-dimensional affine space, given by 
$$\A_n = \left\{ (\chi_1 e_1 +  \chi_2e_2 + \xi e_7 + n e_8) \otimes T\; \st \chi_1, \chi_2, \xi \in \R \right\}\;.$$
\end{proposition}
\begin{proof}
By Wang's theorem, we need to classify $\U(1)$-equivariant maps $\mathfrak{m}\to i \R $. For $(k,l) \neq (1,1)$, we have $0 \notin \{\widehat{k}, \abs{\widehat{l}}, \widehat{m}\}$, so the representations $\langle e_{2i-1}, e_{2i}\rangle$, $i \in \{1,2,3\}$ are irreducible. For the exceptional case $(k,l)= (1,1)$, we have $\widehat{m}=0$, and so, there is a reduction of the isotropy representation $\langle e_1, e_2\rangle = \langle e_1\rangle \oplus \langle e_2\rangle$. The claim now follows from Schur's lemma.
\end{proof}
Let us now consider the $\SO(3)$-case. For the embedding $\theta_n: \U(1) \to \SO(3)$, we can pick a basis so $\mathfrak{so}(3) \cong \langle T_1 \rangle \oplus \langle T_2 , T_3\rangle_n$ for the induced action. In this case, we have
\begin{proposition}[Generic case] \label{prop:homogeneous connections SO3-kl}
Assume $(k,l) \notin \{(1,0), (1,1)\}$. Let $n \in \mathbb{N}$ and consider $Q_n \rightarrow N_{k,l}$ the principal $\SO(3)$-bundle associated to $\theta_{n}$. Then, if $n \notin \{\widehat{k}, \abs{\widehat{l}}, \widehat{m}\}$, all homogeneous connections are abelian up to gauge.

In the remaining cases, up to gauge, we have 
\begin{align*}
    \A_{\widehat{k}} & = \left\{ (\xi e_7+ \widehat{k} e_8 ) \otimes T_1  + \psi ( e_5 \otimes  T_2 +  e_6\otimes T_3) \st  \xi \in \R,\;  \psi \in [0, \infty) \right\} \,\\
    \A_{\abs{\widehat{l}}} & = \left\{ ( \xi e_7 + \widehat{l} e_8 ) \otimes T_1  + \psi ( e_3 \otimes T_2 +  e_4\otimes T_3) \st \xi \in \R\;, \; \psi \in [0, \infty)\right\}\\
        \A_{\widehat{m}} & = \left\{ (\xi e_7+ \widehat{m} e_8 ) \otimes T_1  + \psi ( e_1 \otimes T_2 +  e_2\otimes T_3) \st \xi \in \R, \; \psi \in [0, \infty) \right \}\,
\end{align*}
\end{proposition}
\begin{proof}
The conditions $(k,l) \notin \{(1,0), (1,1)\}$ guarantee that the non-trivial $\U(1)$-representations are pairwise non-isomorphic. For $n=0$, we have
$$\A_{0} =  \left\{ e_7 \otimes T \; \st  T \in \mathfrak{so}(3)\right\}\;.$$
But every element in \(\mathfrak{so}(3)\) is conjugate to an element of the form \(\xi T_{1}\) for some real number $\xi$ and $T_1 \in \so(3)$, so we can gauge $A\in \A_{0}$ to $\widetilde{A} =  \xi e_7 \otimes T_1$, an abelian connection. 

The cases $\A_{\widehat{m}}, \A_{\widehat{k}}$ are straightforward from Schur's lemma, where we use the gauge group to fix the freedom $\langle T_2 , T_3\rangle \mapsto \langle \cos(\theta) T_2 - \sin(\theta)T_3,\;  \sin (\theta) T_2 + \cos(\theta) T_3 \rangle$. Moreover, the adjoint action by $S= \diag (1, -1,-1)$, $\langle T_1, T_2, T_3\rangle \xrightarrow{\operatorname{Ad}(S)} \langle T_1, -T_2, -T_3\rangle$ ensures we can take $\psi \geq 0$.

The case $\A_{\widehat{l}}$ is slightly more subtle,  since $\widehat{l}<0$ and $n\geq0$. However, note that $\langle T_2, T_3 \rangle _n \cong \langle T_3, T_2\rangle_{-n}$. Hence, by Wang's theorem, we have
$$\A_{\abs{\widehat{l}}}  = \left\{ ( \xi e_7 + \abs{\widehat{l}} e_8 ) \otimes T_1  + \psi ( e_3 \otimes T_3 +  e_4\otimes T_2) \st \xi,  \psi \in \R\right\}\;, $$
where we have used the gauge group to fix the rotational freedom on $\langle T_2, T_3\rangle$. Now, the adjoint action by $C$ on our basis is $\langle T_1, T_2, T_3\rangle \xrightarrow{\operatorname{Ad}(C)} \langle -T_1, -T_3, -  T_2\rangle$, so up to relabelling $\xi$ by $-\xi$ and conjugating by $S$, we are done. 
\end{proof}
\begin{proposition}[Exceptional $(1,0)$-case]\label{prop:homogeneous connections SO3-10}
Let $n \in \mathbb{N}$ and consider $Q_n \rightarrow N_{1,0}$ the principal $\SO(3)$-bundle associated to $\theta_{n}$. Then, if $n \notin \{1,2\}$, all homogeneous connections are abelian up to gauge. In the remaining cases, up to gauge, we have 
\begin{align*}
    \A_{1} & = \left\{ (\xi e_7+ e_8 ) \otimes T_1  +\psi_1\left[e_5 \otimes  \big(\cos(\theta) T_2  - \sin(\theta) T_3 \big) +e_6 \otimes  \big(\sin(\theta) T_2  + \cos(\theta) T_3 \big) \right]  \right.  \\ & \qquad \qquad
    \left. +\psi_2 ( e_1 \otimes T_2 +  e_2\otimes T_3)\;\st  \xi,  \psi_1  \in \R,\;  \psi_2 \in [0, \infty),\;  \theta\in S^1\right\}\;,\\
    \A_{2} & = \left\{ ( \xi e_7 - 2e_8 ) \otimes T_1  + \psi ( e_3 \otimes T_2 +  e_4\otimes T_3) \st \xi \in \R , \; \psi \in [0, \infty)\right\}\;.
\end{align*}
\end{proposition}
\begin{proof}
 For $(k,l)=(1,0)$, we have $\widehat{m} = 1 = \widehat{k}$ and $\widehat{l}=-2$, so there are two isomorphic components in the isotropy representation:
$$ \mathfrak{m}  \cong  \langle e_1, e_2\rangle_{1} \oplus \langle e_3, e_4\rangle_{-2} \oplus \langle e_5, e_6\rangle_{1} \oplus \langle e_7\rangle $$
 Hence, for $\A_{1}$, we have two maps $\widehat{\psi}_i \in \Hom (\C, \C) \cong \C$. Our gauge freedom allows us to choose $\widehat{\psi}_2 \in \R$.
 
 The remaining cases follow verbatim from Schur's lemma and the discussion above.
\end{proof}

\begin{proposition}[Exceptional $(1,1)$-case]\label{prop:homogeneous connections SO3-11}
Let $n \in \mathbb{N}$ and consider $Q_n \rightarrow N_{1,1}$ the principal $\SO(3)$-bundle associated to $\theta_{n}$. Then, if $n \notin \{0,3\}$, all homogeneous connections are abelian. In the remaining cases, we have, up to gauge,
\begin{align*}
    \A_{0} & = \left\{ e_1 \otimes c_1 + e_2 \otimes c_2 + e_7 \otimes c_3\; \st c_1, c_2, c_3 \in \mathfrak{so}(3) \right\}\\
    \A_{3} & = \left\{ (\chi_1 e_1 +  \chi_2e_2 + \xi e_7 + 3 e_8) \otimes T_1 + \psi_1 \left[e_5 \otimes(\cos(\theta) T_2 -\sin(\theta) T_3 ) + e_6 \otimes (\sin(\theta) T_2 + \cos(\theta) T_3) \right]\right.\\
    & \qquad  \left. + \psi_2\big(e_3 \otimes T_3 + e_4 \otimes T_2  \big) \; \st  \xi, \chi_1, \chi_2, \psi_1 \in \R, \; \psi_2 \in [0,\infty),\; \theta\in S^1\right\}
\end{align*}
\end{proposition}
\begin{proof}
For $(k,l)= (1,1)$, the isotropy representation decomposes as
\begin{equation*}
\mathfrak{m}  \cong  \langle e_1 \rangle \oplus \langle e_2\rangle \oplus \langle e_3, e_4\rangle_{-3} \oplus \langle e_5, e_6\rangle_{3} \oplus \langle e_7\rangle\;,
\end{equation*}
and the arguments in the previous cases complete the proof.
\end{proof}
The case $\A_{0}$ on $Q_0 \rightarrow N_{1,1}$ has a residual gauge freedom that would allow us to reduce the space of connections up to gauge by two parameters, which we will exploit at a later stage.
\begin{remark}
    Ball and Oliveira overlooked the $\U(1)$-equivariant maps $\langle e_{2i-1}, e_{2i}\rangle_{-n} \to \langle T_2 , T_3\rangle_n$ given by reversing the orientation $\langle T_2, T_3 \rangle \mapsto  \langle T_3, T_2 \rangle$. This oversight has no impact on their work due to their over-counting of morphisms $\U(1) \rightarrow \SO(3)$, except in the case $N_{1,1}$ with $n=3$.
\end{remark}
For completeness, we note 
\begin{lemma}
The homogeneous space $\C P^2\cong \SU(3)/U(2)$ has a reductive complement $\mathfrak{m}  \cong \C^2$, on which the adjoint action restricts to the fundamental action of $\U(2)$ on $\C^2$. 

Similarly, $S^5 \cong \SU(3) / \SU(2)$ has a reductive complement $\mathfrak{m} \cong \C^2 \oplus \R$, on which the adjoint action restricts to the fundamental action of $\SU(2)$ on $\C^2$ and the trivial action on $\R$. 
\end{lemma}
\begin{proposition}\label{prop:connections-S5}
Any homogeneous principal $\U(1)$ or $\SO(3)$ bundle over $\C P^2$ admits a unique homogeneous connection. The space of invariant connections on $S^5$ for the trivial $\U(1)$ and $\SO(3)$ bundles is $ c \omega _7$
for $c \in \R$. The only invariant connection on $S^5$ on the bundle associated to the covering map $\pi: \SU(2) \rightarrow\SO(3)$ is the canonical connection. 
\end{proposition}

\section{Abelian connections}
We begin by studying the case of abelian connections, those with gauge group $G=\U(1)$. The advantage is twofold. On the one hand, since the Lie algebra structure is trivial, the resulting ODEs are linear with non-constant coefficients, which makes their study straightforward. On the other hand, the number of functions involved is small, and the computations are quite short. Taking advantage of this, we will be quite explicit in our derivations in this case, to help guide the reader. For the non-abelian case, we will omit most of the computations, since they follow the same principles; they  are just lengthier. 

We begin by considering the generic case $(k,l) \neq (1,1)$. Let $A^n(t)= \left(ne_8+\xi(t)e_7\right)\otimes T_1$ be a family of abelian connections over the principal part of our cohomogeneity one manifolds. Its curvature is given by
\begin{align*}
F_{A(t)}=&  dt \wedge \partial_t A^n + F^n_A = dt \wedge (\dot{\xi} \;e_7) +\frac{\xi}{\Delta}\Big(m\;e_{12}+l\;e_{34}+ k\;e_{56}\Big) - \frac{n}{\Delta}\Big(\widehat{m}\;e_{12}+\widehat{l}\;e_{34}+\widehat{k}\;e_{56}\Big)\\
=& \frac{\dot\xi}{f}\; \omega_{07}+ \frac{\xi}{\Delta} \left(\frac{m}{a^2}\omega_{12}+\frac{l}{b^2}\omega_{34}+\frac{k}{c^2}\omega_{56} \right)-\frac{n}{\Delta} \left(\frac{\widehat{m}}{a^2}\omega_{12}+\frac{\widehat{l}}{b^2}\omega_{34}+\frac{\widehat{k}}{c^2}\omega_{56} \right)\;.
\end{align*}
where we used the adapted orthonormal basis $\{\omega_i\}$ introduced in \eqref{eq:orthonormal-basis} in the second line. Hence, 
\begin{align*}
    *F_A=& \frac{\dot\xi}{f}\; \omega_{123456}+ \frac{1}{ a^2\Delta} \left( m\; \xi -n\;\widehat{m}\right)\omega_{034567}+  \frac{1}{ b^2\Delta} \left( l\; \xi -n \;\widehat{l}\right)\omega_{012567} + \frac{1}{ c^2\Delta} \left( k\; \xi -n \;\widehat{k}\right)\omega_{012347}\;.
\end{align*}
From this, the $\Spin(7)$-instanton equation $*F_A + \Phi \wedge F_A=0$ reduces to
\begin{equation}\label{eq:abelian-instaton}
    \frac{\dot\xi}{f}+ \frac{\xi}{\Delta} \left(\frac{m}{a^2}+\frac{l}{b^2}+\frac{k}{c^2} \right)-\frac{n}{\Delta} \left(\frac{\widehat{m}}{a^2}+\frac{\widehat{l}}{b^2}+\frac{\widehat{k}}{c^2} \right)=0\;.
\end{equation}
Similarly, for the Yang--Mills equation $d_A * F_A=0 $ we have 
\begin{equation}
\label{eq:abelian-YM}
\partial_t \left( \frac{(abc)^2}{f} \dot{\xi} \right) = mf\left(\frac{bc}{a \Delta}\right)^2 \left( m\; \xi -n\; \widehat{m}\right)+ lf\left(\frac{ac}{b \Delta }\right)^2\left( l\; \xi -n \;\widehat{l}\right)+kf\left(\frac{ab}{c \Delta}\right)^2\left( k\; \xi -n \;\widehat{k}\right)\;.
\end{equation}
Of course, one easily verifies that the $\Spin(7)$-equation is a subsystem of the Yang--Mills equation.

Under the compactifying coordinates \eqref{eq:change-ABC}, the $\Spin(7)$-instanton and Yang--Mills equations transform to 
\begin{align}
    \label{eq:new-abelian-instaton}
    \xi'=&-\xi\left(\frac{m}{\Delta}BC+\frac{l}{\Delta}CA+\frac{k}{\Delta}AB\right)+\frac{n}{\Delta}(\widehat{m}BC+\widehat{l}CA+\widehat{k}AB)\\
    = & - F\xi + n \widehat{F}\\
\xi''=&-2\left(A+B+C+\frac{m}{\Delta}BC+\frac{l}{\Delta}CA+\frac{k}{\Delta}AB\right)\xi' \nonumber\\
&+m\left(\frac{BC}{\Delta}\right)^2(m\xi-n\widehat{m})+l\left(\frac{CA}{\Delta}\right)^2(l\xi-n\widehat{l})+k\left(\frac{AB}{\Delta}\right)^2(k\xi-n\widehat{k})\label{eq:new-abelian-YM}\, 
\end{align}
where we set 
$$ F = \frac{k AB+ lAC+m BC}{\Delta}  \qquad \text{and} \qquad \widehat{F}= \frac{\widehat{k} AB+ \widehat{l}AC+ \widehat{m} BC}{\Delta}$$
for notational convenience.
\begin{lemma}\label{lemma:fixed-points-abelian}
    The fixed points of the extended system projecting to $p_0^i$ for the instanton are:
    $$ \left(p_0^k, \frac{n\widehat{k}}{k}\right)\,, \qquad \qquad \left(p_0^l, \frac{n\widehat{l}}{l}\right)\qquad \qquad \text{and} \qquad \qquad\left(p_0^m, \frac{n\widehat{m}}{m}\right) \;.$$
    The Yang-Mills system has the same fixed points, with $\xi'=0$.
\end{lemma}
Naturally, these fixed points correspond to the canonical abelian connection on the principal bundle over $\C P^2$. In particular, the conditions of Proposition \ref{prop:admissible-bundles-CP2} translate to 
\begin{lemma}
    The fixed points in Lemma \ref{lemma:fixed-points-abelian} have geometric significance only if $\xi_0$ is an integer multiple of $\frac{\widehat{i}}{i}$.
\end{lemma}
Set $X_0=\xi$ and $X_1=\xi'$. The linearisation at the points $\left(p_0^i,\frac{n\widehat{i}}{i},0\right)$ in $(A,B,C,X_0, X_1)$ is given by
$$
\begin{pmatrix}
    L_0^i & & \\ 
     &0 & 1 \\
     & \frac{16 \Delta^2}{i^2} & 0 
\end{pmatrix}
$$
with $L_0^i$ from Equation \eqref{eq:linearisation-ABC}. The critical points are hyperbolic, with a 3-dimensional space of eigenvalue $\frac{-4\Delta}{i}$. These vectors are tangent to the hypersurface
$$X_1=-F X_0 + n \widehat{F}$$
at the critical point. Hence, any local $\Spin(7)$-Yang--Mills solution is necessarily a $\Spin(7)$-instanton, given by $$\xi_\theta(s) = \frac{n \widehat{i}}{i} +  \kappa \, e^{-\frac{4\Delta}{i}s} + o(e^{-\frac{4\Delta}{i}s})$$
for $\theta \in (0,\theta_*]$ and $\kappa \in \R$. 

In the case $(k,l)=(1,1)$, the space of homogeneous abelian connections is 3-dimensional:
$$A = \left(ne_8 + \xi e_7 +\chi_1 e_1 + \chi_2 e_2\right)\otimes T_1\;,$$
with $\xi, \chi_1, \chi_2 \in \R$. For a family of connections on $N_{1,1}$, their curvature is given by 
\begin{align*}
F_{A_t} =&\;\dot{\xi} \;e_{07}+ \dot{\chi_1} \;e_{01}+ \dot{\chi_2} \;e_{02}+n (e_{34}-e_{56})+\frac{\xi}{3}(e_{34}+ e_{56} -2e_{12})\\
&+ \chi_1(e_{46}-6e_{27}-e_{35}) +\chi_2(6e_{17}+e_{36}+e_{45}) \\
=& \frac{\dot \xi}{f}\omega_{07} + \frac{\dot \chi_1}{a}\omega_{01} + \frac{\dot \chi_2}{a}\omega_{02} -\frac{2\xi}{3a^2}\omega_{12} +\left(\frac{\xi}{3b^2}+\frac{n}{b^2}\right)\omega_{34} +\left(\frac{\xi}{3c^2}-\frac{n}{c^2}\right)\omega_{56}\\
&+\chi_1\left(\frac{\omega_{46}-\omega_{35}}{bc}-6\frac{\omega_{27}}{af}\right)+\chi_2\left(6 \frac{\omega_{17}}{af}+\frac{\omega_{36}+ \omega_{45}}{bc}\right)\;,
\end{align*}
where $\omega_i$ is the orthonormal basis. As before, one computes its Hodge dual
\begin{align*}
* F_{A_t} =& \frac{1}{f}\dot{\xi} \;\omega_{123456}+ \frac{1}{a}\left(\dot{\chi_1} \;\omega_{234567}-  \dot{\chi_2} \;\omega_{134567}\right) -\frac{2\xi}{3a^2}\omega_{034567} +\frac{1}{b^2}\left(\frac{\xi}{3}+n\right)\omega_{012567} + \frac{1}{c^2}\left(\frac{\xi}{3}-n\right)\omega_{012347}\\
&+ \chi_1\left(\frac{\omega_{012467} -\omega_{012357}}{bc}-\frac{6\;\omega_{013456}}{af}\right) +\chi_2\left(\frac{\omega_{012457}+\omega_{012367}}{bc}-\frac{6\;\omega_{023456}}{af}\right) 
\end{align*}
From this expression, one easily computes the $\Spin(7)$-instanton and Yang--Mills equation systems. 
\begin{subequations}\label{eq:abelian-instaton for N11}
    \begin{align}
         \frac{\dot \xi}{f} &=  \frac{\xi}{3} \left(\frac{2}{a^2}-\frac{1}{b^2}-\frac{1}{c^2} \right)-n \left(\frac{1}{b^2}-\frac{1}{c^2} \right)\\
         \dot \chi_i &= \left(\frac{6}{f}-\frac{2a}{bc}\right)\chi_i \label{eq:abelian-instanton-2}\;,
    \end{align}
\end{subequations}
and 
\begin{subequations}\label{eq:abelian-YM for N11}
    \begin{align}
         \partial_t \left( \frac{(abc)^2}{f} \dot \xi\right)  &= \frac{f\xi}{9} \left( \frac{(2bc)^2}{a^2} + \frac{(ac)^2}{b^2} +\frac{(ab)^2}{c^2}  \right) - \frac{n f}{3}\left(\frac{(ab)^2}{c^2} - \frac{(ac)^2}{b^2} \right)\\
        \partial_t \left( f(bc)^2 \;\dot \chi_i\right) &= \left(2a^2f+\frac{36(bc)^2}{f}\right)\chi_i
    \end{align}
\end{subequations}
for $i=1,2$. The $S^1$-symmetry between the equations for $\chi_i$ suggests the change of variables $\chi_1(t)= \chi(t) \cos(\theta(t))$ and $\chi_2(t)= \chi(t) \sin(\theta(t))$. In particular, Equation \eqref{eq:abelian-instanton-2} reduces to 
        $$ \dot \chi  = \left(\frac{6}{f}-\frac{2a}{bc}\right)\chi \qquad \qquad \theta(t)= \theta_0 \in S^1 \;.$$
Under the new coordinates, the $\Spin(7)$-instanton and Yang--Mills equations are
\begin{subequations}
\begin{align}\label{eq:new-abelian-instanton-11}
    \xi'=& - F \xi +n \widehat{F}\\
    \chi'=& (6 - 2 A) \chi
\end{align}
\end{subequations}
\begin{subequations}
\begin{align}\label{eq:new-abelian-YM-11}
\xi''=&-2\left(A+B+C+\frac{AB+AC-2BC}{3}\right)\xi'+\left(\frac{2BC}{3}\right)^2\xi +\left(\frac{CA}{3}\right)^2(\xi+3n )+\left(\frac{AB}{3}\right)^2(\xi-3n)\\
\chi''  =& - \left(4A + \frac{AC }{3}  + \frac{ AB}{3}\right) \chi'+ \left( 2A^2 + 36\right) \chi\label{eq:new-abelian-YM-11b}
\end{align}
\end{subequations}
for a free choice of $\theta_0\in S^1$. 

Note that in this case, there is a symmetry $B \leftrightarrow C$, so we identify the two fixed points $(-6,0,-6)$ and $(-6,-6,0)$. We carry out our discussion under the convention $p_0^1 =(-6,-6,0)$.
\begin{lemma}
    The fixed points of the extended system for the instanton equation are 
    $$ \left(p_0^1, 3n, 0 \right) \qquad \qquad \text{and} \qquad \qquad \left(p_0^{-2},0,0\right)\;.$$
    The Yang-Mills equation has the same fixed points, with $\xi'=0=\chi'$.
\end{lemma}
Let $X_0=\xi$, $X_1=\xi'$, $Y_0=\chi$ and $Y_1=\chi'$.
The linearisations at $\left(p_0^{-2},0,0, 0, 0\right)$ and $(p_0^1,3n,0, 0, 0)$ are 
$$
\begin{pmatrix}
     L_0^{-2}&&&&
     \\
    &0&1&&\\
    &{36}&0&&\\
    &&&0&1\\
    &&&{36}&0\\
\end{pmatrix} \qquad \qquad \text{and} \qquad \qquad
\begin{pmatrix}
     L_0^{1}&&&&
     \\
    &0&1&&\\
    &{144}&0&&\\
    &&&0&1\\
    &&&{108}&12\\
\end{pmatrix}
\;$$
respectively. At $(p_0^{-2},0,0)$, we have a 4-dimensional eigenspace with eigenvalue $\lambda_{-2}=6$, tangent to the instanton hypersurface, with local solutions given by 
\begin{equation} \label{eq:local-form-ab-11-spin-pos}
\xi_\theta(s) =  \kappa_1 \,e^{6s} +o\left(e^{6s}\right) \qquad \qquad \chi_\theta(s) = \kappa_2\, e^{6s} +o\left(e^{6s}\right) 
\end{equation}
with $\kappa_i \in \R$.

For $(p_0^1, 3n, 0)$, we have a 3-dimensional eigenspace with eigenvalue $\lambda=-12$ and a 1-dimensional eigenspace with eigenvalue $\mu=-6$. The former is tangent to the instanton hypersurface, while the latter is not. Explicitly, local solutions to $\Spin(7)$-instantons are given by 
\begin{equation} \label{eq:local-form-ab-11-spin-neg}
    \xi_\theta(s) =  3n + \kappa_1\,e^{-12s} +o\left(e^{-12s}\right) \qquad \qquad \chi_\theta(s) = 0\,,
\end{equation} 
whilst local Yang-Mills solutions are given by 
\begin{equation} \label{eq:local-form-ab-11-spin-neg-ym}
 \xi_\theta(s) = 3n+ \kappa_1 \,e^{-12s} +o\left(e^{-12s}\right) \qquad \qquad \chi_\theta(s) = \kappa_2\, e^{-6s} +o\left(e^{-12s}\right)\,,
\end{equation}
with $\kappa_i \in \R$ as usual. 
\subsection{Long time behaviour}
Let us now discuss long time existence of abelian solutions. Note that since we are mostly working with the time reparametrisation $$t =  \displaystyle \int^s_{-\infty}f d\widetilde{s}\;,$$
we have a time reversal phenomenon for curves contained in $\mathcal{C}^-$. Therefore, we shall focus on the case $(M^m_{k,l}, \gamma_\theta^m)$, to avoid the constant time reversal. The arguments carry over verbatim for $\gamma^k_\theta$ and $\gamma^l_\theta$.

Abelian solutions admit the following integral representation
\begin{proposition}\label{prop:condensed-smooth-abelian-ODE}
    Let $\xi(s)$ be an abelian solution to the $\Spin(7)$-instanton equation \eqref{eq:abelian-instaton} over $M^m_{k,l}$. Then it satisfies the ODE
    \begin{equation}\label{eq:condensed-abelian-instaton}
         \partial_s\left( \frac{\xi- \frac{n \widehat{m}}{m}}{f^2}\right)  =  \frac{2n}{m} \frac{A(C-B)}{f^2}\;.
    \end{equation}
\end{proposition}
\begin{proof}
Using Equation \eqref{eq:new-evo-f}, and substituting the $\Spin(7)$-instanton equation yields
\begin{align*}
    f^2 \partial_s\left( \frac{\xi- \frac{n \widehat{m}}{m}}{f^2}\right) &= \xi' - 2\frac{f'}{f}\left(\xi- \frac{n \widehat{m}}{m} \right)=n\left( \widehat{F}-\frac{\widehat{m}}{m}F\right)\\
    & =\frac{2n}{m}A \left(C-B\right)  \;,
\end{align*}
where the last line follows from Lemma \ref{lemma:identities-klm-khatlhatmhat}.
\end{proof}
Integrating Equation \eqref{eq:condensed-abelian-instaton}, we have
$$\frac{\xi(s)- \frac{n\widehat{m}}{m}}{f^2 (s)} - \frac{\xi(s_0)-\frac{n \widehat{m}}{m}}{f^2(s_0)} =  \int_{s_0} ^{s} \frac{2n}{m} \frac{A \left(C-B\right) }{f^2} d\tilde{s}  \;, $$
and taking the limit $s_0 \to -\infty$, we get the closed expression for $\xi(s)$:
\begin{equation}\label{eq:integral_xi}    
\xi(s) = \frac{n\widehat{m}}{m} + f^2(s) \left( \int_{-\infty} ^{s} \frac{2n}{m} \frac{A \left(C-B\right) }{f^2} d\tilde{s} -  \kappa\right)\;,
\end{equation}
where $\kappa \in \R$ corresponds to the choice of initial condition at $-\infty$; $\kappa=\lim\limits_{s_0\to-\infty} \frac{\frac{n \widehat{m}}{m}-\xi(s_0)}{f^2(s_0)}$. Note that the limit indeed exists, since $\xi-\frac{n\widehat{m}}{m} \sim \exp\left(\frac{-4\Delta}{m}s\right)$ and $f \sim \exp\left( \frac{-2\Delta}{m}s\right)$.

It is now convenient to distinguish between the ALC and AC regimes. For the ALC regime, we have the asymptotics $f \sim l_\theta$, the fibre length at infinity, for some $l \in (0,\infty)$ as $s \to \infty$ and $A, B, C \sim s^{-1}$. In particular, we have 
\begin{proposition}
Let $\gamma_\theta^i (s)= (A(s), B(s), C(s), f(s))$ be an ALC $\Spin(7)$-metric on $M^i_{k,l}$ for $\theta \in (0, \theta_*)$, with fibre length at infinity $\ell_\theta$.  We have a one-to-one correspondence between the initial condition $\kappa \in \R$ and the critical line $(0, \ell_\theta, \xi_\infty)$.
\end{proposition}
\begin{proof}
It suffices to compute the limit of $\xi(s)$ as $s \to \infty$. Using the integral expression for $\xi$, Equation \eqref{eq:integral_xi}, we have
$$ \lim _{s \to \infty}  \xi (s)  =  \frac{n\widehat{m}}{m} +\ell_\theta^2\left( \int_{- \infty} ^ \infty \frac{2n}{m} \frac{A(C-B)}{f^2} d \widetilde{s}  - \kappa \right) \;.$$
The claim follows since the improper integral is finite, as the integrand decays like $s^{-2}$ at infinity. The cases for $k$ and $l$ follow analogously. 
\end{proof}
\begin{remark}\label{rmk:monopole}
We can compare our solution with model solutions in the limiting space, the $S^1$-bundle over the $G_2$-cone $(1, \infty) \times \SU(3)/ T^2$. It then follows that, if  $\xi(s)$ converges to $ \xi_\infty$ as $s\to \infty$, the connection $A(s)= \xi(s) e_7 + ne_8$ converges pointwise to the pullback of a $G_2$-monopole on the cone $(1, \infty) \times F_{1,2}$.
\end{remark}
Let us now consider the AC case. First, we have
\begin{lemma}
Let $p^\pm_{k,l}$ be the fixed points of the $(A,B,C)$-system corresponding to the AC ends. Then, the extended system has fixed points $\left(p_{AC}^\pm, n \Lambda ^\pm\right)$, where 
 $$\Lambda ^{\pm}=\frac{ \widehat{F}^\pm}{F^\pm} =\frac{\widehat{m}B^\pm C^\pm+\widehat{l}A^\pm C^\pm+\widehat{k}A^\pm B^\pm}{m B^\pm C^\pm + l A^\pm C^\pm + kA^\pm B^\pm}$$ is an intrinsic quantity of the nearly parallel $G_2$ metric $p_{AC}^\pm= \left( A^\pm, B^\pm, C^\pm\right)$. 

 The linearisation of the ODE system restricted to the $\xi$-coordinate is positive for $p_{k,l}^+$ and is negative for $p_{k,l}^-$.
\end{lemma}
\begin{proof}
The fixed-point computation is straightforward. The linearisation at the fixed point in the $\xi$ direction is $-F^\pm$. The claim follows by noting that $F^+<0$ and $F^->0$.
\end{proof}
For completeness, we show 
\begin{proposition}
\label{prop: positive numerator}
The quantity $\widehat{F}^\pm$ is positive for $k>l$ and vanishes for $(k,l)=(1,1)$.
\end{proposition}
\begin{proof}
    By Proposition \ref{prop:AC-inequalitites}, the critical points $p_{AC}^{\pm}$ come from the two intersection points $(\mathsf{b}^\pm,\mathsf{c}^\pm)$ between an ellipse $\mathsf{L}_1=0$ and a hyperbola $\mathsf{L}_2=0$ in \eqref{eq:homogenised_AC_equation}.
Since $\mathsf{L}_2(\frac{3}{2},\frac{3}{2})=0$ and $\mathsf{L}_1(\frac{3}{2},\frac{3}{2})<0$, the intersection points between $\mathsf{L}_1=0$ and $\mathsf{L}_2=0$ are on the branch of $\mathsf{L}_2=0$ that passes through $(\frac{3}{2},\frac{3}{2})$. Consequently, the inequality $\mathsf{T}_1(\mathsf{b}^\pm,\mathsf{c}^\pm)>0$ is satisfied, where
$$\mathsf{T}_1:=(15k+9l)\left(\mathsf{b}-\frac{3}{2}\right)-(9k+15l)\left(\mathsf{c}-\frac{3}{2}\right)=0$$
is tangent to $\mathsf{L}_2=0$ at $(\frac{3}{2},\frac{3}{2})$. 

On the other hand, we have $\Delta \widehat{F}= \widehat{m}BC+\widehat{l}AC+\widehat{k}AB=A^2\mathsf{L}_3$, where
$$\mathsf{L}_3:=(k-l)\mathsf{b}\mathsf{c}-(2k+l)\mathsf{c}+(k+2l)\mathsf{b}.$$
By Proposition \ref{prop:AC-inequalitites}, we know that  $\mathsf{b}^-,\mathsf{c}^-<1$ and $\mathsf{b}^+,\mathsf{c}^+>2$. Thus, we have $(k-l)(\mathsf{b}^\pm-1)(\mathsf{c}^\pm-1)>0$. Consequently, the inequality 
$$\mathsf{L}_3(\mathsf{b}^\pm,\mathsf{c}^\pm)>\mathsf{T}_2(\mathsf{b}^\pm,\mathsf{c}^\pm)$$
holds, where
$$\mathsf{T}_2:=(2k+l)(\mathsf{b}-1)-(2l+k)(\mathsf{c}-1).$$
Since $8\mathsf{T}_2=\mathsf{T}_1+(k-l)(\mathsf{b}+\mathsf{c}+1)$ for any $\mathsf{b},\mathsf{c}>0$, we conclude that $\mathsf{L}_3(\mathsf{b}^\pm,\mathsf{c}^\pm)>0.$
\end{proof}
\begin{proposition}\label{prop:unique-ab-AC}
    Let $\gamma_{\theta_*}^i (s)= (A(s), B(s), C(s), f(s))$ be the AC  $\Spin(7)$-metric on $M^i_{k,l}$. Then there exists a unique initial condition $\kappa_* \in \R$ such that $\xi_*(s)$ converges to $(p^\pm_{AC}, n \Lambda^\pm)$.
\end{proposition}
\begin{proof}
    Using the integral representation of $\xi$, Equation \eqref{eq:integral_xi}, it is clear that, for any $s\in \R$, there is a bijection $\kappa \rightarrow \xi_\kappa(s)$. The claim now follows since $(p_{AC}^\pm , n \Lambda^\pm)$ is a hyperbolic point, and $(0,...,0,1)$ is an eigenvector with positive eigenvalue (up to the sign of $f$). So there exists a unique solution.
\end{proof}
Alternatively, one can try to visualise this using the integral representation once more. Since $A,B, C$ have a finite limit and $f$ grows exponentially, the improper integral is finite. The previous discussion implies we have the asymptotics
$$ \int_{-\infty}^s \frac{2A(C-B)}{m f^2} d\widetilde{s} =  I_\infty  +  \Lambda^+\, e^{2F^+s}  + o \left( e^{2F^+s}\right)\;,$$
and we are interested in the initial condition $\kappa_* =  I_\infty.$

Let us now consider the exceptional case where the principal orbit is $N_{1,1}$. Since our discussion for $\xi$ carries over verbatim, we are only interested in understanding the possible solutions $\chi$.
\begin{proposition}\label{prop:vanishing-chi-abelian}
Any bounded $\Spin(7)$-instanton solution on $M^i_{1,1}$ has $\chi \equiv 0$.
\end{proposition}
\begin{proof}
From the short time existence discussion in Equations \eqref{eq:local-form-ab-11-spin-neg} and \eqref{eq:local-form-ab-11-spin-pos}, we have $\chi=0$ for all $\gamma^1_\theta$. For $\gamma^{-2}_\theta$, we have the explicit integral solution 
$$ \chi(s) =  \kappa \exp \left(\int_{-\infty} ^s 6 -2A \;d\widetilde{s}\right)\;.$$
In the limit $s \to \infty $, the integrand remains positive both in the ALC and AC regimes, so the only bounded solution that exists for $\chi$ is $\chi=0$ as well. 
\end{proof}
It remains to understand the exceptional Yang-Mills solution in the case $\gamma^1_\theta$. Using the variable $s$ given by $dt = f ds$, the Yang-Mills ODE \eqref{eq:new-abelian-YM-11b} can be written as 
$$ \partial_s \big( (bc)^2  \partial_s \chi \big)  = (2A^2 + 36) \left[(bc)^2 \chi \right]  \;.$$
This suggests the change of variables $u = bc \chi$. A short computation yields the ODE for $u$ is
$$ u'' =  \left( \frac{ (bc)''}{bc} +2A^2 +36 \right) u\;,$$
where $(\;\cdot\;) '$ denotes differentiation with respect to $s$. Let us compute $(bc)''$. We have 
$$\partial_s \left(bc\right) = f\partial_t(bc) = fbc \left( \frac{\dot b}{b} + \frac{\dot c}{c}\right) = bc\left( 2A +\frac{AC+AB}{6}\right).$$
Hence, 
\begin{align*}
    \frac{ (bc)''}{bc} & = \left( 2A +\frac{AC+AB}{6}\right)' + \left( 2A +\frac{AC+AB}{6}\right)^2\\
    &= -\frac{1}{12}A^2(B^2+C^2)- \frac{A^2BC}{18}+\frac{ABC}{9}(B+C+6)+2AB+2AC-\frac{A}{3}(B^2+C^2)-\frac{A^2}{3}(B+C)-2A^2\\
\end{align*}
So $u$ solves $u'' = G u$, with 
 $$ G(s)  = -\frac{1}{12}A^2(B^2+C^2)- \frac{A^2BC}{18}+\frac{ABC}{9}(B+C+6)+2AB+2AC-\frac{A}{3}(B^2+C^2)-\frac{A^2}{3}(B+C) +36  \;.$$
 Now, \cite[Lemma 4.15]{Chi22} implies that the inequality $0\geq B,C\geq A\geq -6$ holds along $\gamma_\theta^{1}$. A direct algebraic calculation shows that
 \begin{proposition}
The inequality $G \geq 36$  holds along $\gamma_\theta^1$, with equality only if $B=C=0$ or $A=0$.
\end{proposition}
Hence, $u(s)$ has a one-parameter family of solutions extending smoothly over the singular orbit. We have
\begin{lemma}
    In the limit $s \to -\infty$, we have $\abs{u} \to \infty$.
\end{lemma}
\begin{proof}
    Let $u $ be a non-trivial solution and consider $v=u^2$. Then $v'' =  2 (u')^2  + 2 G u^2 >0$, since $u$ and $u'$ cannot vanish simultaneously for a non-trivial solution, and $G>0$. Hence $v$ is strictly convex. The convexity bound  gives
    $v(x) \geq v(a) + v'(a)(x-a)$ for any fixed $a$. Since $u$ extends smoothly over the singular orbit at $s=+\infty$, we have $v'< 0$, and the claim follows.
\end{proof}
\begin{proposition}
The only bounded solution to Equation \eqref{eq:new-abelian-YM-11b} for $\gamma^1_\theta$ with $\theta\in (0, \theta^*]$ is $\chi=0$.
\end{proposition}
\begin{proof}
  We distinguish between the ALC and AC regimes. For an ALC metric, we have the asymptotics $A,B, C \sim s^{-1}$, so $G \sim 36$. Therefore $u \sim e^{-6s}$ and $\chi \sim s^{-2}\,e^{-6s}$ as $s\to -\infty$. For the AC metric, we have $p_{1,1}^- = (-6, -3,-3)$, so $G \sim 144$ and $u \sim e^{-12s}$ as $s \to -\infty$. Since $F \sim 6$, and $f$ satisfies $f'= -\frac{F}{2}f$, we have $f\sim e^{-3s}$ and therefore $b,c \sim e^{-3s}$ and $\chi \sim e^{-6s}$. Hence, in both cases, the only bounded solution is $\chi=0$. 
\end{proof}

\section{\texorpdfstring{$\SO(3)$}{SO(3)}-connections -- Local behaviour}
We begin with the generic case $kl\widehat{m} \neq 0$, and carry out the computations for $n=\widehat{m}$. The claims for $n= \widehat{k}, \widehat{l}$ can be deduced easily. For the $\so(3)$ basis $\{T_i\}$, we take the standard convention $[T_i, T_j] =  \varepsilon_{ijk} T_k$.

Using Proposition \ref{prop:homogeneous connections SO3-kl}, we want to consider
$$A(t) = \left( \xi (t) e_7 + \widehat{m}e_8\right) \otimes T_1 + \psi(t)\left( e_1 \otimes T_2 + e_2 \otimes T_3\right)\;.$$
A short computation yields the curvature
\begin{align}
    F_A =& \left( \dot{\xi} e_{07} + \frac{1}{\Delta}\left[ \left(m \xi - \widehat{m}^2 + \Delta\psi^2 \right) e_{12} + \left(l \xi - \widehat{l} \;\widehat{m}\right) e_{34} +\left(k \xi - \widehat{k}\,\widehat{m}\right) e_{56}\right] \right) \otimes T_1 \label{eq:curvature-generic-SO3}\\
    &+ \left(\dot{\psi}e_{01} + \psi \left[e_{46}-e_{35}+(3m+\xi)e_{27}\right]\right)\otimes T_2 + \left(\dot{\psi}e_{02} + \psi \left[e_{45}+e_{36}-(3m+\xi)e_{17}\right] \right)\otimes T_3 \nonumber
\end{align}
where we are taking the convention $[T_i, T_j] =  \varepsilon_{ijk} T_k$ for the $\SO(3)$-basis. 

A tedious computation yields the equations for the $\Spin(7)$-instanton and Yang-Mills conditions:
\begin{proposition}
On the principal bundle $Q_{\widehat{m}}\to N_{k,l}$, the $\Spin(7)$-instanton equations are
\begin{subequations}
    \begin{align}
            \frac{\dot\xi}{f} &= - \frac{\xi}{\Delta}\left(\frac{{m}}{a^2}+\frac{{l}}{b^2}+\frac{{k}}{c^2} \right)+\frac{\widehat{m}}{\Delta} \left(\frac{\widehat{m}}{a^2}+\frac{\widehat{l}}{b^2}+\frac{\widehat{k}}{c^2}\right)  - \frac{\psi^2}{a^2}\\
            \dot \psi &= - \left( \frac{2a}{bc} + \frac{3m+ \xi}{f}\right) \psi
    \end{align}
\end{subequations}
Similarly, the Yang--Mills equations are
\begin{subequations}
    \begin{align}\label{eq:generic-SO3-YM}
\partial_t \left( \frac{(abc)^2}{f} \dot{\xi} \right) =& mf\left(\frac{bc}{a \Delta}\right)^2 \left( m\; \xi -\widehat{m}^2 +\Delta\psi^2\right)+ lf\left(\frac{ac}{b \Delta }\right)^2\left( l\; \xi -\widehat{m} \;\widehat{l}\right) \nonumber \\ &+kf\left(\frac{ab}{c \Delta}\right)^2\left( k\; \xi -\;\widehat{k} \,\widehat{m}\right)  + 2\frac{(bc)^2}{f}(\xi+3m) \psi^2 \\
\partial_t \left( (bc)^2 f \dot{\psi}\right) =& 2 a^2f \psi + \frac{(bc)^2}{f}(\xi+3m)^2\psi + \left(\frac{bc}{a}\right)^2f\frac{\psi\left(m \xi - \widehat{m}^2+\Delta \psi^2 \right)}{\Delta}  
    \end{align}
\end{subequations}
\end{proposition}
In the compactifying coordinates $(A,B,C)$, the $\Spin(7)$-instanton equations translate to 
\begin{subequations}\label{eq:new-generic-SO3-instanton}
    \begin{align}
            \xi'&=- F \xi +  \widehat{m} \widehat{F}- BC  \psi^2\\\
    \psi'&=-(2A+3m+\xi)\psi
    \end{align}
\end{subequations}
Similarly, the Yang--Mills equations are
\begin{subequations}\label{eq:new-generic-SO3-YM}
    \begin{align}
\xi''=&-2\left(A+B+C+F\right)\xi' +m\left(\frac{BC}{\Delta}\right)^2(m\xi-\widehat{m}^2+ \Delta \psi^2)\\
&+l\left(\frac{CA}{\Delta}\right)^2(l\xi-\widehat{l}\,\widehat{m})+k\left(\frac{AB}{\Delta}\right)^2(k\xi-\widehat{k}\, \widehat{m})+2BC(\xi+3m)\psi^2\nonumber\\
\psi''=&-\left(4A+\frac{k}{\Delta}AB+\frac{l}{\Delta}CA\right)\psi'+2A^2\psi+(\xi+3m)^2\psi+BC\psi\frac{(m\xi-\widehat{m}^2+\Delta \psi^2)}{\Delta}
    \end{align}
\end{subequations}
As before, we compute the fixed points
\begin{lemma}\label{lemma:fixed-pts generic SO3}
    The fixed points of the extended instanton system over the fixed point $p_0^m$ are:
    $$ \left(p_0^m, \frac{\widehat{m}^2}{m}, 0 \right) \qquad \text{and} \qquad \left(p_0^m, -3m,  2  \right) \;.$$
    The extended Yang-Mills system has an additional fixed point:
    $$ \left(p_0^m,  -3m +\frac{2\Delta}{m},  1  \right) \;.$$
\end{lemma}
For the instanton system, the first fixed point corresponds to the abelian reduction, so it only makes geometric sense whenever $\widehat{m}= 0 \mod m $ (or more generally $\widehat{i}= 0 \mod i $). The fixed points with $\psi_0 \neq 0$ correspond to the case where the $\SO(3)$-bundle is irreducible over the singular orbit, $Q_{\widehat{m}}$. 

Set $G_0=\psi$ and $G_1=\psi'$, and consider the $(A,B,C,X_0,X_1,G_0,G_1)$ space. The fixed points $\left(p_0^m, \frac{\widehat{m}^2}{m},0,0, 0\right)$ and $\left( p_0^m,-3m ,0, 2, 0 \right)$ have linearisations
$$
\begin{pmatrix}
    L_0^{m}&&& \\
    &0&1&0&0\\
    &\frac{16\Delta^2}{m^2}&0&0&0\\
    &0&0&0&1\\
    &0&0&\frac{16\Delta^2}{m^2}&0
\end{pmatrix} \qquad \qquad \text{and} \qquad \qquad
\begin{pmatrix}
     L_0^{m}&&&&
     \\
    &0&1&0&0\\
    &\frac{48\Delta^2}{m^2} &0& \frac{64\Delta^3}{m^3}&0\\
    &0&0&0&1\\
    &\frac{8\Delta}{m} &0& \frac{32\Delta^2}{m^2}&0
\end{pmatrix}
$$
respectively. At $\left(p_0^m, \frac{\widehat{m}^2}{m},0,0, 0\right)$ we have a 4-dimensional family with eigenvalue $-\frac{4\Delta}{m}$, tangent to the instanton hypersurface; and with all other eigenvalues being negative. Therefore, for each $\theta$, we have a 2-dimensional family of integral curves:
\begin{equation}\label{eq:local-Gamma a,b}
    \Gamma_\theta^i(a,\alpha) =  p_0^i + \big( \cos(\theta)v_1+\sin(\theta)v_2 + a \left( \cos \alpha \,v_3  +\sin \alpha \,v_4 \right)\big)e^{-\frac{4\Delta}{m}s} +  o\left( e^{-\frac{4\Delta}{m}s}\right) \;, 
\end{equation}
where $v_i$ are the obvious extensions of $u_i$ from Equation \eqref{eq:eigenvectors-spin7}, and we have chosen $v_3$ and $v_4$ so that the coordinates of the integral curve are
\begin{equation}
        \xi_{a, \alpha} (s) =  \frac{\widehat{m}^2}{m}+  a \,\cos \alpha  e^{-\frac{4\Delta}{m}s} +  o\left( e^{-\frac{4\Delta}{m}s}\right) \qquad  \text{and} \qquad   \psi_{a, \alpha} (s) =  a \,\sin \alpha  e^{-\frac{4\Delta}{m}s} +  o\left( e^{-\frac{4\Delta}{m}s}\right)\;.
\end{equation}
For $\left( p_0^m,-3m ,0,  2, 0 \right)$, there are two positive eigenvalues: $-\frac{4\Delta}{m}$ with multiplicity three and $-\frac{8\Delta}{m}$ with multiplicity one. The corresponding eigenvectors are
$$
-\frac{4\Delta}{m}: w_1=\begin{pmatrix}
    \\
    u_1\\
    \\
    0\\
    0\\
    0\\
    0\\
\end{pmatrix},\quad 
w_2=\begin{pmatrix}
    \\
    u_2\\
    \\
    6m\Delta\\
    -24\Delta^2\\
    -3m^2\\
    12m\Delta
\end{pmatrix},\quad 
w_3=\begin{pmatrix}
    0\\
    0\\
    0\\
    2m\Delta\\
    -8\Delta^2\\
    -m^2\\
    4m\Delta
\end{pmatrix}, \quad
-\frac{8\Delta}{m}: w_4=\begin{pmatrix}
    0\\
    0\\
    0\\
    -4m\Delta\\
    32\Delta^2\\
    -m^2\\
    8m\Delta
\end{pmatrix}.
$$
where now the eigenvector $w_3$ is not tangent to the instanton hypersurface.  Hence, for each $\theta\in (0, \theta^*]$, we have a one-parameter family of integral curves corresponding to local $\Spin(7)$ solutions:
\begin{equation} \label{eq:local-Omega c}
\Omega_\theta^m(c)\sim p_0^m+\cos(\theta)w_1e^{-\frac{4\Delta}{m}s}+\sin(\theta)w_2e^{-\frac{4\Delta}{m}s}+c w_4 e^{-\frac{8\Delta}{m}s},
\end{equation}
with coordinate asymptotics 
\begin{subequations}
    \begin{align}
        \xi_{c} (s) &=  -3m  +6 m \Delta\sin \theta\, e^{-\frac{4\Delta}{m}s}  - 4c\,m \Delta  \, e^{-\frac{8\Delta}{m}s} +  o\left( e^{-\frac{8\Delta}{m}s}\right) \\
        \psi_c (s) &= 2 -3 m^2  \sin \theta\, e^{-\frac{4\Delta}{m}s}  - c \,m^2   \, e^{-\frac{8\Delta}{m}s} +  o\left( e^{-\frac{8\Delta}{m}s}\right)\;.
    \end{align}
\end{subequations}
\subsection{The exceptional cases}
Let us now study local solutions to the cohomogeneity one $\Spin(7)$-instanton equation. We begin with the case with principal orbit $N_{1,0}$. By Proposition \ref{prop:homogeneous connections SO3-10}, the space of cohomogeneity one connections on the principal part is parametrised up to gauge by 
$$A(t) = (\xi e_7+ e_8 ) \otimes T_1  +\psi_1\left[e_5 \otimes  \big(\cos(\theta) T_2  - \sin(\theta) T_3 \big) +e_6 \otimes  \big(\sin(\theta) T_2  + \cos(\theta) T_3 \big) \right]  +\psi_2 ( e_1 \otimes T_2 +  e_2\otimes T_3)$$
where $\xi, \psi_1, \psi_2$ and $\theta$ depend on $t$. The curvature of $A(t)$ is 
\begin{align*}
    F_A =& \left( \dot{\xi} e_{07} +   \left( \psi_2^2- \xi - 1\right) e_{12} +2 e_{34} +\left(\xi - 1 +\psi_1^2\right) e_{56} + \psi_1\psi_2 \big[\cos(\theta) \left( e_{16}-e_{25}\right)  - \sin(\theta) \left( e_{26}+e_{15}\right)\big]\right) \otimes T_1 \\
    &+ \left[ \left( \dot{\psi_1} \cos (\theta) - \psi_1 \sin(\theta) \dot{\theta}\right)e_{05}  + \left( \dot{\psi_1} \sin (\theta)+ \psi_1 \cos(\theta) \dot{\theta}\right)e_{06} + \dot{\psi_2}e_{01} + \psi_1\cos(\theta) \left[e_{24}-e_{13}+(\xi+3)e_{67}\right]\right.\\&  +\psi_1\sin(\theta)\left[e_{14}+e_{23}-(\xi+3)e_{57}\right] + \psi_2\left[e_{46} -e_{35}+ (\xi-3)e_{27}\right] \Big]\otimes T_2 \\&+ \left[ \left( \dot{\psi_1} \cos (\theta) - \psi_1 \sin(\theta) \dot{\theta}\right)e_{06}  - \left( \dot{\psi_1} \sin (\theta)+ \psi_1 \cos(\theta) \dot{\theta}\right)e_{05} + \dot{\psi_2}e_{02} + \psi_1\cos(\theta) \left[e_{14}+e_{23}-(\xi+3)e_{57}\right]\right.\\&  -\psi_1\sin(\theta)\left[e_{24}-e_{13}+(\xi+3)e_{67}\right] + \psi_2\left[e_{45} +e_{36}- (\xi-3)e_{17}\right] \Big]\otimes T_3\;.
\end{align*}
\begin{proposition}
On the principal bundle $\theta_{1}\to N_{1,0}$, the $\Spin(7)$-instanton equations are
\begin{subequations}\label{eq:SO3-instanton-N10}
    \begin{align}
            \partial_t\left(\frac{\xi}{f^2}\right) &=\frac{1}{f} \left(\frac{1- \psi^2_2}{a^2}-\frac{2}{b^2}+\frac{1- \psi_1^2}{c^2}\right)\\
            \partial_t\log(\psi_1 ) &= - \left( \frac{2c}{ab} + \frac{\xi+3}{f}\right)\\
            \partial_t\log(\psi_2) &= - \left( \frac{2a}{bc} + \frac{\xi-3}{f}\right)\\
             \dot{\theta}&=0
    \end{align}
\end{subequations}
\end{proposition}
In the $(A,B,C)$ coordinates, the instanton equations translate to
\begin{subequations}
    \begin{align}
        \xi' &= (BC-AB) \xi + AB(1-\psi^2_1) - 2AC + BC(1- \psi_2^2)\\
        \psi_1' &= -(2C+3+\xi)\psi_1\\
        \psi_2' & = -(2A -3 +\xi) \psi_2
    \end{align}
\end{subequations}
Note that this case is exceptional since there exists an involution $(A,B,C) \to (-C, -B, -A)$ of the $\Spin(7)$-metric ODE system \eqref{eq:new-evo-spin7}, which maps $p_0^1\in \mathcal{C}^-$ to $p_0^{-1} \in \mathcal{C}^+$, which extends to the $\Spin(7)$ ODE as 
 $$(s, A,B,C,\xi, \psi_1, \psi_2) \mapsto (-s, -C, -B, -A, -\xi, \psi_2, \psi_1)\;,$$
which explains its form. We have 
\begin{lemma}
    The fixed points of the extended instanton system over $p_0^i$ are
    \begin{gather*}
        \left(p_0^1, 1, 0, 0\right) \qquad \qquad \left(p_0^1, -3,  2 ,0  \right) \qquad \qquad \left(p_0^{-1}, -1, 0, 0\right) \qquad \qquad \left(p_0^{-1}, 3,  0 ,2 \right) \qquad \qquad\;.
    \end{gather*}
\end{lemma}
By computing the linearisations of the ODE around the fixed points, we have
\begin{proposition}
At each of the stationary points
\[ (p_0^1,1,0,0) \qquad\text{and}\qquad (p_0^1,-3,2,0), \]
the eigenvectors corresponding to negative eigenvalues of the linearised system are tangent to the hyperplane $\{\psi_2=0\}$.

Similarly, at each of the stationary points 
\[ (p_0^{-1},-1,0,0) \qquad\text{and}\qquad (p_0^{-1},3,0,2), \]
the eigenvectors corresponding to positive eigenvalues of the linearised system are tangent to the hyperplane $\{\psi_1=0\}$.
\end{proposition}
As a direct corollary, we have
\begin{corollary}
A smooth instanton solution over $Q_1 \rightarrow M_{1,0}^1$ or $P_1 \rightarrow M^1_{1,0}$ is of the form $(\xi, \psi_1,0)$. Similarly, a smooth instanton solution on $Q_{-1}$ or $P_{-1}$ over  $M_{1,0}^{-1}$ is of the form $(\xi, 0,\psi_2)$.  In particular, the solutions behave as in the generic case discussed above.
\end{corollary}
Let us move on to the case corresponding to $\theta_3\in \Map(\U(1), \SO(3))$ over $N_{1,1}$.  The space of cohomogeneity one connections is parametrised by
\begin{align*}
    A(t) =& (\chi_1e_1 + \chi_2 e_2+\xi e_7+ 3e_8) \otimes T_1  +\psi_1\left[e_5 \otimes  \big(\cos(\theta) T_2  - \sin(\theta) T_3 \big) +e_6 \otimes  \big(\sin(\theta) T_2  + \cos(\theta) T_3 \big) \right]  \\&+\psi_2 ( e_4\otimes T_2 +  e_3\otimes T_3)
\end{align*}
 The curvature of $A(t)$ is computed as usual. Introducing the change of coordinates $\chi = \chi_1 \cos( \theta) - \chi_2 \sin( \theta)$, we have
\begin{lemma}
On the principal bundle corresponding to $\theta_{3}$ over $N_{1,1}$, the $\Spin(7)$-instanton equations are
\begin{subequations} \label{eq:SO3exceptionalN11-3}
 \begin{align}
    \frac{\dot{\xi}}{f}& =  \left( \frac{2}{3a^2} - \frac{1}{3b^2} -\frac{1}{3c^2}\right)\xi +\left(\frac{3- \psi_1^2}{c^2}-\frac{3- \psi^2_2}{b^2}\right)\\
    \dot{\chi}&= \left(\frac{6}{f} -  \frac{2a}{bc} \right) \chi + \frac{a}{bc}\psi_1\psi_2\\
    \dot{\psi_1} &= - \left(\frac{2c}{ab}+\frac{\xi +3}{f}\right) \psi_1 - \frac{c}{ab} \chi \psi_2\\
    \dot{\psi_2} &= - \left(\frac{2b}{ac}+\frac{3-\xi }{f}\right) \psi_2 - \frac{b}{ac} \chi \psi_1
\end{align}
\end{subequations}
as long as $\psi_1$ and $\psi_2$ don't vanish simultaneously. 
\end{lemma}
Note that the requirement $(\psi_1, \psi_2) (s)\neq (0,0)$ is not a limiting assumption, since
\begin{lemma}
    If there exists $T\in \R $ such that $(\psi_1(T),\psi_2(T))=(0,0)$, then we have $(\psi_1, \psi_2)=(0,0)$ for all time, and the ODE system reduces to the abelian instanton system.
\end{lemma}
Using the $(A,B,C)$ coordinates, the ODE system becomes 
\begin{subequations}\label{eq:exceptional-SO3-11-compactified}
\begin{align}
    \xi'& =  \left( 2BC -AC-AB\right)\frac{\xi}{3} +AB(3- \psi_1^2)+AC(\psi^2_2-3)\\
    \chi' &= \left(6-2A \right) \chi + A \psi_1\psi_2\\
    \psi_1' &= - \left(2C+\xi +3\right) \psi_1 - C\chi \psi_2\\
    \psi_2' &= - \left(2B+3-\xi\right) \psi_2 - B \chi \psi_1
\end{align}
\end{subequations}
Due to the symmetry $B \leftrightarrow C$, we focus on $p_0^1=(-6,-6,0)$. As usual, we have
\begin{lemma}
    The fixed points of the extended system \eqref{eq:exceptional-SO3-11-compactified} over $p_0^1$ are
    $$ \left(p_0^1, 9, 0,0, 0\right) \qquad \text{and}\qquad \left(p_0^1, -3, 0,  2 ,0  \right)\;. $$
    The corresponding linearisations are 
$$\begin{pmatrix}
    L_0^1&&&\\
    &-12&&&\\
    &&18&&\\
    &&&-12&\\
    &&&&18
\end{pmatrix} \qquad  \text{and} \qquad \begin{pmatrix}
    &&&0&0&0&0\\
    &L_0^1&&0&0&0&0\\
    &&&0&0&0&0\\
    0&0&24&-12&0& -144 & 0\\
    0&0&0&0&18&0&-12\\
    0&0&-4&-2&0 &0&0\\
    0&0&0&0&12&0&6
\end{pmatrix}\;.
$$
\end{lemma}
A short computation shows that the eigenspaces of negative eigenvalue are the same as in the generic case. Hence, we have
\begin{proposition}
    A smooth instanton solution on $Q_3$ or $P_3$ over $M_{1,1}^1$ is of the form $(\xi, 0, \psi_1, 0)$.
\end{proposition}
 Note that, unlike the preceding cases, the second matrix has non-real eigenvalues. However, their real part is strictly positive, so they do not provide us with smooth solutions.

Finally, we consider the trivial case $\theta_0 \to N_{1,1}$. A cohomogeneity one connection in temporal gauge takes the form $$A(t) = \sum _{i=1}^3 \left( \chi_i e_1 + \zeta_i e_2 + \xi_i e_7\right) \otimes T_i\;.$$
Its curvature is 
\begin{align*}
    F_A=& \sum_{i=1}^3  \left[\dot{\chi_i}\; e_{01} + \dot{\zeta_i}  \;e_{02}+ \dot{\xi_i}\; e_{07}  + 
    \frac{\xi_i}{3}(e_{34}+ e_{56} -2e_{12})+ \chi_i(e_{46}-6e_{27}-e_{35}) +\zeta_i(6e_{17}+e_{36}+e_{45}) \right. \\& + \left. \sum_{j,k=1}^3 \epsilon_{ijk}\left(\chi_j\zeta_k \;e_{12} + \chi_j\xi_k \;e_{17} + \zeta_j\xi_k \;e_{27}\right)   \right] \otimes T_i
\end{align*}
From this, straightforward computations yield the $\Spin(7)$-instanton equation:
\begin{subequations}\label{eq:flat SO(3)-instaton N11}
\begin{align}
    \frac{\dot \xi_i}{f} &=  \left(\frac{2}{a^2}-\frac{1}{b^2}-\frac{1}{c^2} \right)\frac{\xi_i}{3}- \sum_{j,k=1}^3 \epsilon_{ijk} \frac{\chi_j \zeta_k}{a^2}\\
    \dot \chi_i &= \left(\frac{6}{f}-\frac{2a}{bc}\right)\chi_i - \sum_{j,k=1}^3 \epsilon_{ijk}\frac{\zeta_j\xi_k}{f}\\
    \dot \zeta_i &= \left(\frac{6}{f}-\frac{2a}{bc}\right)\zeta_i +\sum_{j,k=1}^3 \epsilon_{ijk} \frac{\chi_j \xi_k
    }{f}\;.
    \end{align}
\end{subequations}
The previous system still has full gauge symmetry. The reader will readily verify we can break it by imposing $\xi_2 = 0 = \xi_3$, so the $\Spin(7)$-instanton ODE system \eqref{eq:flat SO(3)-instaton N11} simplifies to 
\begin{subequations}
    \begin{align}
    \frac{\dot{\xi_1}}{f}&= \left(\frac{2}{a^2} - \frac{1}{b^2} - \frac{1}{c^2}\right) \frac{\xi_1}{3} + \frac{ \chi_3 \zeta_2- \chi_2 \zeta_3}{a^2}
    \end{align}
    \begin{align}
    \dot{\chi_1} &= \left(\frac{6}{f}- \frac{2a}{bc}\right)\chi_1 \qquad \qquad &\dot{\zeta_1} = \left(\frac{6}{f}- \frac{2a}{bc}\right) \zeta_1&\\
    \dot{\chi_2}& = \left(\frac{6}{f}- \frac{2a}{bc}\right)\chi_2 - \frac{\xi_1\zeta_3}{f} \qquad \qquad
    &\dot{\zeta_2} = \left(\frac{6}{f}- \frac{2a}{bc}\right) \zeta_2+ \frac{\xi_1\chi_3}{f}&\\
    \dot{\chi_3} &= \left(\frac{6}{f}- \frac{2a}{bc}\right)\chi_3 + \frac{\xi_1\zeta_2}{f} \qquad \qquad
    &\dot{\zeta_3} = \left(\frac{6}{f}- \frac{2a}{bc}\right) \zeta_3- \frac{\xi_1\chi_2}{f}&
    \end{align}
\end{subequations}
together with the conserved quantities $F_1= \chi_1 \zeta_3 - \chi_3 \zeta_1$ and $ F_2 =\chi_1 \zeta_2 -\chi_2 \zeta_1$. A short computation yields
\begin{subequations}
    \begin{align}
         \frac{d}{dt}\left( \chi_2^2 - \zeta_3^2\right)&= 2 \left(\frac{6}{f}-\frac{2a}{bc}\right)\left( \chi_2^2 - \zeta_3^2\right)\\
          \frac{d}{dt}\left( \chi_3^2 - \zeta_2^2\right) &= 2 \left(\frac{6}{f}-\frac{2a}{bc}\right)\left( \chi_3^2 - \zeta_2^2\right)
    \end{align}
\end{subequations}
Using the long time behaviour of the abelian solutions, Proposition \ref{prop:vanishing-chi-abelian}, we have 
\begin{proposition}
The ODE system \eqref{eq:flat SO(3)-instaton N11} cannot have smooth bounded solutions unless $$\chi_1=0 \,, \qquad  \zeta_1=0 \,, \qquad  \chi_2= \pm \zeta_3 \qquad \text{and} \qquad  \chi_3 = \pm \zeta_2\;.$$ 
\end{proposition}
\begin{proof}
Take $u_1^2= \chi_2^2 - \zeta_3^2$ and $u_2^2 ={\chi_3^2 - \zeta_2^2}$. The quantities $\chi_1, \zeta_1, u_1$ and $u_2$ are governed by the ODE $\Xi' =  (6-2A) \Xi$, which appeared in the study of abelian connections.

Over $p_0^1$, the only fixed point is $\Xi=0$, and the linearisation is positive, so there are no smooth solutions, since we are working in negative time in this case. Over $p_0^{-2}$, the fixed point $\Xi=0$ has positive linearisation, so there is a smooth local solution to the ODE. However, by Proposition \ref{prop:vanishing-chi-abelian}, we know that any solution with $\Xi_0 \neq 0$ will have exponential growth along the non-compact end. Hence, the only smooth bounded solution to this ODE is the trivial $\Xi \equiv 0$. \end{proof}
Hence, without loss of generality, we are reduced to the system 
\begin{subequations}\label{eq:reduced SO(3) N11-0} \begin{align}
    \xi' &= (2BC- AC-AB) \frac{\xi}{3} -BC\left(\chi^2 +(-1)^{\sigma}  \zeta^2 \right) \\
    \chi' &= \left(6-2A -\xi \right)\chi\\
    \zeta' &= \left(6-2A -(-1)^{\sigma}\xi \right)  \zeta
    \end{align}
\end{subequations}
where we have a remaining choice of sign. 
In the case $\sigma=0$, we have an $S^1$ symmetry between $\chi$ and $\zeta$, so we can set $\chi = \psi \cos(\theta)$ and $\zeta = \psi \sin (\theta)$ for a given $\theta \in S^1$. The ODE system \eqref{eq:reduced SO(3) N11-0} reduces to \begin{subequations}\label{eq:reduced-pure-signature}
    \begin{align}
        \xi' &= \left(2BC-AC-AB \right) \frac{\xi}{3} - BC \psi^2\\
        \psi' & = -(2A -6 + \xi) \psi\;, 
    \end{align}
\end{subequations}
with $\theta$ constant, which is precisely the generic case of the $\Spin(7)$-instanton for gauge group $\SO(3)$ with respect to $\widehat{m}=0$. In particular, one readily checks that 
\begin{corollary}
    All solutions to \eqref{eq:reduced-pure-signature} emanating from $(p_0^1,0,0)$ have $\psi=0$. That is, all instantons over $M^1_{1,1}$ are abelian.
\end{corollary}
In the case $\sigma=1$, we have 
\begin{lemma}
   The fixed points of the extended system of mixed signature over the singular orbit points $p_0^i$ are 
    $$\left(p_0^1, 0, 0,0 \right)\,, \qquad \qquad \left(p_0^{-2}, 0, 0 ,0 \right)\,,\qquad \qquad \left(p_0^{-2}, -6 ,0,2 \right) \qquad \text{and} \qquad \left(p_0^{-2}, 6 ,2,0\right)\;.$$
\end{lemma}
Computing the linearisations at $\left(p_0^1, 0, 0,0 \right)$, we find that the only smooth solutions emanating from $(p_0^1,0,0,0)$ are the abelian ones. The linearisation at $(p_0^{-2},0,0,0)$ has one positive eigenvalue, $\lambda=6$, with multiplicity five and obvious eigenvectors, and  the linearisations at $(p_0^{-2},6,2,0)$ and $(p_0^{-2},-6,0,2)$ have two positive eigenvalues, $\lambda=6$ and $\mu=12$, each with multiplicity two.

In particular, we have one additional parameter in the space of local solutions compared to the generic case, corresponding to the choice $\chi \neq0$ (resp. $\zeta \neq  0$). As we shall see, these new families are unbounded along the non-compact end, so we do not go into further detail.

\section{\texorpdfstring{$\SO(3)$}{SO(3)}-connections -- Long time existence} \label{sec:long-time-SO3}
We focus our discussion on the dynamical system associated with a generic Aloff--Wallach space, since all the exceptional cases admit a reduction to it, up to an $S^1$-freedom (or the mixed case, which will be discussed at the end of the section).

Hence, we would like to understand the long-time existence of the ODE system \eqref{eq:new-generic-SO3-instanton}:
\begin{subequations}
\begin{align}
    \xi' &=-F \xi  + \widehat{m}\widehat{F}  - BC\psi^2 \\
    \psi'&= -(2A+3m+\xi) \psi\;,
\end{align}    
\end{subequations}
and its analogues for $\widehat{k}$ and $\widehat{l}$. Our main tool for doing so will be splitting the $(\xi, \psi)$-phase space into sectors that are invariant under the ODE. We first give the details for the case $n=\widehat{m}$.
\begin{proposition}\label{prop:invarinat sectors}
Consider the sector decomposition of the half plane $(\xi, \psi) \in \R \times [0, \infty)$ for a given time $s\in \R$: 
    \begin{align*}        
        \mathcal{A} & \coloneqq  \left\{ \xi \geq -3m \right\} \cap \left\{ \psi^2  \leq 4 - 2 \left(\frac {A}{B} + \frac{A}{C}\right) \right\}\\
        \mathcal{B} & \coloneqq   \left\{\xi\leq \frac{\widehat{m}^2}{m}\right\}\cup \left(\{\xi\leq -3m -2A\}\cap \{\psi^2\geq 4\}\right)\\
        \mathcal{C} & \coloneqq \left\{\frac{\widehat{m}^2}{m}\leq \xi\leq -3m\right\}\cap \{\psi^2\leq 4\}\\
        \mathcal{F} & \coloneqq  \left( \left\{ \xi \geq -3m -2A \right\} \cap \left\{ \psi^2 \geq 4\right\}\right) \cup \left( \left\{ \xi \geq -3m  \right\} \cap \left\{ 4 - 2 \left(\frac {A}{B} + \frac{A}{C}\right) \leq \psi^2 \leq 4\right\} \right)     \end{align*}
Then,
\begin{enumerate}[label=(\roman*)]
    \item the set $\mathcal{A}$ is forward invariant, 
    \item the set $\mathcal{B}$ is forward invariant, and
    \item the set $\mathcal{F}$ is backwards invariant.
\end{enumerate}
\end{proposition}
Note that the sets $\mathcal{A}$, $\mathcal{B}$ and $\mathcal{F}$ are not time-independent. 
\begin{proof}
    Note that the intersection of any two subsets is contained in the boundary of the respective subsets. Hence it suffices to understand the behaviour at the boundary of each set. 
    \begin{enumerate}[label=(\roman*)]
        \item Along the boundary $4- \psi^2  -2 \left( \frac{A}{B} + \frac{A}{C}\right) =0$, we have 
        \begin{align*}
            \left(4-\psi^2 - 2 \left(\frac{A}{B} + \frac{A}{C}\right)\right)'       &=-2\left(\frac{A}{B}\left(4B-4A-\frac{l}{\Delta}AC+\frac{m}{\Delta}BC\right)+\frac{A}{C}\left(4C-4A-\frac{k}{\Delta}AB+\frac{m}{\Delta}BC\right)\right)\\
            &\quad +2\psi^2(2A+3m+\xi)\\
            &\geq -2\left(\frac{A}{B}\left(4B-4A\right)+\frac{A}{C}\left(4C-4A\right)\right) +4A\psi^2+2(3m+\xi)\psi^2\\
            &=-4A\left[4 - \psi^2 -2\left(\frac{A}{B}+ \frac{A}{C}\right)\right]+2(3m+\xi)\psi^2\\&  \geq 0\
        \end{align*}
    since $\xi \geq -3m $ by assumption. Similarly, at the boundary $\{\xi+3m=0\}$, we have 
    \begin{equation}\label{eq:xiprimeat-3m}
        \xi'=4BC-2AC-2AB-BC\psi^2=BC\left(\frac{4BC-2AC-2AB}{BC}-\psi^2\right)\geq 0\;,
    \end{equation}
    so the set $\mathcal{A}$ is forward invariant.
        \item Similarly, along the boundary $\xi = \frac{\widehat{m}^2}{m}$, we have 
        \begin{equation}
            \xi' =  2 \frac{\widehat{m}}{m} A (C-B) - BC \psi^2 \leq 0\;,
        \end{equation}
        so the set $\left\{ \xi \leq \frac{\widehat{m}^2}{m}\right\}$ is forward invariant. Let us analyse the other piece. Since we assume $2A+\xi+3m\leq 0$, the inequality $\psi^2\geq 4$ is preserved. Similarly, we have 
    \begin{align*}
    \left.(2A+\xi+3m)'\right|_{2A+\xi+3m=0}&= 2A \left( \frac{m}{\Delta}BC- 3A\right) + \left(4 -  \psi^2\right)BC\leq 0\;,
    \end{align*}
    as $\psi^2  \geq 4$. Hence the set $\mathcal{B}$ is invariant.
     \item Finally, we need to show no solution can enter $\mathcal{F}$. Since $\mathcal{A}$ and $\mathcal{B}$ are forward invariant, it suffices to show that a solution cannot cross from $\mathcal{C}$. Hence, we need to show that the boundary $$\mathcal{F} \cap \mathcal{C} =  \left\{ \xi = -3m \;, \psi^2 \in \left[4 - 2 \left(\frac{A}{B} + \frac{A}{C} \right), 4 \right] \right\} \cup \left\{ \xi \in [-3m-2A, -3m] \;,  \psi =2 \right\} \;.$$
     acts as a barrier for our ODE system. But, akin to Equation \eqref{eq:xiprimeat-3m}, we now have 
    \begin{equation}
     \xi' =  BC\left(\frac{4BC-2AC-2AB}{BC}-\psi^2\right)\leq 0\;,
    \end{equation}
     when $\xi = -3m$. Similarly, since $2A +\xi +3m \geq 0$, $\left. (\psi^2)' \right|_{\psi^2=4} \leq 0$, which concludes the proof. \qedhere
     \end{enumerate}
\end{proof}
We refer to the sectors $\mathcal{A}, \mathcal{B}, \mathcal{C}$ and $\mathcal{F}$ as the \textit{abelian}, \textit{blow-up}, \textit{critical} and \textit{forbidden} sectors respectively. A first result that justifies our choice of nomenclature is
\begin{proposition} \label{prop:finite-time-blowup}
    Any integral curve that enters $\mathcal{B}$ with $\psi \neq 0$ blows up in finite time. 
\end{proposition}
Before we give the proof, we need the following lemma.
\begin{lemma} \label{lemma:blow-up-comparsion}
Consider the ODE system 
$$ \dot{x}(t) =F_1(t)x +F_2(t)y^2   \qquad \qquad \dot{y}(t)= (x-\kappa_1)y$$
where $F_1(t)>0$ and $t^2 F_2(t) \geq \kappa_2$ for some constants $\kappa_1 \geq 0$ and $\kappa_2>0$. Then, any solution with initial conditions $x_0\geq 0$, $y_0 > 0$ blows up in finite time.

The same holds if $F_1\geq 0$ and $x_0 \geq \kappa_1$.
\end{lemma}
\begin{proof}
The main idea is to find a simpler comparison ODE that has finite-time blow-up. 

First, note that there exists a finite time $t_1\geq t_0$ for which $x(t_1)\geq\kappa_1$ since $\dot x>0$ by the initial conditions and $\dot y <0$ whenever $0< x <\kappa_1$. Thus, without loss of generality, we may take initial conditions $x_0 \geq  \kappa _1$, $y_0> 0$.

Now, by differentiation, we have
$$ \ddot{y} =  \dot{x} y + (x-\kappa_1)\dot{y} = F_1 xy+ F_2y^3  + (x-\kappa_1)^2 y\geq  \kappa_2 \frac{y^3}{t^2}\;,$$
since the first and last terms are non-negative by the chosen initial condition. Hence, it suffices to prove that $$t^2 \ddot{y} = \kappa_2 y^3$$ blows up in finite time for $\kappa_2>0$ and $y_0>0, \dot y_0 >0$.

Consider the change of variables $\tau = \log(t)$, so the ODE becomes $y'' = \kappa_2 y^3 +y'$, which can be compared with $y'' = \kappa_2 y^3$. Multiplying by $y'$ and integrating, we have
$$ \frac{d}{d\tau }y(\tau) = \sqrt{ \frac{\kappa_2y^4}{2} +K}\;,$$
with $K= {y'_0}^2- \frac{\kappa_2 y_0^4}{2}$. The blow up time is given by the integral
$$ T_\infty  =  \sqrt{2} \int_{y_0}^\infty \frac{1}{\sqrt{ \kappa_2y^4 +2K}} dy\;,$$
which is well defined since $2K +\kappa_2 y^4_0 \geq 0$. Let us check that the integral is indeed finite. As $y \to \infty$, the integrand behaves as $\sim \frac{1}{y2}$, so the tail is finite.

In the limiting case $2K +\kappa_2 y^4(0)=0$, we have a singularity at $y(0)$. Near $y_0$ we have the expansion
$$ \sqrt{y^4- y_0^4} = \sqrt{(y-y_0)} \sqrt{(y+y_0)(y^2+y_0^2)} \ \sim 2 \sqrt{y_0^3
}\;\sqrt{y-y_0}\;, $$
and so $T_\infty$ is always finite, as claimed. 
\end{proof}

\begin{proof}[Proof of Proposition \ref{prop:finite-time-blowup}]
Using Lemma \ref{lemma:identities-klm-khatlhatmhat}, we can rewrite the $\Spin(7)$-ODE system as 
 \begin{align*}
    \xi' & = \left( \frac{\widehat{m}^2}{m} - \xi\right) F  +2\frac{\widehat{m}}{m}A(C-B) -  BC\psi^2\;,\\
    \psi '& = -(2A +3m + \xi) \psi \;.  
 \end{align*}
Consider the change of variables $x =\frac{\widehat{m}^2}{m} - \xi$, $y=\psi $. We set $  F_1 =  - F\geq 0$, $F_2= BC$ and 

$\kappa = \max \left\{ 0 , \sup_\R \left( A + \frac{2\Delta}{m}\right) \right\}$, so the system transforms to 
 \begin{align*}
    x' & = F_1 x - 2\frac{\widehat{m}}{m}A(C-B) + F_2 y^2 \geq F_1x + F_2 y^2\;,\\
    y '& =  xy  - 2\left(A+\frac{2\Delta}{m}\right) y \geq (x- 2\kappa) y \; 
 \end{align*}
Since we have $A,B, C  \sim \frac{1}{s}$  as  $s \uparrow \infty$ (in fact $A,B, C \sim  1$ in the AC case), Lemma \ref{lemma:blow-up-comparsion} gives us finite time blow up for $(\xi , \psi) \in (-\infty, \frac{\widehat{m}^2}{m}] \times (0, \infty)$. 

Now, assume $(\xi , \psi^2) \in (-\infty, {-3m-2A}] \times (4, \infty)$. Under these conditions, we have
\begin{align*}
    -(2A+3m+\xi)' &=  -2A\left(B+C-3A- \frac{k}{\Delta}AB - \frac{l}{\Delta}AC\right) -  \left(\frac{\widehat{m}^2}{m} - \xi\right) F- 2 \frac{\widehat{m}}{m}A(C-B) + BC\psi^2\\
    &\geq  -2A\left(B+C-3A- \frac{k}{\Delta}AB - \frac{l}{\Delta}AC\right) -  \left(\frac{4\Delta}{m} + 2A\right)F- 2 \frac{\widehat{m}}{m}A(C-B) + BC\psi^2\\
    &= 2A\left(3A - \frac{m}{\Delta}BC\right) + BC(\psi^2-4)\\
    & \geq BC(\psi^2-4)\;.
\end{align*}
Hence, if we take $x= -(2A+3m+\xi)$, $y= \sqrt{\psi^2-4}$, $F_1=0$ and $F_2=BC$, Lemma \ref{lemma:blow-up-comparsion} shows that the solutions blow up in finite time, proving the claim.
\end{proof}
Let us study the short time behaviour of the local integral curves $\Gamma (a, \alpha)$ and $\Omega(c)$ with respect to the $\mathcal{ABCF}$-decomposition. We have
\begin{lemma}\label{lemma:short time integrable curves}
Fix a forward complete $\Spin(7)$-metric $\gamma^m_\theta$, with $\theta\in(0,\theta^*]$. We consider the families of integral curves $\Gamma^m_\theta(a,\alpha)$ and $\Omega^m_\theta(c)$ emanating from
$\left(p_0^m,\frac{\widehat m^2}{m},0\right)$ and $(p_0^m,-3m,2)$, respectively, as defined in \eqref{eq:local-Gamma a,b} and \eqref{eq:local-Omega c}. We have the following:
\begin{enumerate}[label=(\roman*)]
    \item For $\alpha \in (0, \pi/2)$, and $s$ sufficiently negative, we have  $\Gamma^m_\theta(a,\alpha) \in \mathcal{C}$ and $\Gamma^m_\theta(a, \pi - \alpha) \in \mathcal{B}$.
    \item For all $c \in \R$, $\Omega^m_\theta(c) \in \mathcal{C}$ for $s$ sufficiently negative.
    \item For all $s$ sufficiently negative, there exists some $c_*>0$ large enough such that $\Omega^m_\theta(c)(s) \in \mathcal{A}$ and $\Omega^m_\theta(-c)(s) \in \mathcal{B}$ for all $c>c_*$.
\end{enumerate}
\end{lemma}
\begin{proof}
Since all eigenvalues have non-zero real parts, the Hartman--Grobman theorem guarantees that the local behaviour of our non-linear ODE is fully determined by its linearisation around the fixed point. In particular, the integral curves $\Gamma^m_\theta (a,\alpha)$ and $\Omega^m_\theta(c)$ are well defined. Claim $(i)$  follows readily. 

For $(ii)$, note that the eigenvector $w_4$ has a larger eigenvalue than the eigenvalue of $w_1$ and $w_2$, so all integral curves $\Omega^m_\theta(c)$ are tangent to 
$v = \cos(\theta) w_1 + \sin(\theta)w_2$ at the fixed point $(p,-3m , 2)$. The claim follows by noticing that the projection of $v$ to the $(\xi, \psi)$-plane is $(6m\Delta \sin \theta, - 3m^2\sin \theta)$, which points towards the interior of the critical sector $\mathcal{C}$.

Finally, for $(iii)$, note that for all sufficiently negative times, we can find initial conditions $C_1$ and $C_2$ such that the corresponding solutions to the linearised equation satisfy $\xi_L(s)\geq -3m$ and $\psi(s) \geq 2$ respectively. Indeed, it suffices to take 
\[ C_1 \geq  \frac{3}{2}\sin \theta \,e^{\frac{4\Delta}{m}s}  \qquad \qquad \text{and} \qquad \qquad C_2 \leq - 3\sin \theta \,e^{\frac{4\Delta}{m}s}\;.\]
Now, consider the one-parameter family of ODEs $\Psi'= L\Psi + \epsilon Q(\Psi)$, so that we get the full non-linear system at $\epsilon=1$ and the linearised system at $\epsilon=0$. Take the smooth function  
$$F(s, \epsilon, C) = \psi_{C,\epsilon} (s) -2\;.$$
For some $s_0$, we have $c_0$ such that $F(s_0, 0, c_0)=0$, and a straightforward application of the implicit function theorem gives us a family $(s_\epsilon, c_\epsilon) $ such that $ F(s_\epsilon, \epsilon, c_\epsilon)=0$, for some $\epsilon$ small enough. 

Now, in order to make sure we can enlarge $\epsilon$ to $\epsilon=1$, we need to effectively control $\frac{d s_\epsilon }{d \epsilon} =  - \frac{\partial_\epsilon F }{\partial _s F}$, which   we can do as long as $s_0$ is sufficiently negative to begin with, since the terms involved in $\varepsilon$ are higher order in $s$.
\end{proof}
We now give a technical result that will be key in understanding solutions in the critical sector $\mathcal{C}$. Its proof is topological in spirit, and uses the following elementary result from point-set topology:
\begin{lemma}\label{lemma:topolgy-closed-chain}
    A descending chain of non-empty closed sets in a compact space has a non-empty closed intersection.
\end{lemma}
\begin{proposition}\label{prop:long-time-critical}
Let $K$ be a connected compact subset of the space of initial conditions of the ODE system \eqref{eq:new-generic-SO3-instanton}, such that the solutions $\{S_p \}_{p \in K}$ are not in $\mathcal{F}$ for any $s\in \R$. Assume that for some $s_* \in \R$,  there exist $p, p' \in K$ such that $S_p(s_*) \in \mathcal{A}$ and $S_{p'}(s_*) \in \mathcal{B}$. Then there exists $q\in K$ such that $S_q(s) \in \mathcal{C}$ for all $s \in \R$. 
\end{proposition}
\begin{proof}
For each $ s \in \R $, we consider the set $$\mathcal{A}_s^{-1} \coloneqq \left\{ \alpha \in  K \st F_\alpha(s) \in \mathcal{A}\right\} \subseteq  K\;.$$
Analogously, we define $\mathcal{B}^{-1}_s$ and $\mathcal{C}^{-1}_s$. The former requires some extra care, as the solution might have blown up at an earlier time. This gives rise to a closed partition, 
$$K =  \mathcal{A}^{-1}_s \cup \mathcal{B}^{-1}_s \cup \mathcal{C}^{-1}_s\;,$$
since the ODE is non-singular. By the forward invariance of $\mathcal{A}$ and $\mathcal{B}$, we have $\mathcal{C}_t^{-1} \subseteq \mathcal{C}_s^{-1}$ whenever $t \geq s$.

Finally, note that $\mathcal{C}_s^{-1} \neq \emptyset$ for any $s\geq s_*$ by continuity with respect to initial conditions, since $\mathcal{A}_s^{-1}$ and $\mathcal{B}_s^{-1}$ are non-empty by assumption and $K$ is connected. By Lemma \ref{lemma:topolgy-closed-chain}, the existence of $q$ follows.
\end{proof}
We have dealt exclusively with the case $Q_{\widehat{m}} \rightarrow M^m_{k,l}$. For $Q_{\widehat{k}} \rightarrow M^k_{k,l}$, the analogous decomposition holds:
\begin{proposition}
    Consider the following sector decomposition of the half plane $(\xi, \psi) \in \R \times [0, \infty)$ for a given time $s\in \R$: 
    \begin{align*}
        \mathcal{A} & \coloneqq  \left\{ \xi \leq -3k \right\} \cap \left\{ \psi^2  \leq 4 - 2 \left(\frac {C}{A} + \frac{C}{B}\right) \right\}\\
        \mathcal{B} & \coloneqq   \left\{\xi\geq \frac{\widehat{k}^2}{k}\right\}\cup \left(\{\xi\geq -3k -2C\}\cap \{\psi^2\geq 4\}\right)\\
        \mathcal{C} & \coloneqq \left\{ -3k\leq \xi\leq \frac{\widehat{k}^2}{k}\right\}\cap \{\psi^2\leq 4\}\\
        \mathcal{F} & \coloneqq  \left( \left\{ \xi \leq -3k -2C \right\} \cap \left\{ \psi^2 \geq 4\right\}\right) \cup \left( \left\{ \xi \leq -3k  \right\} \cap \left\{ 4 - 2 \left(\frac {C}{A} + \frac{C}{B}\right) \leq \psi^2 \leq 4\right\} \right) 
    \end{align*}
Then, the sets $\mathcal{A}$ and  $\mathcal{B}$ are backwards invariant, and the set $\mathcal{F}$ is forward invariant.
\end{proposition}
\begin{proof}
    The proof is almost identical to the case of $\widehat{m}$. The only subtlety is proving $X=  4- \psi^2 - 2 \left(\frac {C}{A} + \frac{C}{B}\right)$ satisfies $X'\leq 0$ whenever $\xi\leq -3k$. We have
    \begin{align*}
        X'&=-2\left[\frac{C}{A}\left(4A-4C+\frac{k}{\Delta}AB-\frac{m}{\Delta}BC\right)+\frac{C}{B}\left(4B-4C+\frac{k}{\Delta}AB-\frac{l}{\Delta}AC\right)\right]+2\psi^2(2C+3k+\xi)\\
        &\leq -4C\left[4 - \psi^2 -2\left(\frac{C}{A}+ \frac{C}{B}\right)\right]+2(3k+\xi)\psi^2  - \frac{2}{\Delta}\frac{AC}{B}\left(kB-lC\right)        
        \\&  \leq 0\
    \end{align*}
    where the last term is negative since $(kB-lC)\leq l(B-C)\leq 0$ since $k \geq l\geq 0$ and $B \leq C \leq 0$ by Proposition \ref{prop:ABC-inequalities}. The rest of the computations follow Proposition \ref{prop:invarinat sectors} verbatim, and are omitted. 
\end{proof}
The analogues of Proposition \ref{prop:finite-time-blowup} and Lemma \ref{lemma:short time integrable curves} can be seen to hold easily for $Q_{\widehat{k}}$ where now the blow-up occurs in finite backwards time. 

For $Q_{\widehat{l}}$, the analogous sector decomposition is no longer invariant under the ODE. We do not have a good geometric justification for this. We introduce the following modified decomposition.
\begin{proposition}
    Consider the following sector decomposition of the half plane $(\xi, \psi) \in \R \times [0, \infty)$ for a given time $s\in \R$: 
    \begin{align*}
        \mathcal{A}_\mu& \coloneqq  \left\{ \xi \leq -3l+\mu(ABC)^{\frac{1}{3}} \right\} \cap \{\psi^2  \leq 4\} \\
        \mathcal{B} & \coloneqq   \left\{\xi\geq \frac{\widehat{l}^2}{l}\right\}\cup \left(\{\xi\geq -3l -2B\}\cap \{\psi^2\geq 4\}\right)\\
        \mathcal{C}_\mu & \coloneqq \left\{ -3l+\mu(ABC)^{\frac{1}{3}}\leq \xi\leq \frac{\widehat{l}^2}{l}\right\}\cap \{\psi^2\leq 4\}\\
        \mathcal{F} & \coloneqq   \left\{ \xi \leq -3l - 2B \right\} \cap \left\{ \psi^2\geq 4  \right\},
    \end{align*}
Then, the set $\mathcal{B}$ is backwards invariant, and the set $\mathcal{F}$ is forward invariant, and there exists some $\mu>0$ such that $\mathcal{A}_\mu$ is backward invariant.
\end{proposition}

\begin{proof}
We verify that there exists some $\mu>0$ such that $(\xi + 3l- \mu(ABC)^{\frac{1}{3}})'\geq 0$ whenever $\xi= - 3l +\mu(ABC)^{\frac{1}{3}}$ and $\psi^2\leq 4$.
\begin{align*}
  (\xi + 3l- \mu(ABC)^{\frac{1}{3}}))' &=  \frac{3l}{\Delta}(kAB+lAC +m BC) + \frac{\widehat{l}}{\Delta}\left( \widehat{k}AB + \widehat{l}AC+ \widehat{m}BC\right)\\
  &\quad - \frac{\mu(ABC)^{\frac{1}{3}}}{\Delta}(kAB+lAC +m BC)- AC\psi^2\\
  &\quad -\mu(ABC)^{\frac{1}{3}}\left(-\frac{A+B+C}{3}-\frac{2m}{3\Delta}BC-\frac{2l}{3\Delta}AC-\frac{2k}{3\Delta}AB\right)\\
  &= ( 4- \psi^2) AC-2AB-2BC\\
  &\quad -\mu(ABC)^{\frac{1}{3}}\left(-\frac{A+B+C}{3}+\frac{m}{3\Delta}BC+\frac{l}{3\Delta}AC+\frac{k}{3\Delta}AB\right)\\
  &\geq  -\mu(ABC)^{\frac{1}{3}}\left(-\frac{A+B+C}{3}\right)-2AB-2BC\\
  &=A^2\left( \mu(\mathsf{b} \mathsf{c})^{\frac{1}{3}}\frac{1+\mathsf{b}+ \mathsf{c}}{3}-2\mathsf{b}-2\mathsf{b}\mathsf{c}\right)
  \end{align*}
where $\mathsf{b}= \frac{B}{A}$, and $\mathsf{c}= \frac{C}{A}$ as usual. Consider $\kappa=\left(\frac{\mathsf{b}}{\mathsf{c}}\right)^{\frac{1}{3}} = \left(\frac{B}{C}\right)^{\frac{1}{3}} $, and take $$\mu=12 \max\limits_{\theta\in[0,\theta_*]} \sup\limits_{ s \in \R} \,\kappa(s)\;.$$
Note that the supremum is always finite since $\kappa$ is continuous and $ \displaystyle{\lim_{{s \to \pm \infty}}\kappa < +\infty}$. Since $0<\mathsf{b}, \mathsf{c} \leq 1$, we have
\begin{align*}
\mu(\mathsf{b}\mathsf{c})^{\frac{1}{3}}\frac{1+\mathsf{b}+\mathsf{c}}{3}-2\mathsf{b}-2\mathsf{b}\mathsf{c}&\geq \frac{\mu}{3}(\mathsf{b}\mathsf{c})^{\frac{1}{3}}-4\mathsf{b}\geq \frac{\mu}{3}(\mathsf{b}\mathsf{c})^{\frac{1}{3}}-4\mathsf{b}^{\frac{2}{3}}=\mathsf{c}^{\frac{2}{3}}\left(\frac{\mu}{3}-4\kappa\right)\kappa \geq 0 \;.\qedhere
\end{align*}
\end{proof}

The analogues of Proposition \ref{prop:finite-time-blowup} and Lemma \ref{lemma:short time integrable curves} can be adapted easily again. The case of (iii) in Lemma  \ref{lemma:short time integrable curves} requires some additional care:
\begin{lemma}
    For any $s$ large enough, there exists $c_*$ such that the integral curve $\Omega_c(s)$ goes into $\mathcal{A}_\mu$ for all $c>c_*$. 
\end{lemma}
\begin{proof}
 As in the previous case, we deal with the linearised problem, and then use the implicit function theorem to perturb to the original problem. For simplicity, let $t= e^{-\frac{4\Delta}{l}s}$. We have the series expansion
$$ \mu (ABC)^{\frac{1}{3}} =  - \mu \Delta\left( \frac{24 \sin \theta}{l} t\right)^\frac{1}{3}  \left[ 1+  \frac{l}{\Delta}\left( \cos \theta - \widehat{l} \sin \theta \right)t  \right]  + O(t^{\frac{7}{3}})\;.$$
Hence, solving the linearised problem corresponds to finding $C$ solving 
  $$6l \Delta \sin (\theta) t  - 4l\Delta C t^2  =  - \mu \Delta\left( \frac{24 \sin \theta}{l} t\right)^\frac{1}{3}  \left[ 1+  \frac{l}{\Delta}\left( \cos \theta - \widehat{l} \sin \theta \right)t  \right]  + \widehat{C}t^{\frac{7}{3}}$$
for $t>0$ small enough, which is straightforward. The rest of the proof follows the previous case. 
\end{proof}
In order to proceed, it becomes useful to distinguish between the asymptotically locally conical and asymptotically conical cases. We begin with the former. 
\subsection{The ALC case}
Let us begin by studying the fixed points of the generic $\Spin(7)$-instanton ODE system and their linearisations at the ALC limit.
\begin{lemma}
    The fixed points of \eqref{eq:new-generic-SO3-instanton} at the ALC limit $(A,B,C)  = (0,0,0)$ are the lines 
    $$ (0, \xi_\infty, 0) \qquad (0, -3m, \psi_\infty)\;.$$
The respective linearisations are 
$$ \begin{pmatrix}
        0  & 0 \\ 0 & -(\xi_\infty+3m)
    \end{pmatrix}\qquad \text{and} \qquad \begin{pmatrix}
        0  & 0 \\ -\psi_\infty & 0
    \end{pmatrix}\;.$$
\end{lemma}
Note that $\partial \mathcal{A} \cap \{(0, \xi_\infty, 0) \st \xi_\infty \in \R \} = \{(0, \xi_\infty, 0) \st \xi_\infty \geq-3m\}$ corresponds precisely to the space where the function $\psi$ decays exponentially fast at infinity. In fact, we have 
\begin{lemma}
    Every integral curve that enters $\mathcal{A}$ converges to $( \xi_\infty, 0)$ for some $\xi_\infty \geq -3m$. 
\end{lemma}
\begin{proof}
    Since the set $\mathcal{A}$ is forward invariant, the function $\psi^2$ decays exponentially to zero in $\mathcal{A}$. Then the function $\xi$ also converges as $s \to \infty$ by our discussion of the abelian case, since $BC = s^{-2}+ o(s^{-2})$ as $s$ goes to infinity, the term $BC\psi^2$ becomes negligible.
\end{proof}
Using this lemma, we have 
\begin{proposition}\label{prop:long-time ALC perturbative}
For each ALC metric $\gamma_{\theta}^m$ on $M_{k,l}^m$, let $a_\theta
^*\in \R$ be the initial condition such that the abelian solution $\Gamma^m_{\theta}(a^*_\theta,0)$ converges to $(-3m,0)$. Then, for all $a >a_\theta^*$, there exists a sufficiently small $\delta>0$ such that $\Gamma^m_{\theta}(a,\alpha)$ converges to some abelian connection $(\xi_\infty,0)$ with $\xi_\infty>-3m$ whenever $\alpha \in [0,\delta)$.

Similarly, there exists $c^*_\theta$ such that, for all $c>c^*_\theta$, the integral curve $\Omega^m_{\theta}(c)$ converges to some abelian connection $(\xi_\infty,0)$ with $\xi_\infty>-3m$.
\end{proposition}
\begin{proof}
If $\xi_\infty>-3m$ for some initial condition $a> a^*$, the integral curve $\Gamma^m_{\theta}(a,0)$, eventually enters the set $\mathcal{A}_\epsilon$ for some $\varepsilon>0$, where $\mathcal{A}_\varepsilon= \mathcal{A} \cap \{\xi \geq -3m +\varepsilon\}$. By continuous dependence on initial conditions, there exists a sufficiently small $\delta>0$ such that for $\alpha\in(0,\delta)$, the integral curve $\Gamma^m_\theta(a,\alpha)$ enters $\mathcal{A}$.

The statement for $\Omega_\theta^m(c)$ follows from Lemma \ref{lemma:short time integrable curves} $(iii)$.
\end{proof}
Now, consider the compact subfamilies $\Gamma_\theta^m(a , \alpha)$ for a fixed $a>a^*$ and $\alpha \in \left[ 0, \pi\right]$, and $\Omega ^m_\theta (c)$ for $c \in [-c_*, c_*]$. Applying Proposition \ref{prop:long-time-critical}, we have solutions that are contained in the critical sector $\mathcal{C}$ for all time:
\begin{proposition}\label{prop:long-time ALC np}$\;$
Let $\gamma_\theta ^m$ be an ALC metric. Then, for each $a> a_\theta^*$, there exists $\alpha^* \in \left[ 0, \pi \right]$ such that $\Gamma_\theta^m(a, \alpha^*) (s) \in \mathcal{C}$ for all $s \in \R$.

Similarly, $\Omega_\theta^i(c^*)$ remains in $\mathcal{C}$ for all $s \in \R$:
\end{proposition}
We have
\begin{proposition}\label{prop:long-time ALC np2}
For each ALC metric $\gamma_{\theta}^m$ on $M_{k,l}^m$, any integral curve with $\psi \neq 0$ contained in the critical sector $ \mathcal{C}$ converges to $(0,-3m,0)$ as $s \to \infty$. Moreover, it has asymptotics
    $$\xi+ 3m \approx  C s^{-2} \qquad \qquad \psi \approx D s^{-2}\;.$$
\end{proposition}
\begin{proof}
    The limiting set within the critical sector $\mathcal{C}$ is $\left[\frac{\widehat{m}^2}{m}, -3m\right] \times \{0\} \cup  \{-3m\} \times [0, 2]$. Computing the linearisation of the ODE at the ALC point, it follows that the segment $\left[\frac{\widehat{m}^2}{m}, -3m\right)$ is unstable in the $\psi$-direction, so our integral curve must converge to $\{-3m\} \times [0, 2]$.
    
    For $\{-3m\} \times [0, 2]$, both eigenvalues of the linearisation  are zero, so the Hartman–-Grobman theorem gives us no information. Let us blow up the ODE system at infinity and exploit the ALC asymptotics from Proposition \ref{prop:ALC-asymptotic-behaviour}. 
    
    Take the change of coordinates $\tau= \log s$, and introduce the coordinate $\widehat{s} = \frac{1}{s}$. By Proposition \ref{prop:ALC-asymptotic-behaviour}, we have $A(\widehat{s}), B(\widehat{s}), C(\widehat{s}) = \widehat{s}+o(\widehat{s})$. In particular, we have $F, \widehat{F}=o(\widehat{s}\,^{2})$ since $k+l+m=0=\widehat{k}+\widehat{l}+\widehat{m}$. 
    
    Taking the change of variables, $X= \xi+3m$, the ODE system becomes the autonomous ODE system 
    \begin{align*}
        \partial_\tau X&=-\widehat{s}\psi^2 + o(\widehat{s})\\
        \partial_\tau  \psi &=  -\left(\frac{X}{\widehat{s}}+2\right) \psi + o(1)\\
        \partial_\tau \widehat{s} &= -\widehat{s} \;.
    \end{align*}
    As in Proposition \ref{prop:ALC-asymptotic-behaviour}, this suggests the blow-up change of coordinate $X=\widehat{s}u$, which yields the system
    \begin{align*}
        \partial_\tau u&= u- \psi^2  + o(1)\\
        \partial_\tau \psi &=  -\left(u+2\right) \psi + o(1)\\
        \partial_\tau  \widehat{s} &= - \widehat{s}
    \end{align*}
    The only fixed point of the blown up system is the origin, with linearisation
    $ \diag (1, -2, -1) $ in the $(u, \psi, \widehat{s})$ coordinates. The asymptotic decay rates for $\xi$ and $\psi$ follow. 
\end{proof}

In particular, any $\SO(3)$-connection over the ALC space is asymptotic to the pullback of an abelian $G_2$-monopole on the cone $(1, \infty) \times \SU(3)/T^2$.

In summary, we have proved the following.
\begin{theorem}\label{thm:ALC-instantons-existence}
   For every ALC metric $\gamma_{\theta}^m$ on $M_{k,l}^m$, the bundle $Q_{\widehat{m}}$ admits a one-dimensional family of $\Spin(7)$-instantons. 

    If, in addition, $\widehat{m}=0 \mod m $, the bundle $P_{\widehat{m}} \to M_{k,l}^m$ carries
    \begin{enumerate} [label=(\roman*)]
        \item a two-dimensional family of non-abelian $\Spin(7)$-instantons converging exponentially to the abelian connection $(\xi_\infty, 0)$, with $\xi_\infty >-3m$ , and 
        \item a one-dimensional family of non-abelian $\Spin(7)$-instantons converging to the abelian connection $(-3m, 0)$.
    \end{enumerate}
\end{theorem}
Applying the same arguments to the cases $\widehat{k}$ and $\widehat{l}$, we obtain the corresponding versions of the theorem. 

The critical value $\xi_\infty = -3i$ has a natural interpretation in terms of the holonomy around the asymptotic circle at infinity. As usual, we focus on the case $i=m$, the other cases being analogous. The diagonal part of the limiting connection is $(\xi_\infty e_7 +\widehat{m} e_8) \otimes T_1$. 

Consider the character $\chi_{\widehat{m}}: \diag( z_1, z_2, z_3) \mapsto z_1z_2^{-1}$.  In our basis \eqref{eq:su3-basis}, we have
$$ d\chi_{\widehat{m}}(E_7)= -3m  \qquad \qquad d\chi_{\widehat{m}}(E_8) = \widehat{m}\;.$$

Hence, $A_0 = (-3m e_7 +  \widehat{m}e_8) \otimes T_1$ is the pullback of the homogeneous pseudo Hermite-Yang-Mills connections on the associated line bundle over the flag manifold $\SU(3)/T^2$. In particular, its holonomy along the circle fibration $$\U(1) \cong T^2/ \U(1)_{k,l} \to N_{k,l} \to \SU(3)/ T^2$$
is trivial. Thus, the limiting connection decomposes as 
    $$ A_\infty  = A_0 +  (\xi_\infty + 3m)e_7 \otimes T_1\;.$$
This identifies $X_\infty  =  \xi_\infty +3m $ as the vertical difference $A_\infty - A_0$ between the limiting connection and lift of the canonical connection. The analogous root characters give \(X_\infty=\xi_\infty+3i\) for \(i=k,l\). 

The $E_7$-flow has period $\frac{\pi}{\Delta}$, so
$$ \operatorname{Hol}(A_\infty) =  \exp \left( \frac{\pi}{\Delta} X_\infty \;T_1\right)\;. $$
In particular, the critical value $X_\infty=0$ corresponds to trivial asymptotic holonomy. The coefficient $X_\infty$ retains more information than this holonomy: the latter is unchanged under $X_\infty\mapsto X_\infty+2\Delta n$, whereas the decay of the particular mode retained by our invariant ansatz depends on $X_\infty$ itself. As expected, the limiting value $X_\infty$ controls the asymptotic decay rates of a solution $(\xi, \psi)$. For $\psi$, since it satisfies the equation $\psi' = -(2A+X) \psi$, and $A = s^{-1}$, we have the asymptotic 
$$ \psi (s)  =  \kappa s^{-2} \exp\left({-s\,X_\infty}\right)\; (1+ o(1))\;,$$
for some $\kappa \in \R$. Since $dt=f,ds$ and $f\to\ell>0$, the corresponding exponential rate in geodesic distance is $X_\infty/\ell$. At $X_\infty=0$, the exponential term vanishes, and we recover Proposition \ref{prop:long-time ALC np2}. To describe the diagonal component, let $\xi_{\mathrm{ab}}$ be the abelian solution having the same limit. The difference $D(s)= \xi(s)- \xi_{ab}(s)$ satisfies the ODE 
$$ D' + F D = -BC \psi^2\;.$$
Hence, if $X_\infty>0$, the difference decays exponentially, $D \sim s^{-6}\exp\left({-2s\,X_\infty}\right)$. For $X_\infty = 0$, we get $D \sim s^{-5}$.

This difference in asymptotic behaviour is consistent with the monopole interpretation of the limiting connection, discussed in the abelian case (cf. Remark \ref{rmk:monopole}). Under the asymptotic circle fibration, the limiting model corresponds to an abelian $G_2$-monopole on the cone over $\SU(3)/T^2$. Relative to the lift determined by the root character, its Higgs field is
$\Phi_\infty=\frac{X_\infty}{\ell}T_1$, and is parallel with respect to the limiting abelian connection.

The corresponding distinction between decay rates appears in the analysis of $\SU(2)$ monopoles on AC $G_2$-manifolds by Fadel--Nagy--Oliveira \cite[Thm.~6.1, Prop.~6.3]{FNO}. For a monopole $(A,\Phi)$ of positive mass, the Higgs field splits the adjoint bundle along the end into the line commuting with $\Phi$ and its orthogonal complement. On the latter, the commutator with $\Phi$ supplies a non-zero limiting mass term. Under their finite intermediate energy and curvature decay hypotheses, they prove exponential decay of the transverse components of $F_A$ and $\nabla_A\Phi$. The longitudinal component has no such mass term; their estimate for the full covariant derivative is polynomial, namely $|\nabla_A\Phi|=O(r^{-6})$ in dimension seven.
\subsection{The AC case}
Let us now study the long time existence of solutions in the limiting cases where the metric $\gamma_\theta$ is asymptotic to a $\Spin(7)$ cone. 
\begin{lemma}
Consider the asymptotically conical $\Spin(7)$-metric $\gamma^i_{\theta^*}$ on $M^i_{k,l}$, and the quantities
\begin{subequations}
\begin{align}
     D_{\widehat{m}}&=4-A^+\left(3\frac{1}{B^+}+3\frac{1}{C^+} +\frac{A^+}{B^+C^+}\right) = 4-\left(\frac{3}{\mathsf{b}} +\frac{3}{\mathsf{c}}+ \frac{1}{\mathsf{b}\mathsf{c}} \right)  \,,\\
     D_{\widehat{l}}&=4-B^-\left(3\frac{1}{A^-}+3\frac{1}{C^-} +\frac{B^-}{A^-C^-}\right) = 4-\left(3{\mathsf{b}} +\frac{3\mathsf{b}}{\mathsf{c}}+ \frac{\mathsf{b}^2}{\mathsf{c}}\right) \,,\\ 
     D_{\widehat{k}}&=4-C^-\left(3\frac{1}{A^-}+3\frac{1}{B^-} +\frac{C^-}{A^-B^-}\right)= 4-\left(3{\mathsf{c}} +\frac{3\mathsf{c}}{\mathsf{b}}+ \frac{\mathsf{c}^2}{\mathsf{b}}\right) \,; \label{eq:def-Dk}
\end{align}
\end{subequations}
where $p_{k,l}^{\pm}=(A^\pm, B^\pm, C^\pm)$ are the fixed points of the system \eqref{eq:new-evo-spin7} corresponding to the two nearly parallel $G_2$-metrics on $N_{k,l}$, and $\mathsf{b}=\frac{B}{A}$ and $\mathsf{c}=\frac{C}{A}$, as usual.

The $\Spin(7)$-instanton equation over $Q_{\widehat{i}}$ has one fixed point, $(p^\pm_{k,l}, \widehat{i}\Lambda^\pm, 0)$. If $D_{\widehat{i}} \leq 0$, it is unique. If $D_{\widehat{i}}>0$, the system has an additional fixed point $(p_{k,l}^\pm, \xi_\infty, \psi_\infty)$, with $\psi_\infty\neq 0$. 
\end{lemma}
\begin{proof}
We outline the proof in the case of $Q_{\widehat{m}}\to M^m_{k,l}$, with $Q_{\widehat{l}}$ and $Q_{\widehat{k}}$ being analogous. If a fixed point of the ODE system \eqref{eq:new-generic-SO3-instanton} is not abelian, the condition $\psi'=0$ forces $\xi_\infty = -3m-2A$. Substituting this in the equation for $\xi'=0$, we have
\begin{align*}
    \psi^2_\infty = \frac{1}{{B^+ C^+}} \left( 4 B^+ C^+ - 2A^+ B^+   -2 A^+ C^+  + \frac{2A^+ }{\Delta}\left(kA^+ B^+ +lA^+ C^+ +mB^+ C^+ \right)\right)= D_{\widehat{m}}\;,
\end{align*}
where we used the identity in Lemma \ref{lemma:identities-klm-khatlhatmhat} and the fact that $p_{k,l}^+ = (A^+, B^+,C^+)$ is a fixed point of \eqref{eq:new-evo-spin7} with $ABC\neq0 $. The claim follows by direct algebraic computation.
\end{proof}
\begin{remark}
    Note that, since we are in the AC regime, the fixed points correspond to $G_2$-instantons for the corresponding nearly parallel $G_2$-structure on $N_{k,l}$. In particular, one can easily see that the quantities $D_{\widehat{i}}^\pm$ are proportional to the $\sigma_i(\varphi^\pm)$ introduced by Ball--Oliveira in \cite{BO19}. While Ball--Oliveira are interested in all $D^\pm_i$ and consider arbitrary coclosed $G_2$-structures, we restrict ourselves to $D_{\widehat{k}}^-, D_{\widehat{l}}^-$ and $D_{\widehat{m}}^+$, and avoid the use of the superscript, since there are no ambiguities.
\end{remark} 
We would like to understand the behaviour of $D_{\widehat{i}}$ for a given Aloff--Wallach space $N_{k,l}$. We have
\begin{proposition}
\label{prop: small l/k for positive D^-_k}
    For any $k\geq l\geq 0$, we have $D_{\widehat{m}}>0$ and $D_{\widehat{l}}<0$. Moreover, $D_{\widehat{k}}>0$ if $\frac{l}{k}\in[0,r_*)$ and  $D_{\widehat{k}}<0$ if $\frac{l}{k}\in (r^*,1]$, where $r^*=-\frac{1}{2}+\frac{\sqrt{861+222\sqrt{19}}}{34}$.
\end{proposition}
\begin{proof}
As in Proposition \ref{prop:AC-inequalitites}, we introduce homogenised coordinates $\mathsf{b} = \frac{B}{A}$ and $\mathsf{c} = \frac{C}{A}$. The homogenised fixed points satisfy equation $\mathsf{L}_1=0$, which is equivalent to 
$$
4\mathsf{b}\mathsf{c}=\frac{5}{4}+\frac{5}{4}(\mathsf{b}+\mathsf{c})^2
-\frac{3}{2}(\mathsf{b}+\mathsf{c}).$$
Therefore, we have 
\begin{equation}
    \frac{BC}{A^2}D_{\widehat{m}}=4\mathsf{b}\mathsf{c}-3(\mathsf{b}+\mathsf{c})-1 =\frac{1}{4}+\frac{5}{4}(\mathsf{b}+\mathsf{c})^2-\frac{9}{2}(\mathsf{b}+\mathsf{c}) >0
\end{equation}
where we used $\mathsf{b}+\mathsf{c}>\frac{23}{5}$ for $p_{k,l}^+$ from Proposition \ref{prop:AC-inequalitites}. Similarly, for $p_{k,l}^-$, Proposition \ref{prop:AC-inequalitites} implies that $\mathsf{b}>\mathsf{c}$ and $2\mathsf{b}>\mathsf{b}+\mathsf{c}>1$. Therefore,
\begin{equation}
    \frac{C}{A}D_{\widehat{l}}=4\mathsf{c}-3\mathsf{b}-3\mathsf{b}\mathsf{c}-\mathsf{b}^2 <\mathsf{c}-4\mathsf{b}\mathsf{c} <0.
\end{equation}
Finally, for $D_{\widehat{k}}$, we proceed to prove the existence and uniqueness of its zero $r^*$ and find its value. Using the homogenised equations \eqref{eq:homogenised_AC_equation}, we obtain 
\begin{equation}
    \mathrm{Res}_\mathsf{b}(\mathsf{L}_1,\mathsf{L}_2)=\mathsf{c}\left((-14k + 54l)\mathsf{c}^3 + (-54k - 132l)\mathsf{c}^2 + (-24k + 116l)\mathsf{c} + (36k - 48l)\right)\;.
\end{equation}
The vanishing of $D_{\widehat{k}}$ at $p_{AC}^-$ corresponds to a zero of the polynomial $\mathsf{L}_3=4\mathsf{b}-3\mathsf{c}-3\mathsf{b}\mathsf{c}-\mathsf{c}^2$. We then have 
\begin{equation}
    \mathrm{Res}_\mathsf{b}(\mathsf{L}_1,\mathsf{L}_3)=68\mathsf{c} ^4 - 96\mathsf{c} ^3 + 272\mathsf{c} ^2 - 288\mathsf{c} + 80=0.
\end{equation}
The simultaneous vanishing of $\mathsf{L}_1,\mathsf{L}_2$ and $\mathsf{L}_3$ then requires 
\begin{equation}
\mathrm{Res}_\mathsf{c}(\mathrm{Res}_\mathsf{b}(\mathsf{L}_1,\mathsf{L}_2),\mathrm{Res}_\mathsf{b}(\mathsf{L}_1,\mathsf{L}_3)),
\end{equation}
a fourth-order homogeneous polynomial in $(k,l)$, to vanish. The only positive root is $$r^*= -\frac{1}{2}+\frac{\sqrt{861+222\sqrt{19}}}{34}\approx 0.75773577\;.$$
Since $D_1=\frac{8 \sqrt{5}-14}{3}>0$ for $N_{1,0}$ and $D_3=-12<0$ for $N_{1,1}$, the proof is complete.
\end{proof}
We have no good geometrical interpretation for the existence of the critical $r^*$, or reason for its value. 

Let us now study the local dynamics of the different fixed points. We have 
\begin{lemma}\label{lem:Ab AC end}
The linearisation of the ODE system \eqref{eq:new-generic-SO3-instanton} at $(p_{k,l}^+, \widehat{m}\Lambda^+, 0)$ is positive definite, and it is negative definite at $(p_{k,l}^-, \widehat{k}\Lambda^-, 0)$ as long as $l < r^* k$. For $(p_{k,l}^-, \widehat{k}\Lambda^-, 0)$ with $l >r^*k$ and for $(p_{k,l}^-, \widehat{l}\Lambda^-, 0)$, the linearisation has a positive and a negative eigenvalue.
\end{lemma}
\begin{proof}
We take the case $Q_{\widehat{k}}$, the others being analogous. A short computation yields the linearisation of the extended system. When restricted to the $(\xi, \psi)$-plane, it reads
$$\begin{pmatrix}
    - F & - 2 AB\psi\\
    -\psi & -(2C +3k+\xi)
\end{pmatrix}\;\;,$$
with $F$ as above. At $(p_{k,l}^\pm, \widehat{k}\Lambda^\pm, 0)$ the anti-diagonal vanishes, and we have 
\begin{align*}
    -(2C+3k+\xi)& =  -\frac{1}{F} \left[(2C^\pm+3k)F+ \widehat{k}\widehat{F}\right]\\
    &= - \frac{ABD_{\widehat{k}}}{F}\;.
\end{align*}
where we used the definition of $D_{\widehat{k}}$, Equation \eqref{eq:def-Dk}. The rest of the claim now follows by noting that $F^+<0$ and $F^->0$.
\end{proof}
Using the previous computations, we have 
\begin{lemma}
\label{lem:nonAb AC end}
The linearisation of the extended system at $(p_{k,l}^+, -3m -2A^+, \sqrt{D_{\widehat{m}}})$, restricted to $(\xi,\psi)$ is 
    $$L^{nAb}_m=\begin{pmatrix}
    -F^+& -2B^+C^+ \sqrt{ D_{\widehat{m}}}\\
    -\sqrt{D_{\widehat{m}}}&0
\end{pmatrix}\;. $$
 Hence, the fixed point has a positive and a negative eigenvalue, 
 $$\lambda_\pm= \frac{1}{2}\left(-F^+ \pm \sqrt{{F^+}^2 + 8B^+C^+ D_{\widehat{m}}}\right)\;.$$
The same holds for $D_{\widehat{k}}$ whenever $l < r^* k$.
\end{lemma}
Using the linearised analysis, we can now establish a long time existence result. First, we need the following straightforward lemma:
\begin{lemma} \label{lemma:unbounded_AC_A}
    Let $a^* \in \R$ be the initial condition such that $\Gamma^m(a^*,0)$ is the unique bounded integral curve from Proposition \ref{prop:unique-ab-AC}. Then, for all $a> a^*$, there exists $\delta>0$ sufficiently small such that for $\alpha \in [0, \delta)$, the integral curve $\Gamma^m(a,\alpha)$ eventually enters $\mathcal{A}$. The analogous statement for $Q_{\widehat{k}}$ and $Q_{\widehat{l}}$ holds.
\end{lemma}
\begin{proposition}\label{prop:connections-AC}
Let $N_{k,l}$ be an arbitrary Aloff--Wallach space and consider the associated cohomogeneity one manifold $M^i_{k,l}$ with $i\neq 0$, equipped with its asymptotically conical $\Spin(7)$ metric. Let $Q_{\widehat{i}} \to M^i_{k,l}$ be the corresponding irreducible principal $\SO(3)$-bundle. Then, there exists a connection $A_*$ on $Q_{\widehat{i}}$ given by the integral curve $\Omega^i(c^*)$, such that if $i=l$, or $i=k$ and $l> r^*\;k$, it converges to the abelian instanton over $Q_{\widehat{i}} \rightarrow N_{k,l}$. If $i=m$ or $i=k$ and $l< r^*\;k$, it converges to the unique irreducible $G_2$-instanton over $Q_{\widehat{i}} \rightarrow N_{k,l}$.   

Assume moreover that $\widehat{i} =0 \mod i $. Then, the reducible $\SO(3)$-bundle $P_{\widehat{i}}$ carries a one-parameter family  of connections $A_a$ with $a \in (a^*, \infty)$, corresponding to integral curves $\Gamma^i\big(a, \alpha(a)\big)$ such that for $i=l$, or $i=k$ with $l> r^*\;k$, they converge to the abelian instanton over $Q_{\widehat{i}} \rightarrow N_{k,l}$, and for $i=m$, or $i=k$ with $l< r^*\;k$, they converge to the unique irreducible $G_2$-instanton over $Q_{\widehat{i}} \rightarrow N_{k,l}$.   
\end{proposition}
\begin{proof}
Lemma \ref{lemma:unbounded_AC_A} and Proposition \ref{prop:long-time-critical} guarantee the existence of a one-parameter family of forward complete integral curves contained in $\mathcal{C}$, as well as the exceptional curve $\Omega(c^*)$ which proves existence. 

In the cases $i=m$ or $i=k$ with $l < r^*k$ the abelian connection is the unique limiting point in $\mathcal{C}$, and the claim follows. In the other cases, the limiting set of $\mathcal{C}$ consists of two points. But the abelian limiting point has index zero (resp. two for $Q_{\widehat{k}}$) when restricted to the $(\xi, \psi)$-space, so the corresponding stable manifold is zero dimensional. Hence, there exists at most one solution converging to it. But that is precisely the abelian connection, so every other solution must converge to the non-abelian limiting point.
\end{proof}
Note that for the exceptional case $N_{1,1}$, we have an additional $S^1$-family of solutions, given by the $S^1$-freedom of the ODE system \eqref{eq:reduced SO(3) N11-0}.
\subsection{The exceptional mixed case}
Let us now consider the exceptional case $Q_0 \rightarrow M_{1,1}^{-2}$, with mixed signature from \eqref{eq:reduced SO(3) N11-0} with $\sigma=1$. Our goal is to show that generic solutions in this case all blow up in finite time. 

The $\Spin(7)$-metrics $\gamma^{-2}_{\theta}$ over $M_{1,1}^{-2}$ satisfy $B=C$ and $B\geq A \geq 0$, so the ODE system reduces to 
\begin{subequations}\label{eq:new-mixed-system}
    \begin{align}
        \xi' &= 2B(B-A)\frac{\xi}{3} +  B^2(\chi^2- \psi^2 ) \\
        \psi'&=  (6 -2A -\xi) \psi\\
        \chi'&=  (6 -2A +\xi) \chi
    \end{align}
\end{subequations}
Before we proceed further, note that for initial conditions $(0, \psi_0, \psi_0)$, we have the explicit solution 
\begin{equation}\label{eq:exceptional_sol}
    \xi(s) =0 \qquad \psi(s) =\chi(s) =  \psi_0 \exp \left( \int_{-\infty}^s 6 -2A \;d \widetilde{s}\right)\;,
\end{equation}
which blows up exponentially as discussed in Proposition \ref{prop:vanishing-chi-abelian}. We now have
\begin{lemma}\label{lemma:aux-mixed}
For some time $T$ large enough, all solutions satisfy either $\xi = 0$ and $\xi' =0$ or $\xi \;  \xi' \geq0$.
\end{lemma}
\begin{proof}
First, note that both sets of conditions are preserved by the ODE, with the first one corresponding to the exceptional solution outlined above. Hence, we need to argue that solutions with $\xi \;  \xi' < 0$ eventually satisfy $\xi \, \xi' \geq 0$.

Let us assume that $\xi_0 >  0 , \, \xi_0' < 0$. In particular $\psi_0 > \chi_0$. But $\log \psi'  - \log \chi' = -2 \xi <0$ since $\xi_0> 0 $. Thus, there exists a finite time such that either $\xi =0$ or $\xi' =0$. Note that if $\xi' =0$ first, there exists a previous time such that $\psi = \chi$.

Let $T$ be the first time such that either $\xi=0$ or $\psi = \chi$. Note that they are mutually exclusive, as otherwise it would be with the exceptional solution with initial condition $(0, \psi_0, \psi_0)$. We have the following: 
\begin{itemize}
    \item If $\xi(T)=0$ with  $\psi(T) > \chi(T)$, we get $\xi'(T) <0$, so $\xi(t) \leq 0$ and $\dot \xi <0 $ for all $t \geq T$. 
    \item If $\psi(T) = \chi(T)$  with $\xi(T) >0$, we have $\xi'(T) >0$, hence $\xi(t) > 0$ and $\xi ' >0 $ for all $t \geq T$. 
\end{itemize}
The case $\xi_0<0  , \dot \xi_0>0 $ follows from the next straightforward lemma.
\end{proof}
\begin{lemma}\label{lemma:aux-mixed-2}
    The ODE system \eqref{eq:new-mixed-system} is invariant under the involution $ \tau:  (\xi, \psi, \chi) \mapsto (-\xi, \chi, \psi)$. 
\end{lemma}
We can now prove
\begin{theorem}
    Any solution different from the exceptional solution \eqref{eq:exceptional_sol} with $\psi_0 \chi_0 \neq 0$ blows up in finite time. 
\end{theorem}
\begin{proof}
By virtue of Lemmas \ref{lemma:aux-mixed} and \ref{lemma:aux-mixed-2}, we may assume $\psi, \chi >0$, and $\xi \geq 0$ by taking $t_0$ large enough. Moreover, we assume that our initial conditions are not in the exceptional solution family $(0, \psi_0, \psi_0)$.

Since $A \leq 3$ for some $s_0$ large enough, the ODE system \eqref{eq:new-mixed-system} is bounded below by the system 
\begin{subequations}
    \begin{align}
        \dot \xi &= \frac{\kappa_1}{t^2}(\chi^2 - \psi^2)\,\\
        \dot \psi&= (\kappa_2 - \xi) \psi\\
        \dot \chi&= (\kappa_2 + \xi) \chi
    \end{align}
\end{subequations}
for suitable $\kappa_1, \kappa_2 >0$. We claim that such an ODE system blows up in finite time.

Let $r=  \log (\chi)  - \log(\psi)$, so $\dot r = 2\xi$, with initial conditions $r_0>0$ and $\dot{r_0}>0$. Moreover, we have 
$$\partial_t (\chi \psi) =  2 \kappa_2 \chi \psi \qquad  \implies \qquad \chi(t)\psi(t) = \kappa_0\exp{(2\kappa_2\,t)} \qquad \text{for } \kappa_0>0\;.$$
Using this, we have 
$$ \ddot{r} =  2 \frac{\kappa_1}{t^2} \left(\chi^2 - \psi^2\right) = 2 \frac{\kappa_1}{t^2} \left(\chi \psi\right) \left ( \frac{\chi}{\psi}  - \frac{\psi}{\chi}\right)= \frac{4 \kappa_0 \kappa_1 }{t^2} \exp{(2\kappa_2\,t)}  \sinh{ r }\;. $$
Setting $y= \exp r$, we have
$$\ddot{ y } = \left( \ddot{r}+ (\dot r)^2 \right) y  \geq \frac{4 \kappa_0 \kappa_1}{t^2} \exp{(2\kappa_2\,t)}  \sinh{ r } \exp {r} \geq \kappa y^2\;$$
for some suitable $\kappa >0$, since we may take $t \geq t_0$ large enough, and $r_0, \dot{r_0}>0$. Finite time blow-up is now straightforward.
\end{proof}
\section{Bubbling phenomena}\label{sec:bubbling}
We now turn to compactness and compactification of the families constructed in the previous section. We first study bubbling phenomena, which, from the dynamical systems perspective, correspond to families
of integral curves breaking at an intermediate fixed point. The following lemma provides the framework for this analysis.

\begin{lemma}\label{lem: bubble existence}
Consider an autonomous system $x'=f(x)$ on $\R^n$, and let $p,q,r$ be fixed points, with $q$ hyperbolic. Suppose that $\widehat{\Gamma}$ is an integral curve joining $p$ to $q$, and that $\{\Gamma_a\}$ is a family of integral curves joining $p$ and $r$ which converges to $\widehat{\Gamma}$ in a neighbourhood of $p$ as $a\to a^*$.

Then there exist a sequence $a_n\to a^*$, shifts
$s_n\to+\infty$, and a maximal integral curve $\Omega: (-\infty, S) \to \R^n$, emanating from $q$ such that, for all compact intervals $I\subset\R$
and $J\subset(-\infty,S)$, we have
\begin{equation}
 \lim_{n \to \infty} d_H\bigl(\Gamma_{a_n}(I),\widehat{\Gamma}(I)\bigr) +
d_H\bigl(\Gamma_{a_n}(s_n+J),\Omega(J)\bigr) =0\;
\end{equation}
uniformly.
\end{lemma}
\begin{proof}
Let $U$ be an open neighbourhood of $q$, on which the Hartman--Grobman theorem applies, and consider $B\subseteq U$ a closed ball centred at $q$. Choose a sequence $a_n \to a^*$ and let $\Gamma_n$ be the corresponding sequence of integral curves. Since the initial conditions of $\Gamma_n$ converge to those of $\widehat{\Gamma}$, we may assume that $\Gamma_n$ converges uniformly to $\widehat{\Gamma}$ on compact time intervals, up to reparametrisation. 

Therefore, for sufficiently large $n$, $\Gamma_n $ enters $B$, and since $\Gamma_n$ converges to $r$, exits the ball at some time $s_n \in \R$. Let $x_n =  \Gamma_n(s_n) \in \partial B$ be the exit point of $\Gamma_n$ from $B$. Since $\Gamma_n$ converges uniformly to $\widehat{\Gamma}$ on compact time intervals, we have $ s_n \to \infty$. Passing to a subsequence, the exit points $\{x_n\}$ converge to some $x \in \partial B$ by compactness. 

Let $\Omega(s)$ be the integral curve through $x$. Then, by continuous dependence on initial conditions, $\Omega(s)=\lim\limits_{n\to\infty}\Gamma_n(s_n+s) $ on compact time intervals. For $s<0$, the points $\Gamma_n(s_n+s)\) lie before the exit from the ball and, as $n\to\infty$, they follow $\widehat{\Gamma}$ approaching \(q\). Thus, $\lim\limits_{s\to - \infty} \Omega(s) =q$. 
    
Therefore $\Omega$ is the integral curve of the statement, where the Hausdorff distance convergence follows from continuity with respect to initial conditions.
\end{proof}
The main goal of this section is to understand the bubbling phenomena that occur in the families $\Gamma(a, \alpha)$ for both the AC and ALC asymptotics. To carry out the details, we take $i=m$, as usual. The reader can easily verify that there are no distinctions for $i=k,l$ other than the time reversal considerations.  We begin with the AC case, as it's simpler. 

Let $\alpha: (a^*, \infty) \to (0, \pi)$ be the function parametrising the non-abelian $\Spin(7)$-instanton curves $\Gamma(a,\alpha)$. The key observation is
\begin{proposition}\label{prop:alpha-limit}
  We have 
   $$ \lim_{a \to \infty} \alpha (a) =  \arctan \left( - \frac{m}{2 \Delta}\right)\;. $$
\end{proposition}
In order to prove this, it is convenient to take a step back to introduce a natural compactified coordinate for $a$. We have been considering solutions to the extended ODE system with initial conditions of the form
$$\Gamma = \left(p_0^m, \frac{\widehat{m}^2}{m}, 0\right) +\left[\cos (\theta)  v_1 + \sin (\theta) v_2 + a \left( \cos( \alpha)v_3 +\sin(\alpha)v_4\right)\right] e^{-\frac{4 \Delta}{m}s} + o\left(e^{-\frac{4 \Delta}{m}s}\right)\;.$$
This choice was made to fix the reparametrisation freedom of $s$ for the $\Spin(7)$-metric. But, we can consider the more general initial conditions
\begin{equation}\label{eq:generalised-ic-gamma}
    \Gamma  = \left(p_0^m, \frac{\widehat{m}^2}{m}, 0\right) + \left[\cos(\beta) \left(\cos (\theta)  v_1 + \sin (\theta) v_2\right) + \sin(\beta)\left( \cos( \alpha)v_3 +\sin(\alpha)v_4\right)\right] e^{-\frac{4 \Delta}{m}s} + o\left(e^{-\frac{4 \Delta}{m}s}\right)\;.
\end{equation}
Note that, as long as $\beta \neq \frac{\pi}{2}$, we can find a reparametrisation of $s$ mapping one set of initial conditions to the other. The following is immediate:
\begin{lemma}
For $\beta = \frac{\pi}{2}$, integral curves of the extended system are $\left(0, -\frac{2 \Delta}{m}, -\frac{2\Delta}{m}, \xi(s), \psi(s)\right)$, where $\xi$ and $\psi$ solve the reduced system
\begin{subequations}\label{eqn:compactifying-subsystem}
    \begin{align}
        \xi'&=-\frac{4\Delta}{m}\xi+\frac{4\Delta\widehat{m}^2}{m^2}-\frac{4\Delta^2}{m^2}\psi^2\\
        \psi'&=-(3m+\xi)\psi\;. 
    \end{align}
\end{subequations}
\end{lemma}
The strategy is to use the subsystem above to study the limit $\beta \to \frac{\pi}{2}$, exploiting the continuity of ODE solutions with respect to initial conditions. First, we have
\begin{proposition}\label{prop:compactifying-integral-curve}
    The subsystem \eqref{eqn:compactifying-subsystem} has the invariant line $$\psi+\frac{m}{2\Delta}\xi-\frac{\widehat{m}^2}{2\Delta}=0\;.$$
    In particular, the integral curve  
     $$ \mathbb{B}(s) = \left( \frac{\widehat{m}^2}{m} +  \frac{1}{e^{\frac{4 \Delta}{m} s} - \frac{m }{4\Delta}}, \,\left( -\frac{m}{2\Delta}\right)  \frac{1}{e^{\frac{4 \Delta}{m} s} - \frac{m }{4\Delta}}\right) \qquad \qquad s \in \R$$
    joins the fixed points $\left(p_0^m, \frac{\widehat{m}^2}{m},0\right)$ and $(p_0^m , -3m , 2)$.  
\end{proposition}
\begin{proof}
    We have 
    \begin{equation}
        \begin{split}
        \left(\psi+\frac{m}{2\Delta}\xi-\frac{\widehat{m}^2}{2\Delta}\right)'&=-(3m+\xi)\psi+\frac{m}{2\Delta}\left(-\frac{4\Delta}{m}\xi+\frac{4\Delta\widehat{m}^2}{m^2}-\frac{4\Delta^2}{m^2}\psi^2\right)\\
        &=-\frac{2\Delta}{m}(2+\psi)\left(\psi+\frac{m}{2\Delta}\xi-\frac{\widehat{m}^2}{2\Delta}\right). \quad
        \end{split}
    \end{equation}
It is now straightforward to check that $\mathbb{B}$ is an integral curve of the reduced system joining the two fixed points. 
\end{proof}
\begin{proof}[Proof (of Proposition \ref{prop:alpha-limit})]
Set $\alpha^*= \arctan\left( \frac{-m}{2 \Delta} \right)$, the slope of the critical line from the previous proposition. The $\mathcal{A}, \mathcal{B}, \mathcal{C}, \mathcal{F}$ sector decomposition restricts naturally to the system \eqref{eqn:compactifying-subsystem}. 

Consider $\widetilde{\Gamma}(\alpha)$ the integral curves emanating from $\left(\frac{\widehat{m}^2}{m}, 0\right)$ for the limiting system \eqref{eqn:compactifying-subsystem}. Then, for $s$ large enough, we have $\widetilde{\Gamma}_\alpha(s) \in \mathcal{A}$ for $\alpha< \alpha^*$ and  $\widetilde{\Gamma}_\alpha(s) \in \mathcal{B}$ for $\alpha>\alpha^*$. The claim follows by the continuity of ODE solutions with respect to initial conditions. 
\end{proof}
We now have the desired bubbling result:
\begin{proposition}\label{prop:bubbled-off-connection AC}
Let $\Omega(c^*)$ be the curve from Proposition \ref{prop:connections-AC} and $\mathbb{B}$ be the integral curve joining $\left(p_0^m, \frac{\widehat{m}^2}{m},0\right)$ and $(p_0^m , -3m , 2)$ from Proposition \ref{prop:compactifying-integral-curve}. Consider $\Gamma(a, \alpha)$ the one-parameter family of instantons of Proposition \ref{prop:connections-AC}. Then, as $a\to \infty$, the Hausdorff distance between $\mathbb{B} \cup \Omega^m(c^*)$ and $\Gamma(a, \alpha)$ tends to zero. 

In particular, the family of connections $A_a$ associated with $\Gamma^m(a, \alpha)$ converges pointwise, away from $Z_0$, to the connection $\widetilde{A}_{c^*}$ associated with $\Omega^m(c^*)$ .
\end{proposition}
\begin{proof}
 By the local theory around $(p_0^m, -3m, 2)$, the curve $\Omega$ from Lemma \ref{lem: bubble existence}, must be $\Omega^m(c)$ for some $c \in \R$.
 
 But since $\Gamma_a$ remains in $\mathcal{C}$ for all time, so does $\Omega(c)$, hence $c=c^*$. The Hausdorff distance claim follows by the uniform convergence along compact intervals, the inclination Lemma applied to a neighbourhood of $(p_0^m, -3m, 2)$, and the fact that $\Gamma(a, \alpha)$ and $\Omega(c)$ have the same endpoint.
 
 To prove the pointwise convergence of the connections, it suffices to estimate the time it takes $\Gamma_a$ to travel from $(p_0^m, \frac{\widehat{m}^2}{m}, 0)$ to a neighbourhood of $(p_0^m, -3m,2)$. If we fix the time parametrisation coming from the curve $\Omega^m(c)$, it is clear it tends to $-\infty$ as $a \to \infty$ (or rather $\beta \to \frac{\pi}{2}$). 
\end{proof}
To justify the bubbling terminology, we must show that this limiting process indeed exhibits the gauge-theoretic bubbling behaviour conjectured by Tian \cite{Tian00} and Tao--Tian \cite{TT04}(cf. \cite{LO18}).

Consider $A_a$ the one-parameter family of connections emanating from $\left(p_0^m, \frac{\widehat{m}^2}{m},0\right)$, with $a> a^*$. Let  $A_{*}$ be the corresponding bubbled-off connection emanating from $\left(p_0^m,-3m, 2\right)$ corresponding to the integral curve $\Omega_*$. We denote their curvatures by $ F_a$ and $F_*$ respectively.
\begin{proposition}\label{prop:energy-conservation}
 We have 
\begin{equation}
    \lim_{a \to +\infty} \left( \abs{F_a}^2 -  \abs{F_*}^2  \right)\vol_g \rightharpoonup  \frac{8 \pi^2}{\abs{m}} \delta_{Z_0}
\end{equation}
where $Z_0 \cong \C P^2$ is the singular orbit and the convergence is understood in the sense of currents. 
\end{proposition}
We need some technical lemmas first.
\begin{lemma}\label{lemma:curve-length}
Under the assumptions of Lemma \ref{lem: bubble existence}, consider the ball of radius $\varepsilon>0$ centred at $q$, $B_\varepsilon$. We have the estimate 
    \begin{equation}
        \operatorname{Length}( \Gamma_a \cap B _\varepsilon) \leq C \varepsilon
    \end{equation}
uniformly on $a$ for some constant $C$.
\end{lemma}
\begin{proof}
    Since $q$ is hyperbolic, we may choose local coordinates $(u,v)$ with $(u,0)$ and $(0,v)$ parametrising the stable and unstable manifolds respectively. After suitably shrinking the neighbourhood, the coordinates of the curve will satisfy the differential inequalities
    $$ \frac{d}{ds} \abs{u(s)} \leq - c \abs{u} \qquad \qquad \frac{d}{ds} \abs{v(s)} \geq c \abs{v}$$
    for some $c>0$. 
    
    Consider the rectangle $R_\varepsilon\coloneqq \left\{ \abs{u} \leq  \varepsilon\;, \abs{v}\leq \varepsilon\right\}$. For $\varepsilon>0$ sufficiently small, let $[s_-,s_+]$ be the unique time interval of $\Gamma_a(s)$ through $R_\varepsilon$.
    
    Since the vector field vanishes at the origin, we also have
    $$ \left| \frac{d}{ds}\Gamma_a(s)\right| \leq K\bigl(|u(s)|+|v(s)|\bigr)$$
with $K$ independent of $a$ and $\varepsilon$. Integrating, we have 
$$\int_{s_-}^{s_+}|u(s)|\,ds \leq \frac{|u(s_-)|}{c}\leq\frac{\varepsilon}{c},
    \qquad \text{and}
\qquad  \int_{s_-}^{s_+}|v(s)|\,ds
       \leq \frac{|v(s_+)|}{c}\leq\frac{\varepsilon}{c}.
$$
Hence
\[
    \operatorname{Length}\bigl(\Gamma_a\cap B_\varepsilon \bigr)
       \leq \frac{2K}{c}\,\varepsilon. \qquad \qedhere
\]
\end{proof}
\begin{lemma}\label{lemma:curvature-density}
In terms of $(A,B,C)$, and a connection $A= (\xi, \psi)$, we have
\begin{equation}
    \abs{F}^2 \vol_g=  \frac{f^4}{A^2} \left(  \frac{2}{BC} (\xi+3m)^2\psi^2 + \left( \frac{m\xi  - \widehat{m}^2 + \Delta \psi^2}{\Delta} \right)^2+ O(A)\right) ds  \wedge e_{1234567}
\end{equation}
\end{lemma}
\begin{proof}
We have $\abs{F}^2 \vol_g = -\frac{1}{4} \operatorname{Tr}\left(F \wedge * F \right)$. In our basis $\{T_i\}$, satisfying $[T_i, T_j]= \varepsilon^{ijk} T_k$, we have $\abs{T_i}^2 =\frac{1}{2}$. Using the expression for the curvature of a cohomogeneity one connection, Equation \eqref{eq:curvature-generic-SO3}, we have 
\begin{align*}
   -\frac{1}{4}\operatorname{Tr}\left(F  \wedge  * F \right)=& \frac{1}{2}\left[ \left( \frac{\dot{\xi}}{f}\right)^2 + \left(\frac{m\xi - \widehat{m}^2+ \Delta \psi^2}{\Delta a^2}\right)^2 +\left(\frac{l\xi - \widehat{m}\,\widehat{l}}{\Delta b^2}\right)^2+\left(\frac{k\xi - \widehat{m\,}\widehat{k}}{\Delta c^2}\right)^2\right] \omega_{01234567}  \\
    &+\left[ \left( \frac{\dot{\psi}}{a}\right)^2 + \left(\frac{2 \psi  }{bc}\right)^2 +\left(\frac{(3m +\xi)\psi}{af}\right)^2\right]\omega_{01234567} \\
    = & \left[\left(\frac{m\xi - \widehat{m}^2+ \Delta \psi^2}{\Delta A }\right)^2 +\left(\frac{l\xi - \widehat{m}\,\widehat{l}}{\Delta B}\right)^2+\left(\frac{k\xi - \widehat{m\,}\widehat{k}}{\Delta C}\right)^2 + \frac{(m\xi - \widehat{m}^2+ \Delta \psi^2)(l\xi - \widehat{m}\,\widehat{l})}{\Delta^2 AB}\right. \\
    &\;\left.+ \frac{(m\xi - \widehat{m}^2+ \Delta \psi^2)(k\xi - \widehat{m}\,\widehat{k})}{\Delta^2  AC}+ \frac{(l\xi - \widehat{m}\,\widehat{l})(k\xi - \widehat{m}\,\widehat{k})}{\Delta^2 BC}\right] f^4 ds\, e_{1234567} \\
    &+\frac{2}{BC}\left[4 \psi^2+\frac{(3m +\xi)^2\psi^2}{A^2} + \frac{2(3m +\xi)\psi^2}{A} \right] f^4 ds\, e_{1234567} \;.
\end{align*}
Factoring $A^{-2}$, and discarding higher order terms, the claim follows.
\end{proof}
\begin{proof} [Proof of Proposition \ref{prop:energy-conservation}]
Let $A_a$ be a smooth connection over a bundle over $M^m_{k,l}$. Consider the associated density 1-form 
\begin{equation*}
    \mathcal{D}: (A,B, C, \xi, \psi) \mapsto  \int_{N_{k,l}} \abs{F}^2 \vol_g \;,
\end{equation*}
corresponding to the fibre integration of the $L^2$-norm of the curvature along the (generic) orbits of the $\SU(3)$ action. By Lemma  \ref{lemma:curvature-density}, $\mathcal{D}_a$ is continuous for every $\Spin(7)$-metric. 

For convenience, let $\tau= \exp \left(\frac{-2\Delta}{m}s\right)$, so the singular orbit is at $\tau=0$, and $f \sim \tau$ and $A \sim \tau^2$ as $\tau \to 0$. More concretely, in a neighbourhood of the singular orbit, we have  
$$f^4/A^2  =  \frac{f^2}{a^2} (bc)^2   =  \left(\frac{2 \Delta}{m}\right)^2 \frac{\vol (Z_\tau)}{\vol(Z_0) }  =   \left(\frac{2 \Delta}{m}\right)^2 \frac{2}{\pi^2} \vol (Z_\tau)\;,$$
where $Z_\tau$ is the image of $\C P^2$ at time $\tau$ under the exponential map, and we used Lemma \ref{lemma:volume_normalisation}. Therefore, we see that the density $\mathcal{D}_a$ remains bounded in the limit $\tau \to 0$, and we can extend it to $\tau=0$.

Let $K\subset M^i_{k,l}$ be a compact set in $M^i_{k,l}$. It suffices to prove
        $$\lim_{a \to \infty} \int_K \abs{F_a}^2 -  \abs{F_*}^2 \,\vol_g = \frac{8\pi^2}{\abs{m}}  \mathsf{H}^{4}(K \cap{Z}_0) \,,
        $$
        where $\mathcal{H}^{4}$ is the induced 4-dimensional Hausdorff measure. 
        
        Since the connections are $\SU(3)$-invariant, we can take $K$ to be $\SU(3)$-invariant, so it matches the orbits of the group action. In particular, letting $I = \operatorname{proj}(K) \subseteq \R$, we have 
            $$\lim_{a \to \infty} \int_K \abs{F_a}^2 -  \abs{F_*}^2\; \vol_g  = \lim_{a \to \infty} \left( \int_{\Gamma_a(I)} \mathcal{D} - \int_{\Omega_*(I)}\mathcal{D}\right)\;.$$
        If $0 \notin I$, we have that $\Gamma_a(I)$ converges uniformly over compact sets to $\Omega_*(I)$ by Lemma \ref{lem: bubble existence}. Hence, Lebesgue dominated convergence yields
        $$\lim_{a \to \infty} \int_K \abs{F_a}^2 -  \abs{F_*}^2 \vol_g = 0\;.$$
        It remains to prove the second assertion. Let $I_\epsilon= (0, \epsilon)$ for $\epsilon>0$. Since $\mathcal{D}$ has a removable singularity at $\left(0, -  \frac{2\Delta}{m}, -\frac{2\Delta}{m}\right)$, and using Lemma  \ref{lemma:curve-length}, we have 
        $$\lim_{\epsilon \to 0} \lim_{a \to \infty}  \left( \int_{\Gamma_a(I_\epsilon)} \mathcal{D} - \int_{\Omega_*(I_\epsilon)}\mathcal{D}\right) =  \int_{\mathbb{B}} \mathcal{D}\;. $$
        Now, along the bubble $\mathbb{B}$, we have $A=0$, $B=C = -\frac{2\Delta }{m}$ and $\frac{f^4}{A^2} = \left(\frac{2\Delta
        }{m}\right)^2 \frac{2}{\pi^2} \vol( Z_0)$, where this last equality should be understood as the limiting value of the removable singularity. To lighten notation, we set $q= \frac{4\Delta}{m}$. We have
        \begin{align*}
            \int_{\mathbb{B}} \mathcal{D} &=  \frac{f^4}{2A^2} \vol_1(N_{k,l})\int_{\mathbb{R}} \frac{16}{q^2} (3m+\xi)^2 \psi^2 + 2 \left(\frac{m \xi - \widehat{m}^2 + \Delta \psi^2}{\Delta}\right)^2 ds\\
            &= \frac{12q^4\pi^2}{\Delta}\vol(Z_0) \int_\R  e^{2qs} \left(q e^{q s} - 1 \right)^{-4}ds\\
            &= -\frac{12q\pi^2}{\Delta}\vol(Z_0) \int_0^{-1} u(u+1) du=  -\frac{ 2q\pi^2}{\Delta} \vol (Z_0) \\            
            &= - \frac{8\pi^2}{m}\vol(Z_0) \;,
        \end{align*}
        where we used the change of variables $u = \left(qe^{qs} - 1\right)^{-1}$.
\end{proof}
Proposition \ref{prop:energy-conservation} should be understood as a statement on $L^2$-energy conservation along the bubbling process. The energy of the bubble $\mathbb{B}$ suggests that the connections are ``bubbling off" an anti-self-dual (ASD) instanton along the normal bundle of the singular orbit $Z_0$. We will show that this is precisely the case. 

We start with a more general statement that will prove useful in the next section. First, we introduce some notation to for the rescaling that we will need. Take $p \in Z_0$ and $\delta> 0$. We denote as  
\begin{align}
    r_\delta : N_p Z_0  &\to N_p Z_0  \label{eq:normal-fibre-rescaling}\\
      (p, f) & \mapsto  (p, \delta f) \nonumber
\end{align}
the fibre rescaling map of the normal bundle of the singular orbit by a factor $\delta$. We have
\begin{proposition} \label{prop:limmit-ASD}
    Let $A_a(t)$ be the family of connections corresponding to the integral curves $\Gamma(a, \alpha)$ for a fixed $\alpha$. Take $\lambda>0$ and consider the family of rescalings $r_{\delta_a}$ with $\delta_a =  \sqrt{\frac{\lambda}{a}}$. 
    
    The rescaled limit connection
    $$ A^\lambda_\alpha =  \lim_{a \to \infty} r^*_{\delta_a} (A_a)$$
    is a spherically symmetric ASD connection with respect to the flat metric on $\R^4/\Z_{\abs{i}}$. 

    Similarly, let $A_c(t)$ be the family of connections corresponding to the integral curves $\Omega(c)$. Then, for $\delta_c =  \left(\frac{\lambda}{12c\,m^2}\right)^{\frac{1}{4}}$, the rescaled limit connection is a spherically symmetric ASD connection with respect to the flat metric.   
\end{proposition}
\begin{proof}
    Take $\{dr, \sigma_i\}$ as a basis of $1$-forms in a neighbourhood of the singular orbit, with $\sigma_i$ from Lemma \ref{lemma:normal_metric}.
    Then the connection one-form becomes 
    $$A_a = (\xi_a e_7 + \widehat{m}e_8)\otimes T_1 + \psi_a(e_1\otimes T_2 + e_2 \otimes T_3) = \widetilde{\xi_a} \sigma_1 \otimes T_1 + \psi_a (\sigma _2\otimes T_2 + \sigma_3 \otimes T_3)\;,$$
    with $\widetilde{\xi_a} = \frac{\widehat{m}^2 - m \xi_a}{2 \Delta}$. In a neighbourhood of the singular orbit, we have the asymptotics 
    $$ \widetilde{\xi_a} = - \frac{m\, a}{2 \Delta} \cos (\alpha) t^2  + o(t^2) \qquad \qquad \psi_a =  a \sin( \alpha) t^2 +o(t^2)\;.$$
    For fixed $\lambda>0$, set $\delta_a= \sqrt{\frac{\lambda}{a}}$ and consider the functions
    $$ \Xi = \lim_{a \to \infty} r^*_{\delta_a} \left( \widetilde{ \xi_a} \right) \qquad   \qquad \Psi = \lim_{a \to \infty} r^*_{\delta_a} \left( \psi_a\right)   $$
    Undoing the change of variables $dt = f ds$, using $f = -\frac{2\Delta}{m} t +o(t)$, the pair $(\Xi, \Psi)$ will satisfy the limiting ODE system.
    \begin{subequations}\label{eq:normal-limit}
    \begin{align}
        \tau \frac{d}{d\tau } \Xi  & = 2 \Xi - \Psi^2\\
        \tau \frac{d}{d\tau } \Psi & = (2- \Xi)\Psi
    \end{align} 
    \end{subequations}
    with asymptotics 
    $$ \Xi =  - \frac{m\lambda \cos(\alpha)}{2\Delta}\tau^2 +o(\tau^2) \qquad \qquad \Psi  = \lambda \sin(\alpha) \tau^2 +o(\tau^2)\;.$$
Geometrically, the system \eqref{eq:normal-limit} is precisely the  $\U(2)$-invariant ASD system \eqref{eq:spherical-Nahm} restricted to our set-up, with 
 $$ X_1  =  \Xi \; T_1 \qquad X_2  = \Psi \, T_2 \qquad  X_3 = \Psi\, T_3 \;. $$
 Similarly, for $\Omega(c)$, we have the asymptotic expansion
 $$ \widetilde{\xi_c} = 2 + 2c\,m^2 \tau^4+o(\tau^4) \qquad  \qquad \psi_c  = 2 -  c\,m^2 \,\tau^4 +o(\tau^4)\;.$$
 Taking $\delta_c=  \left(\frac{\lambda}{12 m^2\, c}\right)^{\frac{1}{4}}$, the rescaled limits will satisfy the $\U(2)$-invariant ASD equations \eqref{eq:normal-limit}, with asymptotics 
\begin{equation}\label{eq:nahm-asymptotics-omega}
\Xi = 2 + \frac{\lambda}{6} \tau^4 +o(\tau^4)\qquad \qquad \Psi = 2 -\frac{\lambda}{12} \tau^4 +o(\tau^4)\;.\qedhere
\end{equation}
\end{proof}
It now suffices to show that
\begin{proposition}
    The limit ASD connection for $\alpha =  \arctan \left( -\frac{m}{2 \Delta}\right)$ is the standard radial instanton of scale $\lambda$:
    $$ A^\lambda  =  \frac{2\lambda \tau^2}{ 1+ \lambda \tau^2} \sum_{i=1}^3 \sigma_i \otimes T_i\;.$$
\end{proposition}
\begin{remark}
    The additional factor $2$ appearing is due to our choice of basis $\{T_i\}$ satisfying $[T_i, T_j] = \varepsilon_{ijk} T_k$ rather than the usual choice $[T_i, T_j] = 2\varepsilon_{ijk} T_k$ in the literature.
\end{remark}
\begin{proof}
    Using the classical trigonometric identities 
    $$ \cos( \arctan(x)) =  \frac{1}{\sqrt{x^2+1}} \qquad \qquad \sin( \arctan(x)) =  \frac{x}{\sqrt{x^2+1}}\;,$$ 
    we see that the asymptotics of the pair $(\Xi, \Psi)$ becomes
    $$\Xi =  2\lambda \tau^2 +o(\tau^2) \qquad \quad \Psi = 2\lambda \tau^2 +o(\tau^2)\;,$$
    for a suitable new choice of $\lambda$. 

    Note that the ODE \eqref{eq:normal-limit} preserves the condition $\Xi=\Psi$, so we are reduced to the ODE
    $$ \tau \frac{d}{d\tau}\Xi = (2- \Xi)\Xi\;,$$
    which is the well-known logistic equation, up to the reparametrisation $\widetilde{\tau\,} = \log (\tau)$. It is straightforward to verify that 
    \begin{equation}
        \Xi(\tau)  = \Psi(\tau) =  \frac{2 \lambda \tau^2}{1 +\lambda \tau^2 }
    \end{equation}
    is the sought-after solution.
\end{proof}
So we have established all the expected bubbling properties for the limit $\Gamma(a, \alpha) \to \Omega(c^*)$. 

Let us now focus our attention to the ALC case. Recall that, in this case, we similarly have a function $\alpha: [a^*_\theta, +\infty) \to [0, \pi]$ that parametrises a two-dimensional family of solutions as follows. For $a > a^*_\theta$ and $\beta <\alpha$, the integral curve $\Gamma(a, \beta)$ converges exponentially to $(\xi_\infty,0)$ with $\xi_\infty>-3m$. We also have the boundary curve: For $a\geq  a^*_\theta$, the integral curves $\Gamma(a, \alpha)$ converge polynomially to $(-3m, 0)$. 

Replicating the proof of Proposition \ref{prop:alpha-limit} and Proposition \ref{prop:bubbled-off-connection AC} from the AC case, we have that $\alpha$ converges to $\alpha^* =\arctan\left( -\frac{m}{2\Delta}\right)$ in the limit  $a \to \infty$, and the boundary curve $\Gamma(a\cos(\alpha), a \sin (\alpha))$ converges to $\mathbb{B} \cup \Omega_\theta(c^*)$.

It is then a natural question whether the rest of the $\Omega(c)$ family can be realised as a bubbling limit of suitable $\Gamma(a, \beta)$. We shall show that this is indeed the case. First, we need the following ordering lemma.
\begin{lemma}\label{lemma:ordering}
The partial ordering 
$$(\xi_1,\psi_1)\succeq(\xi_2,\psi_2)
\quad\Longleftrightarrow\quad
\xi_1\geq\xi_2,\quad\psi_1\leq\psi_2$$
is preserved by the $\Spin(7)$-instanton ODE system \eqref{eq:new-generic-SO3-instanton} for $\widehat{m}$. For $\widehat{k}$ and $\widehat{l}$, the inequality for $\psi_i$ must be reversed.
\end{lemma}
\begin{proof}
We give the proof for $i=m$, the other cases being analogous. Once more, we argue by barriers. Along the boundary $\xi_1= \xi_2$, we have 
\[(\xi_1 - \xi_2)'  = -  BC( \psi_1^2 - \psi^2_2) \geq 0\;.\]
Similarly, along the boundary $\psi_1= \psi_2$, we have
\[ (\psi_1 - \psi_2)' =  - \left(\xi_1 -\xi_2\right) \psi_1 \leq 0\;. \qquad \qedhere\]
\end{proof}
In particular, we get the following useful characterisation. 
\begin{corollary}$~$
For $\beta_1 \leq \beta_2 \leq \frac{\pi}{2}$, we have $\Gamma^m_\theta\left( a, \beta_1 \right)(s) \succeq \Gamma^m_\theta\left( a, \beta_2 \right)(s)$ for all $s \in \R$. For $d\geq c \geq 0$, we have $\Omega^m_\theta(d)(s) \succeq \Omega^m_\theta(c)(s)$.
\end{corollary}
\begin{proposition}
For every $c>c_*$, there exists $a_0 \in \R^+$ large enough, and a curve $\beta^c: [a_0, +\infty] \to [0, \alpha_*]$ such that 
$$ \Gamma (a, \beta^c(a)) \xrightarrow{a \to \infty }\Omega(c)\;,$$
where the convergence is along compact sets away from the singular orbit. 
\end{proposition}
\begin{proof}
Fix $c>c_*$ and a finite time $s_0$, and set $L = \xi_{\Omega(c)}(s_0)$. The key insight is that, for $a$ large enough, there exists a curve $\beta^c(a)$ such that the integral curve $\Gamma(a, \beta^c)$ satisfies $\xi_{\Gamma_{a,\beta}} (s_0) = L$.

Indeed, on one side, we have
$$ L > \xi_{\Omega(c^*)}(s_0) = \lim_{a\to \infty} \xi_{\Gamma_{a,\alpha}} (s_0)\;,$$
while on the other, we have  
$$  \xi_{\Gamma_{a,0 }} (s_0)  \xrightarrow{\, a \; \to \, \infty \, } + \infty \;.$$
Continuity (and monotonicity) guarantee the existence (and uniqueness) of $\Gamma(a, \beta^c)$ for a given $s_0$, $L$ and $a$. It remains to prove that our candidate family converges to $\Omega(c)$.

 By construction, we have $(\xi_{{a,\beta^c}}, \psi_{{a,\beta^c}})\succeq (\xi_{{a,\alpha(a)}}, \psi_{{a,\alpha(a)}})$, so Lemma \ref{lemma:ordering} bounds $\psi_{{a,\beta^c}}$ on compact subsets of \((-\infty,s_0]\). The ODE for $\xi$ with the starting condition $\xi(s_0)= L$ bounds $\xi$. 
 
 Hence, the family $(\xi_{{a,\beta^c}}, \psi_{{a,\beta^c}})$ converges smoothly on compact subsets to a solution $(\widehat{\xi\,}, \widehat{\psi})$, which satisfies $\widehat{\xi\,}(s_0) = L$, and 
 \begin{equation}
     \widehat{\xi\,}(s) \geq \lim_{a \to \infty} \xi_{a,\alpha(a)}(s) = \xi_{c^*}(s) \qquad\qquad \widehat{\psi\,}(s) \leq \lim_{a \to \infty} \psi_{a,\alpha(a)}(s) = \psi_{c^*}(s)\;.
 \end{equation} 
In particular, we have 
$$ \liminf_{s \to -\infty} \widehat{\xi\,} \geq  -3m \qquad \qquad \limsup_{s \to -\infty} \widehat{\psi} \leq 2\;.$$
The proof is essentially complete, provided we show that $\limsup_{s \to -\infty} \widehat{\xi\,} < \infty$. In that case, the solution $( \widehat{\xi\,}, \widehat{\psi}\,)$ must emanate from $(-3m ,2)$, since it is the only limiting point of the ODE over $p_0^m$ with $\xi \geq -3m$. Therefore the solution $(\widehat{\xi\,}, \widehat{\psi})$ must be $\Omega(\widehat{c})$ for some $\widehat{c}$. Uniqueness of solutions for the extended system $(A,B,C, \xi, \psi)$ and monotonoicity forces $\widehat{c}=c$.

In order to see that the limit superior of $\widehat{\xi\,}$ is bounded, note it satisfies the ODE with 
$$ \xi' =  -F(s) \xi(s) +g(s)\;,$$
where $F(s) \to \frac{4\Delta}{m}$ and $\abs{g(s)} \to C < +\infty$ as $s \to -\infty$, since $\limsup_{s \to -\infty}\widehat{\psi}\leq 2 <\infty$. In particular, taking $H \in \R$ large enough yields $\xi'|_{\xi=H} >0$ for $s$ sufficiently negative, so $H$ acts as a backwards barrier, as needed.
\end{proof}
We have established that every member of the one-parameter family $\Omega(c)$ can be realised as a suitable limit $\Gamma\left(a, \beta^c(a)\right)$. By comparing with the limiting system \eqref{eqn:compactifying-subsystem}, it follows that necessarily $\beta^c(a) \xrightarrow{a \to \infty} \alpha ^*$, and so we have 
\begin{proposition}\label{prop:bubbling-ALC}
    For every $c\geq c^*$, there exists a curve $\beta^c:[a^*, \infty)\to [0, \pi]$ such that 
    $$ \lim_{a \to \infty} d_H \big(\mathbb{B}\cup \Omega^m(c), \Gamma(a, \beta^c (a))\big) =0 \;.$$
    In particular, the family of connections $A_{a, \beta^c}$, associated to $\Gamma^m(a, \beta^c(a))$, converges pointwise to the connection $\widetilde{A}_{c}$, associated to the integral curve $\Omega^m(c)$ away from $Z_0$. Moreover, the bubbling properties of Propositions \ref{prop:energy-conservation} and \ref{prop:limmit-ASD} hold. 
\end{proposition}

\section{The moduli spaces of \texorpdfstring{$\Spin(7)$}{Spin(7)}-instantons}
We wrap up our discussion by considering the different moduli spaces of cohomogeneity one $\Spin(7)$-instantons for a given principal $\SO(3)$-bundle $P \rightarrow M^i_{k,l}$. From Proposition \ref{prop:admissible-bundles-CP2} and Proposition \ref{prop:homogeneous connections SO3-kl}, we are interested in the irreducible $\SO(3)$-bundles $Q_{\widehat{i}}$ and the reducible bundles $P_{\widehat{i}}$ in the exceptional cases when $\widehat{i} = 0 \mod i$. We denote the moduli spaces of instantons by $\mathcal{M}(Q_{\widehat{i}})$ and $\mathcal{M}(P_{\widehat{i}})$, respectively.

From a dynamical systems point of view, these correspond to the collection ofbounded integral curves $(A,B,C, \xi, \psi)$ solving the extended ODE system \eqref{eq:new-evo-spin7} and \eqref{eq:new-generic-SO3-instanton} emanating from $(-3i, 2)$ in the case of $Q_{\widehat{i}
}$  and from $ \left(\frac{\widehat{i}^2}{i}, 0  \right)$ for the exceptional case $P_{\widehat{i}}$. 

In fact, from a purely dynamical point of view, we can consider the \textit{extended} moduli space:
$$\mathcal{M}(P) \coloneqq \bigcup_{k,l \in \R} \mathcal{M} (P)_{k,l}$$
for a given principal $\SO(3)$-bundle $P$, since there is no dynamical reason why $k,l$ must be fixed, or more generally, must be integers. Let us consider the case where $k$ and $l$ vary over the reals. Taking the radial-like parametrisation $(k, l)= (\mu, \mu \,r)$ with $\mu \in \R_+ $ and $r\in [0,1]$, it is easy to verify that the moduli space $\mathcal{M}$ carries a natural $\R_+$-action given by 
$$  \mu \circ :  (s\,; A, B, C, \xi, \psi) \mapsto (\mu s \,; \mu A,\mu B, \mu C,\mu \xi , \psi)\;.$$
The quotient of the moduli space by this action is isomorphic to $ \widetilde{ \mathcal{M}} (P) \coloneqq \bigcup_{r \in [0,1]} \widetilde{\mathcal{M}}_r(P)$, where $\widetilde{\mathcal{M}}_r(P)$ is a choice of representative of the $\R^+$. Of course, there is a canonical choice for $\widetilde{\mathcal{M}}_r(P)$ when $r \in [0,1] \cap \mathbb{Q}$; the original moduli space $\mathcal{M}(P)$.

Under this construction, we have a straightforward
dynamical interpretation of the critical parameter $r^*$:
\begin{proposition}
    The critical parameter $r^*$ identified in Proposition \ref{prop: small l/k for positive D^-_k} corresponds to a bifurcation of the fixed point $(\widehat{k}\Lambda^-,0)$ into the pair of critical points $(\widehat{k}\Lambda^-,0)$ and $(-3k-2C^-, \sqrt{D_{\widehat{k}}})$.
\end{proposition}
Next, let us investigate what happens to the $\Spin(7)$-instantons during the ALC $\to $ AC transition,  We have 
\begin{lemma}\label{lemma:unique_curve}
Fix the AC $\Spin(7)$-metric $\gamma^i_\theta$. The curve $\alpha:[a^*, \infty) \to [0, \pi]$ parameterising bounded $\Spin(7)$-instantons on $P_{\widehat{i}}$ is unique.

Similarly, $\Spin(7)$-instanon $\Omega^i(c^*)$ on $Q_{\widehat{i}}$ is unique.
\end{lemma}
\begin{proof}
We prove it for $\gamma_{k,l}^m$ by contradiction. Fix $\check{a}\in [a^*, \infty)$ and assume there exist two angles, $0<\alpha_1<\alpha_2<\frac{\pi}{2}$, such that $\Gamma(\check{a}, \alpha_i) \in \mathcal{C}$ for all time $s\in \R$. Then, by the ordering lemma, Lemma \ref{lemma:ordering}, we have
$$ U(s)= \xi_1(s) - \xi_2(s) \geq 0 \qquad \text{and}  \qquad V(s) = \psi_2(s) - \psi_1(s) \geq 0\;.$$
for all $s \in \R$. Hence, the difference function $U$ satisfies the ODE inequality
$$ U' =  - F U +  BC( \psi_1 +\psi_2) V \geq -FU\;,$$
where $F= \frac{kAB+lAC+mBC}{\Delta}<0$. So $U$ is unbounded as $s$ grows, giving the required contradiction.

The same argument shows that there cannot exist two solutions $\Omega_c(s)$ and $\Omega_d(s)$ that remain in $\mathcal{C}$ for all time. 
\end{proof}
Using this uniqueness result, we now have
\begin{proposition}
Consider the $\Spin(7)$-metrics $\gamma_\theta^i$ with $\theta \in (0, \theta^*]$. Let $\alpha_\theta: [ a_\theta^* , \infty) \to [0, \pi]$ be the ALC boundary curve of $\Spin(7)$-instantons,  and $\alpha_{\theta^*}:[a^* , \infty) \to [0, \pi]$ the AC curve parameterising $\Spin(7)$-instantons. Then, for every $a>a^*$
$$ \lim_{\theta \to \theta^*} \alpha_\theta (a) =  \alpha_{\theta^*}(a)\;.$$
Moreover, the convergence is uniform on compact subsets of $[a^*, \infty)$.
\end{proposition}
\begin{proof}
    Fix some $\check{a}\in [a^*, \infty)$ and consider a sequence $\theta_n \to \theta^*$. Up to taking subsequences, we have 
    $\alpha_{n}(\check{a}) \to \check{\alpha}$ for some $\check{\alpha}$. Smooth dependence on initial conditions implies 
    $$ \Gamma \left (\check{a}, \alpha_n \right)\to \Gamma(\check{a}, \check{\alpha})$$
    on compact subspaces of the AC manifold. Now, since  critical sectors $\mathcal{C}_\theta$ pass to the limit $\mathcal{C}_{\theta^*}$ for all times $s\in \R$, the limiting solution $\Gamma(\check{a}, \overline{\alpha})$ is contained in the AC critical sector, and Lemma \ref{lemma:unique_curve} implies $\alpha_{\theta_*}(\check a)=\check\alpha$.

    The same argument allows $a$ to vary, and continuity of $\alpha_{\theta^*}$ yields uniform convergence over compact subsets. 
\end{proof}
Note that, while this proposition makes sense for any $n \in \mathbb{N}$, corresponding to the principal $\SO(3)$-bundle over the principal part given by $\theta_n: \U(1) \to \SO(3)$, geometrically, this only arises in the  two exceptional families $(i),(ii)$ and three exceptional cases $(iii) - (v)$ from Proposition \ref{prop:admissible-bundles-CP2}.
\subsection{Compactifications and ideal boundary data}
Let us now discuss the compactness properties of the moduli spaces $\mathcal{M}_{\theta}(P)$ for fixed $\theta\in(0,\theta^*]$ and a principal $\SO(3)$-bundle $P$. On each moduli space, the underlying topology is the one induced by smooth convergence on compact subsets in a fixed invariant gauge. From the point of view of the dynamical system, this coincides with the natural Euclidean topology on the space of initial conditions.

When the bundle changes, convergence means smooth convergence away from the singular orbit, together
with convergence to the corresponding bubble after rescaling.
The disjoint unions below specify the strata of the compactification, rather than its topology. 

We give the proofs for $i=m$; the remaining cases follow using
the corresponding ordering and normalisation. As usual, we set
$\alpha^*=\arctan(-m/(2\Delta))$. We begin with the AC case.  
\begin{theorem}
    Let $\gamma^i_{k,l}(\theta^*)$ be the AC $\Spin(7)$-metric on $M^i_{k,l}$. The moduli space of cohomogeneity one $\Spin(7)$-instantons on $Q_{\widehat{i}}$ is compact.

    In the exceptional case, $\widehat{i}=0 \mod i$, the 
    moduli space $\mathcal{M}(P_{\widehat{i}})$ admits the natural compactification
    $$ \overline{\mathcal{M}(P_{\widehat{i}})} = \mathcal{M}(P_{\widehat{i}})\sqcup \mathcal{M}(Q_{\widehat{i}})\;.$$
\end{theorem}
\begin{proof}
For $Q_{\widehat{m}}$, Proposition \ref{prop:connections-AC} and Lemma \ref{lemma:unique_curve} imply that there exists a unique $\Spin(7)$-connection, so the moduli space consists of a single point, $\mathcal{M}(Q_{\widehat{m}})\cong \{c^*\}$, and compactness is clear. 

For $P_{\widehat{m}}$, Proposition \ref{prop:connections-AC} and Lemma \ref{lemma:unique_curve} give 
$$\mathcal{M}( P_{\widehat{m}}) \cong \big\{ \Gamma^m \big(a, \alpha (a) \big) \;\big| \; a \in [a^*, \infty )  \big\}\;.$$
For any sequence $\{a_n\}$, up to taking a subsequence, there are two possibilities. If $a_n \to \overline{a}< \infty$, continuous dependence on the initial data gives convergence to $\Gamma(\overline{a},\alpha(\overline{a}))$, including across the singular orbit.

If $a_n \to \infty$, we have bubbling along the singular orbit. Proposition \ref{prop:bubbled-off-connection AC} gives convergence to \(\Omega(c^*)\) away from the singular orbit, together with the standard normal ASD bubble. Thus, the parametrisation of the moduli space compactifies to $[a^*, \infty]$, with the point at infinity corresponding to $\Omega(c^*)$, which is precisely $\mathcal{M}(P_{\widehat{m}}) \sqcup\mathcal{M}(Q_{\widehat{m}})$.
\end{proof}
In the ALC case, the moduli spaces admit a natural compactification characterised by an $\U(2)$-invariant ASD connection on $\R^4$, identified as the normal directions to the singular orbit.
\begin{theorem}
    Let $\gamma^i_{r}(\theta)$ be an ALC $\Spin(7)$-metric on $M^i_{k,l}$, and set $\alpha^*= \arctan\left(- \frac{m}{2\Delta}\right)$. The moduli space of cohomogeneity one $\Spin(7)$-instantons $\mathcal{M}_{\theta}(Q_{\widehat{i}})$ admits a natural one-point compactification.
    
    In the exceptional case, $\widehat{i}=0 \mod i$, the 
    moduli space $\mathcal{M}(P_{\widehat{i}})$ admits the natural compactification
    $$ \overline{\mathcal{M}(P_{\widehat{i}})} = \mathcal{M}(P_{\widehat{i}})\sqcup \mathcal{M}(Q_{\widehat{i}}) \sqcup \big(\left[0, \alpha^* \right] \sqcup \{\pi\}\big) \;,$$
    where $\mathbb{I} = [0, \alpha^*] \sqcup \{\pi\}$ corresponds to limiting ideal solutions, given by solutions of the $\U(2)$-invariant ASD equation \eqref{eq:normal-limit} with prescribed asymptotics.
    \end{theorem}
\begin{proof}
    We focus on the case $\widehat{m}=0 \mod m$. The other case follows by time reversal. From Theorem \ref{thm:ALC-instantons-existence}, we have 
    \begin{align*}
        \mathcal{M}(P_{\widehat{m}}) &\cong \big\{\Gamma(a, \alpha) \;\big|\;  a\in [0,\infty )\,,\; 0 \leq\alpha\leq \alpha_\theta(a)   \; \text{or } \;\alpha =\pi\big\} \subseteq \R^+ \times [0, \pi]\;,\\
        \mathcal{M}(Q_{\widehat{m}})& \cong \big\{\Omega(c)  \; \big|\;  c\in [c_\theta^*,\infty) \big\}\;.
    \end{align*}
    Let $s_0$ be some finite time, and consider the rescaling  function
    \begin{align}
        \nu:  \R & \to (-1,1) \\
        x &\mapsto  \frac{x}{\sqrt{1+x^2}}\;.
        \end{align}
    We can embed the moduli space $\mathcal{M}(P_{\widehat{m}})$ in the compact space $D_1^+ \times [-1,1]$, via the map
    \begin{align*}
        \mathbb{E}_{s_0}: \mathcal{M}(P_{\widehat{m}}) & \to D_1^+ \times [-1,1] \\
        (a, \alpha) & \mapsto \big(\nu(a) \cos \alpha,\; \nu (a) \sin \alpha , \;\nu\left( \xi_{a,\alpha}(s_0)\right)\big)
    \end{align*}
    where $D_1^+$ is the closed positive half of the unit disk.

    Under this embedding, the boundary of the image of $\mathcal{M}(P_{\widehat{m}})$ has four components:
    \begin{enumerate}[label=(\roman*)]
        \item \textbf{Bounded initial data} with $\nu(a)<1$. By continuity with respect to initial conditions, this  corresponds to the limiting curve $\alpha_\theta$.
        \item \textbf{Unbounded initial data and bounded evaluation} with $\nu(a)=1$ and $\abs{\nu(\xi)}<1$. By Proposition \ref{prop:bubbling-ALC} (and its construction), it follows that these limiting points correspond precisely to $\{\Omega(c)\st c\in [c^*, \infty)\}\cong \mathcal{M}(Q_{\widehat{m}})$. 
        \item \textbf{Unbounded initial data and unbounded evaluation} with $\nu(a)=1$, $\nu(\xi)=1$ and $\alpha\in [0,\pi]$. After taking a subsequence, we can identify this boundary point with an integral curve of the limiting system \eqref{eqn:compactifying-subsystem} with $\beta=\frac{\pi}{2}$. In this setup, the limiting curves $[[\alpha]]$ with $\alpha \in [0, \alpha^*)$ correspond to an integral curve emanating from $\left(\frac{\widehat{m}^2}{m},0\right)$ and compactify curves $\Gamma(a_n, \alpha_n)$ with $a_n \to+\infty $  and $\alpha_n \to \alpha < \alpha^*$. The limiting curve $[[\alpha^*]]$ corresponds to the integral curve emanating from $(-3m,2)$ (rather than the bubble $\mathbb{B}$) and compactifies the family $\Omega(c_n)$ with $c_n \to +\infty$. 
        
        \item \textbf{Exceptional ray} with $\nu(a)=1$, $\nu(\xi)=-1$ and $\alpha=\pi$. This gives the other limiting abelian connections. 
    \end{enumerate}
    This proves the desired compactification:
    $$ \overline{\mathcal{M}(P_{\widehat{i}})} \cong  \mathcal{M}(P_{\widehat{i}})\sqcup \mathcal{M}(Q_{\widehat{i}}) \sqcup \big(\left[0, \alpha^* \right] \sqcup \{\pi\}\big) \;.$$
    The fact that ideal curves $[[\alpha]]$ correspond to solutions to the $\U(2)$-invariant ASD equations with prescribed asymptotics follows from Proposition \ref{prop:limmit-ASD}. 
    
    Finally, although the embedding $\mathbb{E}_{s_0}$ depends on the choice of regular time $s_0$, the boundary components are canonically identified for different choices of $s_0$, establishing naturality.
\end{proof}
Let us take a closer look at the solutions of the $\U(2)$-invariant equations. Up to a reparametrisation of the scale parameter $\lambda$ we are interested in solutions with asymptotics 
$$\Xi^\beta(t) =  2 \lambda t^2 + o(t^2) \qquad \qquad \Psi^\beta(t) = 2\lambda \beta t^2 + o(t^2)$$ for $\lambda \in (0,\infty)$ and $\beta\in [0,1)$, corresponding to the compactification of the $\Gamma$-family, and with asymptotics
$$ \Xi(t) =2 +  \frac{\lambda}{6}t^4 + o(t^4) \qquad \qquad \Psi(t) = 2 - \frac{\lambda}{12}t^4 + o(t^4) \;,$$  corresponding to the compactification of the $\Omega$-family.

By rescaling time, we may fix $\lambda =\frac{1}{2}$. For $\beta=0$, the solution is simply $(2 \lambda t^2,0)$. For $\beta\in (0,1)$, Equations \eqref{eq:normal-limit} have the explicit solutions 
\begin{equation}
    \Xi^\beta  = q \; t^2 \coth  \left(  \frac{q}{2}t^2 +  z\right) \qquad \qquad  \Psi^\beta =  q \; t^2 \csch \left(  \frac{q}{2}t^2 +  z\right) \;,
\end{equation}
where $q =  \sqrt{1- \beta^2}$ and $z=\arctanh(q)$.

These solutions were originally considered in the physics literature in \cite[Eq. (5.12)]{HIM99}. In the limit $\beta \to 0$, we recover the abelian $(t^2, 0)$ solution. In the limit $\beta \to 1$, for fixed $t$, we have 
$$z+ \frac{q}{2}t^2 =  \arctanh (q) + \frac{q}{2}t^2 = q \left( 1 + \frac{t^2}{2}\right) +o(q)\;,$$
and since $\coth(x)$ and $\csch(x)$ both have a simple pole with residue one at $x=0$, we have 
$$ \lim_{\beta \to 1}\, \Xi^\beta(t)  = \frac{2t^2 }{2 + t^2} + o(t^2) \qquad  \qquad  \lim_{\beta \to 1} \,\Psi^\beta(t)  = \frac{2t^2 }{2 + t^2} + o(t^2) \;,$$
which recovers the charge one ASD connection on $\R^4$.

If instead we take $\beta \to 1$  while rescaling time as $\tau^2 = q \, t^2$, we have the limit
$$\Xi(\tau)= \tau^2 \coth \left( \frac{\tau^2}{2}\right)= 2 + \frac{\tau^4}{6} + o(\tau^4) \qquad \qquad \Psi(\tau)= \tau^2 \csch \left( \frac{\tau^2}{2}\right)=2 -  \frac{\tau^4}{12} + o(\tau^4)\;,$$
which are precisely the asymptotics of the $\Omega$-family we derived in Equation \eqref{eq:nahm-asymptotics-omega}. Hence, we see how the solutions to \eqref{eq:normal-limit} ``split" into two, corresponding to the breaking of the integral curve, as expected.

\section{The exceptional case \texorpdfstring{$M^0_{1,0}$}{M010}}
Finally, let us consider the exceptional case $M^0_{1,0}$, where the exceptional orbit is diffeomorphic to $S^5$. In this case, the coordinates $(A,B,C)$ do not compactify the dynamical system\footnote{In \cite{Lehmann22}, Lehmann introduces new coordinates that compactify the $\Spin(7)$-metric equations. Unfortunately, they do not compactify the $\Spin(7)$-instanton equations.}, so we are forced to study a singular initial value problem. 

In his original work, Reidegeld \cite[Thm 6.1]{Reidegeld_11} showed that there exists a one-parameter family of local $\Spin(7)$-metrics, $\gamma^0_\mu(t)$, characterised by the squashing of the $S^1$-fibre of the Hopf-fibration $S^5 \to \C P^2$ of the singular orbit.

Near the singular orbit, the $\Spin(7)$-metrics have the power series expansion
\begin{subequations}\label{eq:asymptotics-metric-S5}
\begin{align}
a(t)&=1+\frac{\mu}{6}t
+\left(1-\frac{5\mu^2}{72}\right)t^2
+\left(-\frac{7\mu}{90}+\frac{167\mu^3}{6480}\right)t^3+ o(t^3)\\[2mm]
b(t)&=2t+\left(-\frac43+\frac{\mu^2}{27}\right)t^3+o(t^3),\\[2mm]
c(t)&=1-\frac{\mu}{6}t
+\left(1-\frac{5\mu^2}{72}\right)t^2
+\left(\frac{7\mu}{90}-\frac{167\mu^3}{6480}\right)t^3+o(t^3)\\[2mm]
f(t)&=-\mu-\frac{\mu^3}{6}t^2 + o(t^3)
\end{align}
\end{subequations}
In particular, a short computation akin to Lemma \ref{lemma:normal_metric} shows that the metric in the normal directions to the singular orbit converges to the flat metric on $\R^3$. 

By the work of Lehmann \cite{Lehmann22}, we have
\begin{theorem}
    There exists $\mu^*\in (0, \infty)$ such that, for $\mu\in (0,\mu^*)$, the local metric $\gamma^0_\mu$ is (forward) complete and asymptotically locally conical. The limiting metric $\gamma^0_{\mu^*}$ is complete and asymptotically conical. 
\end{theorem}
\begin{remark}
    We have adopted the opposite sign conventions to Lehmann to remain coherent with our previous notation.
\end{remark} 
Let us study $\Spin(7)$-instantons for the metrics $\gamma^0_\mu$. By Proposition \ref{prop:admissible-bundles-S5}, we only have two smooth bundles, the reducible (flat) bundle $P_0$, and the irreducible $\SO(3)$-bundle $Q_2$. 

Over $P_0 \rightarrow M^0_{1,0}$, we have 
\begin{proposition}
    Consider $\gamma^0_{\mu}$ with $\mu \in (0, \mu^*)$. Then there is a one-parameter family of abelian instantons over $P_0$, given by 
$$ \xi(t) =  \kappa f^2(t)  =  \kappa\mu^2 + \frac{\kappa \mu^4}{3} t^2 +o(t^3)$$
with $\kappa \in \R$. For the AC metric $\gamma^0_{\mu^*}$, the only bounded $\Spin(7)$-instanton is the flat connection, i.e. $\xi(t)=0$.
\end{proposition}
\begin{proof}
Recall that we can rewrite the abelian $\Spin(7)$-instanton ODE as 
$$\partial_t \left( \frac{\xi}{f^2}\right) =  0$$
since, in this case, $n=0$. Integrating, we get the desired claim.
\end{proof}
\begin{remark}
    Note that, in this case, the connection coefficient $\xi$ converges to an arbitrary constant $\kappa \mu^2 \in \R$ over the singular $S^5$. This matches precisely with Proposition \ref{prop:connections-S5}, as expected.
\end{remark}
The $\Spin(7)$-instanton equations over $Q_2\to M^0_{1,0}$ are
\begin{subequations}\label{eq:SO(3)inst-M010}
\begin{align}
    \frac{\dot{\xi}}{f} &= \frac{\xi-2}{a^2} + \frac{4- \psi^2}{b^2} - \frac{\xi+2}{c^2}\\
    \dot \psi  & =- \left( \frac{2b}{ac}+ \frac{\xi}{f}\right) \psi
\end{align}
\end{subequations}

Letting $\xi(t) = t u (t)$ and  $\psi= 2+t^2 v (t)$, we get the induced ODE system 
\begin{subequations}
    \begin{align}
        \frac{du }{dt} & = -\frac{1}{t}\left(u - \mu v  -4 \mu\right) +R_1(t, u,v)\\
        \frac{d v}{dt}&= -\frac{1}{t}\left( 2v- \frac{2u }{\mu}  + 8\right) + R_2(t,u,v)
    \end{align}
\end{subequations}
with $R_i$ smooth in $(t, u,v)$ in a neighbourhood of the singular orbit. We use Malgrange’s \cite{Mal74} existence and uniqueness result for singular IVP, following the formulation of Foscolo--Haskins in \cite[Thm. 4.7]{FH17}. In their language, we have 
$$M_{-1} = \begin{pmatrix}
    \mu v + 4\mu- u \\  \frac{2u }{\mu}-2v - 8
\end{pmatrix}\;,$$
which vanishes for $(u_0,v_0)=(\lambda,\frac{\lambda}{\mu}-4)$, and we have
$$ \det \left( h \operatorname{Id} - d_{(u_0,v_0)} M_{-1} \right) = \det \begin{pmatrix}
     h+ 1 & -\mu \\ -\dfrac{2}{\mu} & h+2 
\end{pmatrix} = (h+1)(h+2)-2 = h(h+3)\;.$$
So, we have
\begin{proposition}
    There exist a one-parameter family of local solutions $\Omega(\lambda)(t)$ to the ODE system \eqref{eq:SO(3)inst-M010} with the local formal expansion
    \begin{subequations}\label{eq:formal-expansion}
\begin{align}
        \xi(t) & = \lambda t+
    \left[
        \frac{\lambda^2}{5\mu}
        +\frac{77\mu^2-288}{270}\lambda
        +\frac{4\mu^3}{15}
        \right]t^3+o(t^4),\\[2mm]
        \psi(t)
        &=2+\left(\frac{\lambda}{\mu}-4\right)t^2+\left[
        \frac{7\lambda^2}{20\mu^2}
        -\frac{342-8\mu^2}{135\mu}\lambda
        +\frac{1260-32\mu^2}{135}
        \right]t^4+o(t^4).
\end{align}        
    \end{subequations}
\end{proposition}
Similarly to the generic case, we have 
\begin{lemma}
The partial ordering 
$$ (\xi_1, \psi_1) \succeq (\xi_2, \psi_2) \quad \iff\quad   \xi_1 \geq \xi_2  \;,\;\;  \psi_1\geq \psi_2$$
is preserved by the ODE \eqref{eq:SO(3)inst-M010}. In particular, we have $\Omega(\lambda_1) \succeq  \Omega(\lambda_2)$ for all $\lambda_1 \geq \lambda_2$.
\end{lemma}
\begin{proof}
Once more, we have suitable barrier-like identities. Along $\xi_1=\xi_2$, we have
$$\left( \dot \xi_1 -  \dot \xi_2\right) =  - \frac{f}{b^2}\left( \psi_1^2- \psi_2^2\right) \geq 0\;, $$
since $\psi_1 \geq \psi_2 \geq 0$ and $f<0$. Similarly, along $\psi_1=\psi_2$, we have
$$ \left( \dot \psi_1 -  \dot \psi_2\right) = - \frac{\psi_1}{f}\left( \xi_1- \xi_2\right) \geq 0\;.$$
The claim for the local solutions $\Omega(\lambda)$ now follows from the solution's formal expansion, Equation \eqref{eq:formal-expansion}.
\end{proof}
Hence, we conjecture
\begin{conjecture}
    Let $\gamma^0_\mu(t)$ be a complete $\Spin(7)$-metric, with $\mu \in (0, \mu^*]$. For the ALC metrics, $\mu <\mu^*$, there exists a one-parameter family of $\Spin(7)$-instantons, indexed by $\lambda \in(-\infty,  \lambda^*_\mu]$. For the AC metric $\gamma^0_{\mu^*}$, there exists a unique $\Spin(7)$-instanton, corresponding to $\lambda^*$ and satisfying $$\lim_{\mu \to \mu^*} \lambda^*_\mu = \lambda^* \;.$$
\end{conjecture}
Even without long-time existence, we may still study the local solutions near the singular orbit as \(\lambda\to-\infty\). Let \(r_\delta\) denote the normal rescaling map, analogous to the one in Equation~\eqref{eq:normal-fibre-rescaling}, with
$$ \delta^2_\lambda = -\frac{\mu\kappa^2}{3\lambda}
\qquad \text{for } \lambda\ll 0. $$
We define the functions
$$
\Xi_\delta = -\frac{\delta}{\mu}\, r^*_\delta(\xi),
\qquad\qquad
\Psi_\delta = \frac{1}{2}\, r^*_\delta \psi.
$$
Then, for the rescaled time \(\tau = t/\delta\), these functions satisfy the ODE system
\begin{subequations}\label{eq: prenormalised-BPS}
    \begin{align}
    \frac{d}{d \tau}   \Xi_\delta &= \delta f \left(\frac{1}{a^2} - \frac{1}{c^2}\right) \Xi_\delta + 2 \delta^2\frac{f}{\mu} \left(\frac{1}{a^2} + \frac{1}{c^2}\right)  - \frac{4 \delta^2}{b^2} \frac{f}{\mu} \left( 1- \Psi_\delta^2\right)\\
      \frac{d}{d \tau}  \Psi_\delta &= -\left(\frac{2\delta b}{ac}  - \frac{\mu}{f} \Xi_\delta \right)\Psi_\delta
    \end{align}
\end{subequations}
Taking the limit $\delta \to 0$, since 
$$r^*_\delta a \to 1+o(1)\qquad r^*_\delta b \to 2\delta \tau +o(\delta)\qquad  r^*_\delta c \to 1+o(1)\qquad r^*_\delta f \to  -\mu+o(1)\;,$$
we get the limiting equations
\begin{equation}\label{eq:inst-to-Bogo}
     \frac{d}{d \tau}  \Xi = \frac{1-  \Psi^2}{\tau^2} \qquad \qquad      \frac{d}{d \tau} \Psi = - \Xi\Psi\;, 
\end{equation}
with asymptotics at the origin
\begin{equation}
    \Xi = \frac{\kappa^2}{3} \tau +  o(\tau) \qquad \qquad \Psi = 1 -  \frac{ \kappa^2}{6} \tau^2  + o(\tau^2)\;.
\end{equation}
The system \eqref{eq:inst-to-Bogo} is the well-known radial Bogomolny system \eqref{eq:spherical-bogomolny}, and the asymptotics are precisely those of the explicit BPS solution:
 $$\Xi(\tau ) = \kappa \coth (\kappa \tau) - \frac{1}{\tau} \qquad \qquad \Psi(\tau)  = \frac{\kappa \tau}{\sinh (\kappa \tau)}\;.$$
Hence, we have
\begin{theorem}
    In the limit $\lambda \to -\infty$, the rescaled connection $r^*_{\delta_\lambda}{A_\lambda}$ in the normal directions to the singular orbit converges to the BPS monopole in $\R^3$ of scale $\kappa$.
\end{theorem}
\begin{proof}
Let $A_{\lambda}^{\perp}$ denote the normal component
of the connection. Contracting its tangential component
with the generator of the Hopf fibration, normalised to have unit length at the singular orbit, defines the Higgs field
$$ \Phi_\lambda=\frac{\xi_\lambda(t)}{\mu}T_1 \;.$$
Thus, under the normal dilation $r_\delta(x)=\delta x$,
the rescaled connection and Higgs field  are $r_{\delta_\lambda}^*(A_{\lambda}^{\perp})$ and $
\delta_\lambda r_{\delta_\lambda}^*(\Phi_\lambda)$. 
These are precisely $\Psi_\delta$ and $-\Xi_\delta$.

The factor $\delta_\lambda$ in the Higgs field is required
by the scaling of the Bogomolny equation. Indeed, for the rescaled metric $g_\delta=\delta^{-2} r_{\delta}^*g_\perp$, the Hodge star on one-forms satisfies
\[*_{g_\delta}  =\delta^{-1}r_\delta^*(*_{g_\perp}).\qquad \qedhere\]
\end{proof}
\appendix
\section{Explicit examples}\label{sec:explicit-examples}
We briefly study the gauge theory for the explicit Bryant--Salamon and Calabi metrics.
\subsection{The Bryant-Salamon case}
In the new coordinates, the Bryant--Salamon metric on $M_{1,1}^2$ is given explicitly by
$$A=  \frac{6 e^{6s}}{ 1 +5 e^{6s} }\qquad B=3 \qquad  C=3\;.$$ 
By straightforward computations, one checks that the straight line $\xi-3\psi=0$ is invariant under the $\Spin(7)$-instanton ODE. Therefore, if $\xi \neq  0$, the non-abelian ODE system on the Bryant--Salamon manifold reduces to 
\begin{equation}
    \psi' =\psi \left(6-3\psi - \frac{12 e^{6s}}{1 +5e^{6s}}\right) \end{equation}
This is a Bernoulli-type ODE, with solution 
$$ \psi(u)  = \frac{6 u }{ 1 + 5 u + \kappa (1+5u)^{2/5}},$$
where we introduced the change $u =  e^{6s}$ to lighten notation, and $\kappa \in \R$. The condition that $\psi \geq 0 $ forces $\kappa \geq -1$. It is straightforward to verify that, for $u \to \infty$, $ \psi(u) \to \frac{6}{5}$, which is the non-abelian fixed point. For $u \to 0$, we see that $\psi \to 0$ for all $\kappa >-1$. For $\kappa=-1$, we have
$$\lim_{u \to 0 } \psi = \lim_{u \to 0 } \frac{6 u }{ 1 + 5 u - (1+5u)^{2/5}}  =  \lim_{u \to 0 } \frac{6 u }{ 3u + O(u^2)} = 2\;,$$
giving an explicit example of the bubbling behaviour considered in Section \ref{sec:bubbling}.
\subsection{The Calabi case}
The Calabi metric on $M_{1,1}^1\cong T^*\mathbb{C}P^2$ is given by
$$A=-6 \qquad B  = -3  \left( 1+ \frac{1 }{\sqrt{1+ e^{-12 s}}}\right) \qquad C  = -3  \left( 1- \frac{1 }{\sqrt{1+ e^{-12 s}}} \right)\;,$$ 
where we have taken the convention $p_0^1 = (-6,-6,0)$ once more. Note that we have the relationship $B+C=-6$ for all time and the difference $x= B-C = -\frac{6}{\sqrt{1+e^{-12s}}}$ satisfies the ODE $x'= x \left( 6 -  \frac{x^2}{6}\right)$ . In these new coordinates, the instanton ODE reduces to
\begin{subequations}\label{eq:instantons-Calabi}
    \begin{align}
        \xi'& = -\xi\left(6+\frac{x^2}{6}\right) -18x+3(x-6)\psi^2\\
        \psi' &= \left(x+3-\xi\right) \psi\;.
    \end{align}
\end{subequations}
A straightforward computation yields
 \begin{proposition}
    The sets
    \begin{align*}
        \mathcal{K}&=\{4\xi+6x-(x-6)\psi^2\leq 0\}\cap \{2\xi-x\geq 0\}\cap \{\psi,-x,x+6\geq 0\}\,,\\
        C_1&=\{\psi=0\} \cap\left\{2\xi +3x =0\right\}\cap \{\psi,-x,x+6\geq 0\}\; \text{and}\\
        C_2& =\{x-2\xi=0\}\cap \{8\xi-(\xi-3)\psi^2=0\}\cap \{\psi,-x,x+6\geq 0\}
    \end{align*}
    are compact and forward invariant under the ODE system \eqref{eq:instantons-Calabi}. 

    In particular, the curves $C_1$ and $C_2$ correspond to the traces of the integral curves $\Gamma^1(48,0)$ and $\Omega^1(0)$, respectively. 
\end{proposition}
\begin{proof}
    At $2\xi-x=0$, we have 
    \begin{align*}
    (2\xi-x)'&=-2\xi\left(6+\frac{x^2}{6}\right) -36x+6(x-6)\psi^2-x \left( 6 -  \frac{x^2}{6}\right)\\
    &=-48x+6(x-6)\psi^2\\
    &\geq 12(2\xi - x) \\
    &=0.
    \end{align*}
    
    At $4\xi+6x-(x-6)\psi^2= 0$, we have 
    \begin{align*}
    (4\xi+6x-(x-6)\psi^2)'&=-4\xi\left(6+\frac{x^2}{6}\right) -72x+12(x-6)\psi^2+6x \left( 6 -  \frac{x^2}{6}\right)\\
    &\qquad -x \left( 6 -  \frac{x^2}{6}\right)\psi^2-2(x-6)(x+3-\xi)\psi^2 \\
    &=(2\xi-x)(4\xi+6x)+\left(-12-\frac{x}{6}(6+x)+2(x+3-\xi)\right)\left(4\xi+6x-(x-6)\psi^2\right)\\
    &=(2\xi-x)(x-6)\psi^2\\
    &\leq 0 \qedhere
    \end{align*}
\end{proof}

Since we know there exists a smooth one-parameter family of integral curves $\Gamma^1(a, \alpha(a))$ joining $(-6,-3,0)$ and $(0,0,0)$, there must exist an invariant smooth surface $\mathcal{S}$ contained in $\mathcal{K}$ that parametrises it, and is bounded by the curves $C_1$ and $C_2$. We ask 
\begin{question}
    Does the surface $\mathcal{S}$ admit a closed-form description? In particular, is it algebraic as in the Bryant--Salamon case?
\end{question}

The Calabi metric on $T^*\C P^2$ is actually a hyperkähler metric. This allows one to consider more restrictive types of instantons, known as $\operatorname{Sp}(2)$- and $\SU(4)$-instantons (cf. \cite{AMW25} for definitions). The curves $C_1$ and $C_2$ were already constructed in \cite{AMW25} and shown to be the only smooth bounded $\operatorname{Sp}(2)$-instantons on $M^1_{1,1}$. 

\section{Instantons on \texorpdfstring{$\Spin(7)$}{Spin(7)} metric cones}
Let us now consider $\Spin(7)$-instantons on a metric cone with homogeneous link $N$. $\Spin(7)$-instantons on a $\Spin(7)$ metric cone have a natural interpretation from the point of view of Chern--Simons Floer theory for nearly parallel $G_2$, as outlined in  \cite{DS11}. On a nearly parallel $G_2$-manifold $(N^7, \varphi)$ and principal $G$-bundle $P$, we have the Chern-Simons functional:
\begin{align*}
    CS_\varphi: \mathcal{A}(P) & \to \R \\
    A  & \mapsto  \frac{1}{4}\int_N \operatorname{Tr}\left(F_A^2 - F_{A_0}^2 \right) \wedge \varphi = \int_N CS_0(A, A_0) \wedge \psi\;,
\end{align*} 
where $A_0$ is a fixed background connection and $CS_0(A,A_0) =  \operatorname{Tr}\left[ \alpha \wedge \left(F_A+\frac{1}{2}d_{A_0}\alpha+  \frac{1}{3} \alpha \wedge \alpha \right) \right]$ with $\alpha = A-A_0$ is the usual Chern--Simons 3-form.

The critical points of this Chern--Simons-type functional $CS_\varphi$ are precisely $G_2$-instantons on $(N, \varphi)$, whilst $\Spin(7)$-instantons on the metric cone correspond to gradient flow lines of $CS_\varphi$ joining critical points. The main goal of this section is to establish the following theorem.
\begin{theorem}\label{thm:gradient-flow-CS}
 Denote by $(N_{k,l},g^\pm)$ the Aloff--Wallach space equipped with the nearly parallel $G_2$-structure associated to $p^\pm_{k,l}$, and consider the homogeneous principal $\SO(3)$-bundles $L^n \to N_{k,l}$ characterised by $\theta_n: \U(1) \to \SO(3)$.  

 Then, the Chern--Simons functional $CS\left(L^{\widehat{m}}\right)$ over $(N_{k,l}, g^+)$ has at least two critical points, and a gradient flow line joining them. For $l<r^*k$, the same holds for $CS\left(L^{\widehat{k}}\right)$ over $(N_{k,l}, g^-)$.
\end{theorem}
The two critical points correspond to the abelian and non-abelian fixed points $(p^\pm_{k,l}, \widehat{i}\Lambda^\pm , 0)$ and $\left(p^\pm_{k,l}, \xi_\infty, \sqrt{D_{\widehat{i}}}\right)$ respectively, and have been discussed in the main body of the paper. The existence of the gradient flow line translates in the cohomogeneity one setting to the existence of an integral curve of the ODE
\begin{subequations} \label{eq:ODE-metric cone}
    \begin{align}
    \xi' &= - F^+ \xi + \widehat{m} \widehat{F}^+ - B^+C^+\psi^2\\
    \psi'&= -(2A^++3m  +\xi) \psi
\end{align}
\end{subequations}
which is a subsystem of the ODE system \eqref{eq:new-generic-SO3-instanton} with the fixed $(A,B,C)=p_{k,l}^+ $ coordinates. In order to prove this, we follow the same strategy as in the smooth case. We outline the details for $\widehat{m}$. The case $\widehat{k}$ with $l < r^* k$ follows similarly.

By straightforward computation, we have
\begin{lemma}
Consider the $(\xi,\psi)$-plane division given by 
    \begin{align*}
        \mathfrak{A} & \coloneqq  \left\{ \xi \geq -3m-2A^+ \right\} \cap \left\{ \psi^2  \leq D_{\widehat{m}}^+\right\}\\
        \mathfrak{B} & \coloneqq  \left\{\xi\leq \widehat{m}\Lambda^+\right\}\cup \left(\left\{\xi\leq -3m-2A^+\right\}\cap \{\psi^2\geq D_{\widehat{m}}^+\}\right)\\
        \mathfrak{C} & \coloneqq\left\{ \widehat{m}\Lambda^+\leq \xi\leq -3m-2A^+\right\}\cap \{\psi^2\leq D_{\widehat{m}}^+\}\\
        \mathfrak{F} & \coloneqq  \left\{ \xi \geq -3m-2A^+ \right\}\cap \{\psi^2\geq D_{\widehat{m}}^+\}\,.
    \end{align*}
    The sectors $\mathfrak{A}$ and $\mathfrak{B}$ are forward invariant and $\mathfrak{F}$ is backwards invariant. 
\end{lemma}
Using the same arguments as in the smooth cases, there exists an integral curve $G(s)= (\xi_G, \psi_G)$ emanating from $(\widehat{m} \Lambda^+, 0)$ that remains in $\mathfrak{C}$ for all $s\in \R$, and so it must converge to the non-abelian fixed point $\left(-2A^+ -3m, \sqrt{D_{\widehat{m}}}\right)$.

This proves the existence of the desired integral curve and finishes the proof of Theorem \ref{thm:gradient-flow-CS}.
\section{Invariant ASD connections and monopoles}
We give a brief derivation of the $\U(2)$-invariant equations associated to $\SU(2)$ ASD connections on $\left(\R^4\setminus\{0\}, g_{Eucl}\right)$ and the Bogomolny equations associated to monopoles on $\left(\R^3\setminus\{0\}, g_{Eucl}\right)$ under a spherical symmetry assumption. 
\subsection{\texorpdfstring{$\U(2)$}{U(2)}-invariant ASD equations} \label{sec:invariant-ASD}
We view $\R^4\setminus \{0\}$ as the metric cone over the round 3-sphere $S^3 \cong \SU(2)$. Take $\{dt , t\sigma_i\}$ the standard orthonormal co-frame satisfying $d\sigma_i = - 2\varepsilon_{ijk} \sigma_{jk}$, with ${i,j,k} \in\{1,2,3\}$, and $T_i\in \mathfrak{su}(2)$ generators satisfying $[T_i, T_j] = \varepsilon_{ijk} T_k$. Since $\U(2) \subseteq  \SO(4)$, a $U(2)$-invariant connection can be written in temporal gauge as
$$ A(t) = a(t)\sigma_3 \otimes T_3+ b(t)(\sigma_1\otimes T_1+\sigma_2\otimes T_2) +c(t)(\sigma_1\otimes T_2 - \sigma_2\otimes T_1)\;.$$
The curvature reads
\begin{align*}
    F_A(t) = &\left[ \left(\frac{dc }{dt} \sigma_2 - \frac{db }{dt}\sigma_1  \right) \wedge dt + (a-2)\left(c \sigma_1 +b\sigma_2\right)\wedge \sigma_3\right]\otimes T_1\\&+
\left[- \left(\frac{dc }{dt} \sigma_1 + \frac{db }{dt}\sigma_2  \right) \wedge dt + (a-2)\left(c \sigma_2 - b\sigma_1\right)\wedge \sigma_3\right]\otimes T_2 \\&+ \left( \frac{da }{dt} dt \wedge \sigma_3 +( b^2 +c^2-2a)\sigma_1\wedge\sigma_2\right)\otimes T_3\;,
\end{align*}
and the ASD condition $* F_A =-  F_A$ yields the ODE system:
\begin{equation}\label{eq:pre-spherical-Nahm}
 t  \frac{d a}{dt} = 2a -b^2 -c^2\;,\qquad  t  \frac{d b}{dt} = (2- a)b\;, \qquad t  \frac{d c}{dt} = (2- a)c\;.
\end{equation}
Note that $t \displaystyle\frac{d}{dt}\left(\frac{b}{c}\right)=0$, so the angle $\theta = \arctan \left( \frac{c}{b}\right)$ is constant. Hence, the constant gauge transformation $g = \exp (\theta\, T_3)$ allows us to reduce ourselves to the ODE:
\begin{equation}\label{eq:spherical-Nahm}
 t  \frac{d a}{dt} = 2a -b^2\;,\qquad  t  \frac{d b}{dt} = (2- a)b\;.
\end{equation}
\subsection{Radial Bogomolny equations}
We refer the reader to \cite{Oli14} for a more detailed derivation. As above, we view $\R^3\setminus \{0\}$ as the metric cone of the round $2$-sphere $S^2 \cong \SU(2)/\U(1)$. 

Take $\sigma_i$ left-invariant 1-forms on $\SU(2)$ satisfying $d\sigma_i = - \varepsilon_{ijk} \sigma_{jk}$ with the $\U(1)$-action generated by the dual of $\sigma_1$. Pulling back to $\R^3\setminus\{0\}$ the homogeneous Hopf bundle $$P=SU(2)\times_\chi SU(2)\to SU(2)/U(1)$$ with $\chi: z \mapsto \diag( z, z^{-1})$, we can write the flat metric as 
$$  g_{\R^3\setminus \{0\}} =  dt^2 + t^2 (\sigma_2^2 + \sigma_3^2)\;,$$
where we omit the pullback map for convenience. 

By Wang's theorem \cite{wang1958invariant}, any homogeneous connection $A$ and Higgs field  $\Phi$ in temporal gauge take the form
$$ A=\sigma_1 \otimes T_1 +a(t)(\sigma_2\otimes T_2+\sigma_3\otimes T_3)  \qquad \qquad  \Phi  = - \phi(t) T_1,$$
up to gauge. The curvature and covariant derivative are
\begin{subequations}
    \begin{align}
    F_A &=\frac{da}{dt} dt\wedge\left(\sigma_2\otimes T_2
         +\sigma_3\otimes T_3\right)
         +(a^2-1)\sigma_{23}\otimes T_1,\\
    d_{A}\Phi &=- \frac{d\phi}{dt} dt\otimes T_1
         + \phi\;a(\sigma_2 \otimes T_3-\sigma_3\otimes T_2).
\end{align}
\end{subequations}
The Bogomolny equation $F_{A}=*d_{A}\Phi$ reduces to the radial BPS equation \cite{PS75}:
\begin{equation}\label{eq:spherical-bogomolny}
    \frac{da}{dt}=  - a\phi\, \qquad \qquad 
    \frac{d\phi}{dt}= \frac{1 - a^2}{t^2}\;.
\end{equation}

\bibliographystyle{alpha}
\bibliography{Ref_def}
\bigskip
\noindent\rule{\textwidth}{0.3pt}
\bigskip

{\small
\noindent
\textsc{Xi'an Jiaotong-Liverpool University, Suzhou, China}\\
\textit{E-mail}\,: \href{mailto:hanci.chi@xjlu.edu.cn}{hanci.chi@xjlu.edu.cn}

\medskip
\noindent
\textsc{Beijing Institute of Mathematical Sciences and Applications, Beijing, China}\\
\textit{E-mail}\,: \href{mailto:esolefarre@bimsa.cn}{esolefarre@bimsa.cn}
\par}

\end{document}